\documentclass[10pt]{amsart}
\usepackage[a4paper, total={6in,8in}]{geometry}
\usepackage[utf8]{inputenc}
\usepackage[T1]{fontenc}
\usepackage[english]{babel}
\usepackage{amsmath}
\usepackage{amsfonts}
\usepackage{amssymb}
\usepackage{amsthm}
\usepackage{esint}
\usepackage{mathtools}
\usepackage{bbm}
\usepackage{graphicx}
\usepackage{xcolor}
\usepackage{enumitem}

\makeatletter
\newcommand{\proofpart}[2]{%
	\par
	\addvspace{\medskipamount}%
	\noindent\emph{Part #1: #2}\par\nobreak
	\addvspace{\smallskipamount}%
	\@afterheading
}
\makeatother

\makeatletter
\DeclareFontFamily{OMX}{MnSymbolE}{}
\DeclareSymbolFont{MnLargeSymbols}{OMX}{MnSymbolE}{m}{n}
\SetSymbolFont{MnLargeSymbols}{bold}{OMX}{MnSymbolE}{b}{n}
\DeclareFontShape{OMX}{MnSymbolE}{m}{n}{
    <-6>  MnSymbolE5
   <6-7>  MnSymbolE6
   <7-8>  MnSymbolE7
   <8-9>  MnSymbolE8
   <9-10> MnSymbolE9
  <10-12> MnSymbolE10
  <12->   MnSymbolE12
}{}
\DeclareFontShape{OMX}{MnSymbolE}{b}{n}{
    <-6>  MnSymbolE-Bold5
   <6-7>  MnSymbolE-Bold6
   <7-8>  MnSymbolE-Bold7
   <8-9>  MnSymbolE-Bold8
   <9-10> MnSymbolE-Bold9
  <10-12> MnSymbolE-Bold10
  <12->   MnSymbolE-Bold12
}{}

\let\llangle\@undefined
\let\rrangle\@undefined
\DeclareMathDelimiter{\llangle}{\mathopen}%
                     {MnLargeSymbols}{'164}{MnLargeSymbols}{'164}
\DeclareMathDelimiter{\rrangle}{\mathclose}%
                     {MnLargeSymbols}{'171}{MnLargeSymbols}{'171}
\makeatother

\DeclareMathOperator{\Div}{div}
\DeclareMathOperator{\tr}{tr}

\newtheorem{definition}{Definition}
\newtheorem{lemma}{Lemma}
\newtheorem{theorem}{Theorem}

\newtheorem*{remark}{Remark}

\newcommand{\weakstarto}{\overset{\ast}{\rightharpoonup}}
\newcommand{\weakto}{\rightharpoonup}
\newcommand{\Rxv}{\mathbb{T}^N_x \times \mathbb{R}^N_v}
\newcommand{\vmean}[1]{\left\langle #1 \right\rangle}
\newcommand{\vvmean}[1]{\left\llangle #1 \right\rrangle}

\author{Paulo Sampaio}
\title[Hydrodynamic limits for the BFD equation]{Hydrodynamic limits for the Boltzmann--Fermi--Dirac equation: Acoustic and Stokes--Fourier}
\date{}
\address{CEREMADE, CNRS, Université Paris-Dauphine, PSL Research
University, Place de Lattre de Tassigny, 75016 Paris, France}
\email{paulo.alves-sampaio@dauphine.psl.eu}

\begin{document}
\maketitle

\begin{abstract}
	We establish the acoustic and Stokes--Fourier hydrodynamic limits for the Boltzmann--Fermi--Dirac (BFD) equation on the torus of dimension 2 or more.
	For each limit, we consider fluctuations of appropriately scaled solutions of the BFD equation around a constant background Fermi--Dirac distribution.
	By adapting the weak solution framework developed by Bardos, Golse and Levermore for the classical Boltzmann equation to handle the cubic nonlinearity of the BFD collision operator, we show that these fluctuations converge weakly in time to a unique limit governed by the respective hydrodynamic equation.
\end{abstract}

\tableofcontents

\section{Introduction}

Introduced in the early 20th century, the Boltzmann-Fermi-Dirac (BFD) equation is a modification of the classical Boltzmann equation in which the interacting particles obey Fermi-Dirac statistics.
This equation is written as
\[
	\partial_t f + v \cdot \nabla_x f
	= Q(f).
\]
Here, $f(t,x,v) \geq 0$ is the particle number density.
At any time $t \geq 0$, the quantity $f(t,x,v) dxdv$ represents the number of particles occupying the infinitesimal phase-space volume $dxdv$ centered at $(x,v) \in \Rxv$.

The distinguishing feature of Fermi-Dirac particles is that they are subject to the Pauli exclusion principle, which prevents more than one particle from occupying the same quantum state.
This restriction imposes the upper bound
\[
	f(t,x,v)
	\leq \delta^{-1}
	= \frac{m^N g}{h^N},
\]
where $h$ is Planck's constant, $m$ is the particle's mass and $g$ is the statistical weight of a quantum state (often referred to as the degeneracy factor).
This constraint induces a saturation effect in phase space, where the probability of a collision decreases as the local phase space becomes highly occupied.

To account for these effects, Nordheim \cite{nordheim1928}, and subsequently Uehling and Uhlenbeck \cite{uehling1933}, proposed a modification of the classical Boltzmann collision integral, given by
\begin{equation}
	\label{eq:collision-integral}
	Q(f) = \iint_{\mathbb{R}^N \times \mathbb{S}^{N-1}}
		b(v-v_*, \omega) \big(
			f' f'_* (1 - \delta f) (1 - \delta f_*)
			- f f_* (1 - \delta f') (1 - \delta f'_*)
		\big) \, dv_* d\omega.
\end{equation}

The collision kernel $b(v-v_*, \omega) \geq 0$ is a measurable function, generally taken to be identical to its classical analog in the Boltzmann equation.
We have used the common abbreviations $f$, $f_*$, $f'$ and $f'_*$ to denote the function $f(t,x,\cdot)$ evaluated at the velocities $v$, $v_*$, $v'$ and $v'_*$, respectively.

The unprimed variables $v, v_* \in \mathbb{R}^N$ conventionally denote the pre-collision velocities, while the primed variables $v', v'_*$ denote the post-collision velocities.
These pairs are coupled by the microscopic conservation of linear momentum, $v + v_* = v' + v'_*$, and kinetic energy, $|v|^2 + |v_*|^2 = |v'|^2 + |v'_*|^2$.
The solutions to this algebraic system are parameterized by the unit vector $\omega \in \mathbb{S}^{N-1}$.
Therefore, in the integral \eqref{eq:collision-integral}, the post-collision velocities are defined by the relations
\[
	v' = v - [(v - v_*) \cdot \omega] \,\omega, \quad
	v'_* = v_* + [(v - v_*) \cdot \omega] \,\omega.
\]

A natural question that arises when studying kinetic models is whether a rigorous connection can be established between these microscopic models and the macroscopic fluid equations of hydrodynamics.
The mathematical formalization of this problem traces back to Hilbert's Sixth Problem \cite{hilbert-1902}, which called for the axiomatic treatment of physics and the rigorous derivation of macroscopic mechanics from atomistic views.

The study of what are now called \emph{hydrodynamic limits} consists of identifying the appropriate asymptotic scalings and limiting processes through which the averages of the microscopic quantity $f$ yield macroscopic fluid equations.
In this paper, we aim to prove two hydrodynamic limits for the BFD equation: the acoustic limit and the Stokes-Fourier limit.

\subsection{Presentation of the problem}
\label{subsec:presentation-of-the-problem}

The usual argument in hydrodynamic limits consists of decoupling the different parts of the BFD equation based on their characteristic time scales.
Let $L$ be the macroscopic characteristic length of the domain and $v_{th}$ the typical thermal velocity of the particles.

Macroscopically, the transport term on the left-hand side is governed by an \emph{advection time} $t_{adv} = L / v_{th}$,
representing the time it takes for particles to traverse the characteristic length $L$.
Meanwhile, the overall distribution function $f(t,x,v)$ evolves over an \emph{observation time} $t_{obs}$.
We decouple these two scales by introducing the Strouhal number $\sigma$, defined as the ratio of the advection time to the observation time, $\sigma = t_{adv} / t_{obs}$.

Microscopically, the collision integral $Q(f)$ operates on a much shorter \emph{collision time} $t_{coll} = \lambda / v_{th}$, representing the mean free time between particle collisions, where $\lambda$ is the mean free path.
To decouple this fast microscopic scale from the macroscopic transport dynamics, we introduce the Knudsen number $\kappa$, defined as the ratio of the collision time to the advection time, $\kappa = t_{coll} / t_{adv} = \lambda / L$.

As in the classical Boltzmann equation, a standard non-dimensionalization argument yields the scaled Boltzmann-Fermi-Dirac equation
\begin{equation}
	\label{eq:scaled-bfd}
	\sigma \partial_t f + v \cdot \nabla_x f
 		= \frac{1}{\kappa} Q(f).
\end{equation}

The fundamental assumption underlying any hydrodynamic limit is that collisions occur much faster than advective transport, which corresponds to the limit $\kappa \to 0$.
Multiplying \eqref{eq:scaled-bfd} by $\kappa$ and taking this limit, we formally expect $f$ to approach a local equilibrium state of the collision integral $Q(f)$.

These equilibrium states, known as local Fermi-Dirac distributions, are characterized by the local chemical potential $\varPi(t,x)$ (which implicitly determines the local density), the local bulk velocity $U(t,x)$ and the local temperature $\Theta(t,x)$, taking the general form
\[
	F_{\varPi, U, \Theta}(t,x,v)
	= \left[
		1
		+ \exp\left(
			\frac{|v - U(t,x)|^2 - 2\varPi(t,x)}{2\Theta(t,x)}
		\right)
	\right]^{-1}.
\]

In this work, we are interested in low-Mach limits, where the bulk velocity of the fluid is much smaller than the individual thermal velocities of the particles.
We consider small fluctuations around a global equilibrium state with constant parameters $(\varPi, U, \Theta) = (\varPi_0, U_0, \Theta_0)$.
By an appropriate choice of units and reference frame, we normalize this background state to $(\varPi_0,0,1)$ without loss of generality.

To capture these small fluctuations, we introduce the ansatz $(\varPi, U, \Theta) = (\varPi_0,0,1) + \mu(\varpi, u, \theta)$, where $\mu \to 0$ is the Mach number of the flow.
Taylor expanding the distribution $F_{\varPi, U, \Theta}$ with respect to $\mu$, we obtain
\[
	F_{\varPi, U, \Theta}(v)
	\approx F_{\varPi_0, 0, 1}(v) + \mu  F_{\varPi_0, 0, 1}(v) (1 -  F_{\varPi_0, 0, 1}(v))
	\left[
		\varpi + u \cdot v + \theta
		\left(
			\frac{|v|^2}{2} - \varPi_0
		\right)
	\right].
\]

In order to obtain the local density fluctuation from the chemical potential fluctuation $\varpi$, we recall that the local mass density is given by integrating the distribution function over the velocity space.

Integrating the above Taylor expansion, we can write the total density as $\rho_{total} = \rho_0 + \mu \delta \rho$, where $\rho_0$ is the background density.
Since the odd term $u \cdot v$ integrates to zero by symmetry, $\delta\rho$ depends only on $\varpi$ and $\theta$.
Dividing $\delta\rho$ by the integral of $F_{\varPi_0, 0, 1} (1 -  F_{\varPi_0, 0, 1})$ defines a normalized first-order density fluctuation $\rho$, yielding the relation $\rho = \varpi + C_{\varPi_0} \theta$, where $C_{\varPi_0}$ is a constant depending only on $\varPi_0$.
Substituting this relation back into the Taylor expansion to eliminate the chemical potential fluctuation $\varpi$ yields that there exists a constant $K_{\varPi_0}$ such that
\begin{equation}
	\label{eq:F-varpi-U-theta-approx}
	F_{\varPi, U, \Theta}(v)
	\approx F_{\varPi_0, 0, 1}(v) + \mu  F_{\varPi_0, 0, 1}(v) (1 -  F_{\varPi_0, 0, 1}(v))
	\left[
		\rho + u \cdot v + \frac{1}{2} \theta \left(|v|^2 - K_{\varPi_0}\right)
	\right].
\end{equation}

The saturation effect inherent to Fermi-Dirac statistics dictates that adding new particles to the system becomes increasingly difficult as the phase space fills up.
The chemical potential $\varPi$ is the thermodynamic quantity that measures this difficulty.
In the limit $\varPi \to -\infty$. the system behaves as a classical dilute gas, where one can add as many particles as desired.
In the fully degenerate limit $\varPi \to \infty$, the system becomes completely saturated, making it impossible to add new particles.

In our perturbative framework, the overall physical regime of the fluid is governed by the constant background potential $\varPi_0$.
For finite values of $\varPi_0$, however, the mathematical structure of the problem remains the same.
Therefore, up to an adjustment of constants and renormalization factors, our analysis holds for any choice of background chemical potential $\varPi_0$.
In keeping with standard conventions in the literature (e.g., \cite{jiang-2022}, \cite{jiang-2024} and \cite{jiang-2026}), from now on we set $\varPi_0 = 1$ and denote the background Fermi-Dirac distribution $F_{1,0,1}$ simply as $F$.
Moreover, we denote $K_1$ as $K$, deferring its expression to the next section.

The standard next step in establishing hydrodynamic limits is to demonstrate that the fluctuation variables $(\rho, u, \theta)$ in the expansion \eqref{eq:F-varpi-U-theta-approx} are indeed governed by macroscopic fluid equations.
Formally, this is achieved by substituting the expansion into the scaled equation \eqref{eq:scaled-bfd}, multiplying by a collision invariant $\zeta \in \operatorname{span}\{1, v_i, |v|^2\}$ and integrating over velocity space, before passing to the appropriate limits in the parameters $(\sigma, \kappa, \mu)$.
We refer the reader to \cite{golse-2012} and \cite{golse-2002} for the classical Boltzmann case, which is analogous to the quantum case treated here, and to \cite{zakrevskiy-thesis-2015} and \cite{jiang-2024} for the formal arguments in the Fermi-Dirac case.

\subsection{State of the art}

The mathematical history of these limits traces back to Hilbert himself \cite{hilbert-1912}, who proposed expanding $f$ as a formal power series in the Knudsen number $\kappa$.
This approach, while purely formal, serves as the basis for many early rigorous justifications, such as the derivation of the compressible Euler equations by Caflisch \cite{caflisch-1980}.
For a comprehensive overview of these classical limits, we refer to \cite[Section 1.4]{saint-raymond-2009}.

Compared to its classical counterpart, the hydrodynamic limits of the Boltzmann-Fermi-Dirac equation have been significantly less studied.
Formal derivations of the compressible Euler, as well as both the compressible and incompressible Navier-Stokes equations were first established in the 2015 thesis of Zakrevskiy \cite{zakrevskiy-thesis-2015}, with the compressible Euler result published in \cite{zakrevskiy-euler-2015}.
Building on this formal justification, subsequent efforts have successfully provided rigorous justifications:
the compressible Euler limit was established by Jiang and Zhou in \cite{jiang-2024}, while the incompressible Navier-Stokes limit was proved by Jiang, Xiong and Zhou in \cite{jiang-2022}.
More recently, the regularity requirements for the incompressible Navier-Stokes limit were relaxed in \cite{jiang-2026}.

These results rely on perturbative energy methods in the spirit of Caflisch \cite{caflisch-1980} and Guo \cite{guo-2006}, which are inherently restricted to regimes where the existence of strong solutions to both the kinetic equation and the limiting fluid equation can be guaranteed.
To overcome this limitation in the classical setting, Bardos, Golse and Levermore established a program in the early 1990s (now widely known as the BGL program) to derive weak solutions of fluid models from the DiPerna--Lions renormalized solutions of the Boltzmann equation \cite{bgl-1991, bgl1993}.
By exploiting the entropic structure of the Boltzmann equation alongside weak convergence methods, this framework provided a robust path for justifying low-Mach limits globally in time, relying solely on natural physical \emph{a priori} bounds of the solutions and imposing no restrictions on the size of the initial data.

Over the years, several fluid models were derived within this program, including the acoustic \cite{bgl-acoustic-2000, golse-2002}, the Stokes and incompressible Euler \cite{lions-masmoudi-2001}, and the Stokes-Fourier \cite{golse-2002} systems.
This culminated in the breakthrough compactness argument by Golse and Saint-Raymond \cite{golse-saint-raymond-2004}, which enabled them to recover the incompressible Navier-Stokes system.

Despite these successes, the BGL framework has not yet been extended to the Fermi-Dirac setting and analogous weak-solution limits for the Boltzmann-Fermi-Dirac equation remain absent from the literature.
The present work aims to address this problem.
Supposing only the regularity provided by the physical \emph{a priori} bounds, we leverage the entropic structure of the BFD equation to rigorously recover the acoustic and Stokes-Fourier systems.

\section{Main results}
\label{sec:main-results}

\subsection{Notation}

Throughout the paper, we denote by $F$ the standard Fermi-Dirac distribution $F_{1,0,1}$ and define the weight $\tau = \frac{1}{4} (1 + |v|^2)$.
For a function $\varphi = \varphi(v)$, we denote by $\vmean{\varphi}$ its weighted integral with respect to the velocity variable,
\[
	\vmean{\varphi} = \int_{\mathbb{R}^N} \varphi F(1-F) \, dv.
\]
Let $K = \vmean{|v|^2} / \vmean{1}$.
Observe that this definition implies that the functions $1$, $v_i$ and $\frac{1}{2} (|v|^2 - K)$ are mutually orthogonal in the space $L^2(\mathbb{R}^N; F(1-F) dv)$.

For a function $\varphi = \varphi(v, v_*, \omega)$, we define the double-bracketed mean associated with the collision kernel,
\begin{equation}
\label{eq:vv-mean}
	\vvmean{\varphi} = \iiint_{\mathbb{S}^{N-1} \times \mathbb{R}^{2N}} b(v-v_*, \omega) FF_*(1-F')(1-F'_*) \varphi\, d\omega dvdv_*.
\end{equation}

Unless otherwise specified, the domains of our function spaces are implicitly defined by the indicated measure, where $dt$, $dx$ and $dv$ are understood to denote integration over $(0,\infty)$, $\mathbb{T}^N$ and $\mathbb{R}^N$, respectively.
Thus, for example, the space $L^p(F(1-F) dvdxdt)$ corresponds to the space $L^p((0,\infty) \times \Rxv; dtdx F(1-F) dv)$.
Finally we denote by $\mathcal{D}$ the space of compactly supported $C^\infty$ functions, frequently indicating the domain by its underlying measure (for instance, $\mathcal{D}(dx)$ for test functions on $\mathbb{T}^N$).

\subsection{Analytical setting}
We assume the collision kernel $b = b(v, \omega) \geq 0$ depends only on $|v|$ and $|(v, \omega)|$, and satisfies the bound
\begin{equation}
	\label{eq:collision-kernel-bound}
		b(v, \omega) \leq C_b (1 + |v|).
\end{equation}
This includes hard spheres, as well as hard potentials and Maxwellian molecules under a strong angular cutoff assumption.

We define a weak solution of BFD as follows, formulating it using the scaled version \eqref{eq:scaled-bfd} so that it is already present in the hydrodynamic limit scaling.
This provides a direct reference for our subsequent analysis and is mathematically equivalent to the unscaled formulation.

\begin{definition}
	\label{def:weak-solution}
	Let the initial data $f^{in} \in L^1(\tau dvdx)$.
	We say that $f \in L^\infty(\tau dvdxdt) \cap C(dt; w-L^1(dxdv))$ is a weak solution of the BFD equation with initial data $f^{in}$ if $f$ solves \eqref{eq:scaled-bfd} in the sense of distributions and satisfies $f(t) \weakto f^{in}$ weakly in $L^1(dxdv)$ as $t \to 0$.

	Furthermore, $f$ conserves total mass and linear momentum and satisfies the decay of kinetic energy.
	That is, for almost every $t \geq 0$, we have
	\[
		\begin{split}
			\iint_{\Rxv} \zeta f(t,x,v) \, dxdv
			&= \iint_{\Rxv} \zeta f^{in}(x,v) \, dxdv, \\
			\iint_{\Rxv} |v|^2 f(t,x,v) \, dxdv
			&\leq \iint_{\Rxv} |v|^2 f^{in}(x,v) \, dxdv,
		\end{split}
	\]
	for $\zeta = 1, v_1, v_2, \dots, v_N$.

	Finally, letting $\eta = \sqrt{\kappa \sigma}$, the solution $f$ satisfies the entropy inequality for almost every $t \geq 0$:
	\begin{equation}
		\label{eq:entropy-inequality}
		H(f|F)(t) + \frac{1}{\eta^2} \int_{0}^t R(f) \, d\tau \leq H(f^{in}|F).
	\end{equation}
	Here, $H(f|F)(t)$ is the quantum relative entropy,
	\[
		H(f|F)(t) = \iint_{\Rxv} f \log\left(\frac{f}{F}\right) + (1-f) \log\left(\frac{1-f}{1-F}\right) \, dxdv,
	\]
	and $R(f)$ is the quantum entropy dissipation,
	\begin{equation}
		\label{eq:entropy-dissipation}
		\begin{split}
			R(f) = \frac{1}{4}
			\iiint_{\mathbb{T}^N \times \mathbb{R}^{2N}_{v,v_*} \times \mathbb{S}^{N-1}}
				b(v-v_*, \omega) &(f' f'_{*} (1-f) (1-f_{*}) - f f_{*} (1-f') (1-f'_{*})) \\
				&\quad\times \log\left(
					\frac{f' f'_{*} (1-f) (1-f_{*})}{f f_{*} (1-f') (1-f'_{*})}
				\right)
				\, dxdvdv_*d\omega.
		\end{split}
	\end{equation}
\end{definition}

Weak solutions to the BFD equation were first constructed by Dolbeault \cite{dolbeault-1994} for integrable collision kernels $b \in L^1(\mathbb{R}^N \times \mathbb{S}^{N-1})$ in the full space $x \in \mathbb{R}^N$.
Because all the spatial estimates hold analogously on the torus, this result extends naturally to yield periodic solutions in the sense of Definition \ref{def:weak-solution}.

Since the requirement of an integrable collision kernel is somewhat restrictive, Lions \cite{lions-1994} later relaxed this assumption by proving a compactness result for this equation.
This establishes the existence of solutions for certain locally integrable kernels $b \in L^1_{loc}(\mathbb{R}^N \times \mathbb{S}^{N-1})$, which encompasses the cases in \eqref{eq:collision-kernel-bound}.
Moreover, as noted by Lions in \cite[Remark V.1 iii)]{lions-1994}, this extension remains valid in a periodic box $\mathbb{T}^N$.

In order to rigorously justify the hydrodynamic limits, we consider a sequence of parameters $(\sigma_n, \kappa_n, \mu_n)$, with $\kappa_n, \mu_n \to 0$.
Let $(f^{in}_n)_n$ be a sequence of initial data belonging to $L^1((1+|v|^2) dvdx; \Rxv)$ and define the corresponding sequence of initial fluctuations $(g^{in}_n)_n$ by the relation $f^{in}_n = F + \mu_n F(1-F) g^{in}_n$.

Expanding the relative entropy functional $H(f^{in}_n | F)$, we find that the first-order terms in $\mu_n$ vanish, making this functional of order $\mu_n^2$ in the limit $\mu_n \to 0$ (a more detailed argument is given in Section \ref{subsec:lower-semi-continuity}).
Therefore, in the low-Mach limit, it is natural to assume that our sequence of initial data satisfies the initial entropy bound
\begin{equation}
	\label{eq:initial-relative-entropy}
	H(f^{in}_n|F) \leq C^{in} \mu_n^2.
\end{equation}
Furthermore, to ensure the initial fluctuations do not capture the global mass, linear momentum and kinetic energy, we assume that the initial distributions are matched to the background $F$ for these quantities.
That is, for each collision invariant $\zeta = 1, v_1, \dots, v_N, |v|^2$, we have
\begin{equation}
	\label{eq:fn-invariants-concentrated}
	\iint_{\Rxv} \zeta f^{in}_n \, dxdv
	= \iint_{\Rxv} \zeta F \, dxdv.
\end{equation}

For each $n$, let $f_n$ be a weak solution to the scaled BFD equation with parameters $(\sigma_n, \kappa_n)$ and initial data $f^{in}_n$.
That is, $f_n$ satisfies
\begin{equation}
	\label{eq:scaled-bfd-fn}
	\begin{cases}
		\sigma_n \partial_t f_n + v \cdot \nabla_x f_n
	 		= \frac{1}{\kappa_n} Q(f_n), \\
	 	f_n|_{t = 0} = f^{in}_n.
 	\end{cases}
\end{equation}
We define the corresponding sequence of fluctuations $(g_n)_n$ around the global Fermi-Dirac distribution by the relation
\begin{equation}
	\label{eq:micro-macro-decomposition}
	f_n = F + \mu_n F (1-F) g_n.
\end{equation}
Rigorously justifying the formal approximation \eqref{eq:F-varpi-U-theta-approx} thus amounts to establishing the necessary compactness for $(g_n)_n$ and proving that, as $n \to \infty$, these fluctuations converge in an appropriate sense to the macroscopic hydrodynamic variables $\rho + u \cdot v + \frac{1}{2} \theta (|v|^2 - K)$.

\subsection{Acoustic limit}

Acoustic waves in an ideal fluid are small-amplitude fluctuations around an equilibrium state.
Accordingly, in the fluid dynamics setting, the acoustic system is obtained by linearizing the compressible Euler equations around a constant solution.

To deduce the acoustic limit of the Boltzmann-Fermi-Dirac equation, we start by matching the scale of the observation time to the advection time, setting $\sigma = 1$.
A Hilbert power series argument then traditionally yields the compressible Euler equations in the limit $\kappa \to 0$ even in the Fermi-Dirac case, as shown in \cite[Chapter 1]{zakrevskiy-thesis-2015}.
Further imposing the low-Mach limit $\mu \to 0$, we capture the first-order fluctuations and expect to recover the acoustic system.

A complete formal argument yielding the acoustic limit may be found in \cite{golse-2002} for the classical Boltzmann equation, and an analogous argument holds for the quantum case.
While this formal derivation establishes the acoustic limit without any restrictions on the relative size of the scalings $\mu$ and $\kappa$, our rigorous proof demands the technical scaling \eqref{eq:acoustic-limit-technical-scaling} presented below.
This restriction, also encountered by Bardos, Golse and Levermore in the classical setting, arises from the technique used to establish the vanishing of conservation defects.
Specifically, if $\mu = \kappa^m$, this scaling requires that $m > 1/2$, that is, the Mach number must decrease faster than the square root of the Knudsen number.

\begin{theorem}
	\label{thm:acoustic-limit}
	Let $(\sigma_n, \kappa_n, \mu_n)$, be a sequence of parameters such that $\sigma_n = 1$ and $\kappa_n, \mu_n \to 0$ as $n \to \infty$, satisfying
	\begin{equation}
		\label{eq:acoustic-limit-technical-scaling}
		\mu_n \sqrt{\log\left(\frac{1}{\mu_n}\right)}
		= o \left( \sqrt{\kappa_n} \right).
	\end{equation}

	Let $(f_n)_n$ be a sequence of weak solutions to the scaled BFD equation \eqref{eq:scaled-bfd} with initial data $(f^{in}_n)_n$ satisfying \eqref{eq:initial-relative-entropy}, and let $(g_n)_n$ be the sequence of corresponding fluctuations defined by \eqref{eq:micro-macro-decomposition}.

	Suppose that for some $(\rho_0, u_0, \theta_0) \in L^2(dx; \mathbb{R} \times \mathbb{R}^N \times \mathbb{R})$ the family of initial fluctuations $g^{in}_n$ satisfies, in the sense of distributions,
		\[
			(\rho_0, u_0, \theta_0)
			= \frac{1}{\vmean{1}}
			\lim_{n \to \infty}
			\left(
				\vmean{g^{in}_n},
				\vmean{\frac{N}{K} v g^{in}_n},
				\vmean{\frac{|v|^2 - K}{\gamma K} g^{in}_n}
			\right).
		\]

	Then, as $n \to \infty$, the fluctuations $g_n$ converge weakly in $L^1([0,T] \times \mathbb{T}^N \times \mathbb{R}^N; dtdx F(1-F) dv)$ for all $T>0$ to an infinitesimal Fermi-Dirac distribution $g$, given by
	\[
		g(t,x,v) = \rho(t,x)
			+ u(t,x) \cdot v
			+ \theta(t,x) \frac{1}{2} \left( |v|^2 - K \right).
	\]

	Moreover, $(\rho, u, \theta) \in L^\infty(dt; L^2(dx))$ is the unique weak solution to the acoustic system
	\[
		\begin{cases}
			\partial_t \rho + \frac{K}{N} \Div_x u = 0,\\
			\partial_t u + \nabla_x (\rho + \gamma\theta) = 0,\\
			\partial_t \theta + \frac{2}{N} \Div_x u = 0,
		\end{cases}
	\]
	where $\gamma = \vmean{|v|^4} / (2 \vmean{|v|^2}) - K / 2$, subject to the initial data $(\rho, u, \theta)|_{t = 0} = (\rho_0, u_0, \theta_0)$.
	In addition, this solution satisfies
	\[
		\int_{\mathbb{T}^N} \rho \, dx = 0, \quad
		\int_{\mathbb{T}^N} u \, dx = 0, \quad
		\int_{\mathbb{T}^N} \theta \, dx = 0.
	\]
\end{theorem}

\begin{remark}
	The reader may notice that in the formal derivation of the acoustic system referenced earlier \cite{jiang-2024}, the resulting equations appear slightly different.
	In that work, the system is symmetrized to visually mirror the classical Boltzmann acoustic limit by defining the effective density fluctuation $\sigma \equiv \rho + (\gamma - 1) \theta$.
	This variable substitution yields the classical momentum equation $\partial_t u + \nabla_x (\sigma + \theta) = 0$, but shifts the quantum thermodynamic complexity into the mass equation.

	We have retained the above expression for the acoustic limit so that the macroscopic quantities $(\rho, u, \theta)$ directly represent the physical thermodynamic fluctuations established in \eqref{eq:F-varpi-U-theta-approx}.
	Moreover, by projecting onto the set of functions $\{ 1, v_i, |v|^2 - K \}$, which are orthogonal in $L^2(F(1-F) dv)$, we obtain a convenient algebraic structure for the perturbative theory studied in the next section.
\end{remark}

\subsection{Stokes-Fourier limit}

The Stokes-Fourier limit is obtained by identifying the observation time $t_{obs}$ with a diffusion time $t_{diff}$.
Recalling the characteristic scales defined in Section \ref{subsec:presentation-of-the-problem}, we approximate the gas as a random walk with a step size equal to the mean free path $\lambda$ and heuristically define $t_{diff}$ as the time it takes this process to travel the macroscopic distance $L$.

Given that the mean free time is $t_{coll}$, the number of steps taken over this diffusion time is $t_{diff} / t_{coll} = 1/(\kappa\sigma)$.
However, a standard result of probability theory dictates that the number of steps required to diffuse across a distance $L$ with a step size of $\lambda$ is proportional to $(L / \lambda)^2 = 1 / \kappa^2$.
Equating both expressions for the number of steps yields the scale matching $\sigma = \kappa$.

Furthermore, the Stokes system models a fluid regime where the viscous diffusive forces dominate the inertial convective forces.
The hydrodynamic quantity that determines this regime is the \emph{Reynolds number} $\varrho$, defined as the ratio of inertial to viscous forces within a fluid.
This quantity can be expressed as a function of the Knudsen and Mach numbers through the \emph{von Kármán} relation $\varrho \sim \mu / \kappa$.
Therefore, to obtain the Stokes limit ($\varrho \to 0$), we must formally impose $\mu = o(\kappa)$.

As shown in the formal argument of Zakrevskiy \cite[Chapter III]{zakrevskiy-thesis-2015}, these scalings for $(\sigma, \kappa, \mu)$ are sufficient to recover the Stokes-Fourier limit.
However, we encounter the same difficulty here as in the classical case \cite{golse-2002} in establishing the vanishing of the conservation defects.
This forces a technical scaling introduced in  \eqref{eq:stokes-limit-technical-scaling} below.
This condition is not overly restrictive, since if $\mu = \kappa^m$, both the physical requirement $\mu = o(\kappa)$ and the technical scaling \eqref{eq:stokes-limit-technical-scaling} demand that $m > 1$.

The following theorem establishes the full Stokes-Fourier limit for the BFD equation.
Because the density and temperature fluctuations are coupled in this regime by the Boussinesq relation $\rho + \gamma \theta = 0$, the limiting infinitesimal Fermi-Dirac distribution $g$ can be written as a function of bulk velocity and temperature alone.

\begin{theorem}
	\label{thm:stokes-limit}
	Let $(\sigma_n, \kappa_n, \mu_n)$ be a sequence of parameters such that $\kappa_n, \mu_n \to 0$, $\sigma_n = \kappa_n$ and the scaling
	\begin{equation}
		\label{eq:stokes-limit-technical-scaling}
		\mu_n \sqrt{\log\left(\frac{1}{\mu_n}\right)}
		= o \left( \kappa_n \right).
	\end{equation}

	Consider a sequence of initial data $(f^{in}_n)_n$ belonging to $L^1(\tau dvdx)$ satisfying \eqref{eq:initial-relative-entropy}.
	Let $(f_n)_n$ be a sequence of weak solutions to the scaled BFD equation \eqref{eq:scaled-bfd} and $(g_n)_n$ be the sequence of corresponding fluctuations, defined by \eqref{eq:micro-macro-decomposition}.

	Let $K' = \vmean{|v|^4} / (2 \vmean{|v|^2})$ and $C_1 = (K')^2 \vmean{1} - \vmean{|v|^4} / 4$.
	Suppose the initial fluctuations satisfy
	\[
		(u_0, \theta_0)
		= \lim_{n \to \infty}
		\left(
			\Pi \vmean{\frac{N}{K \vmean{1}} v g^{in}_n},
			\vmean{ \frac{1}{C_1}
				\left(
					\frac{|v|^2}{2} - K'
				\right) g^{in}_n}
		\right)
	\]
	in the sense of distributions, where $\Pi$ is the orthogonal projection from $L^2(dx; \mathbb{R}^N)$ onto divergence-free vector fields.

	Then, as $n \to \infty$, the fluctuations converge weakly in $L^1([0,T] \times \mathbb{T}^N \times \mathbb{R}^N; dtdx F(1-F) dv)$ for all $T > 0$ to an infinitesimal Fermi-Dirac distribution of the form
	\[
		g(t,x,v) = u(t,x) \cdot v
			+ \theta(t,x)
			\left[
				\frac{1}{2} \left( |v|^2 - K \right) - \gamma
			\right].
	\]

	Moreover $(u, \theta)$ is the unique weak solution to the Stokes-Fourier system
	\[
		\begin{cases}
			\partial_t u + \nabla_x p = \nu \Delta_x u,\\
			\partial_t \theta - k \Delta_x \theta = 0.
		\end{cases}
	\]
	coupled with the incompressibility relation $\Div_x u = 0$, corresponding to the initial data $(u, \theta)|_{t = 0} = (u_0, \theta_0)$.
	In addition, this solution satisfies
	\[
		\int_{\mathbb{T}^N} u \, dx = 0, \quad
		\int_{\mathbb{T}^N} \theta \, dx = 0.
	\]
\end{theorem}
\begin{remark}
	The weak formulation of the Stokes momentum equation is understood in the standard incompressible sense.
	That is, weak solutions are defined by testing against smooth, divergence-free vector fields $\varphi \in \mathcal{D}((0,\infty) \times \mathbb{T}^N; \mathbb{R}^N)$.
	In this formulation, the pressure term naturally vanishes, since $\int_0^\infty \int_{\mathbb{T}^N} \nabla_x p \cdot \varphi \, dxdt = 0$.
\end{remark}

The remainder of the paper is organized as follows.
In Section \ref{sec:fluctuation-theory}, we analyze the structure of the entropy inequality \eqref{eq:entropy-inequality} to deduce weak compactness results and bounds for the sequence of fluctuations $(g_n)_n$.
We also introduce the renormalization procedure necessary to control specific quantities within the equation.
Next, in Section \ref{sec:linearized-operator}, we identify the limiting form of every weakly converging sequence $(g_n)_n$, rigorously establishing the approximation \eqref{eq:F-varpi-U-theta-approx}.
Section \ref{sec:vanishing-conservation-defects} addresses the conservation defects arising from this renormalization and proves decay estimates that ensure these defects vanish as $n \to \infty$.
Finally, in Sections \ref{sec:acoustic-limit-proof} and \ref{sec:stokes-limit}, we prove Theorems \ref{thm:acoustic-limit} and \ref{thm:stokes-limit}, respectively, by showing that the macroscopic parameters $\rho$, $u$ and $\theta$ obey their respective hydrodynamic equations.

\section{Consequences of the entropy inequality}
\label{sec:fluctuation-theory}

The primary tool in a BGL approach for establishing hydrodynamic limits is the relative entropy inequality \eqref{eq:entropy-inequality}.
Combined with the initial bound \eqref{eq:initial-relative-entropy}, this inequality provides the essential bounds required to control both the relative entropy and the entropy dissipation.
In this section, we leverage these estimates to establish several compactness results independent of the hydrodynamic scaling.
To do this, recall that if $\Psi$ is a strictly convex function and $\Psi^*$ is its Legendre transform, the \emph{Young-Fenchel} inequality (see, for example, \cite{arnold-1989}) states that
\begin{equation}
	\label{eq:young-fenchel}
	pz \leq \Psi(z) + \Psi^*(p)
	\quad \text{for all } p,z > -1.
\end{equation}

In order to extract uniform bounds from $H(f_n | F)$ and $R(f_n)$ using \eqref{eq:young-fenchel}, we respectively apply the auxiliary functions $h$ and $r$.
Notably, while our setting concerns the quantum Fermi-Dirac statistics, these are the identical functions employed for the classical framework in \cite{bgl1993}.
They are defined for all $z > -1$ as
\[
	h(z) = (1+z) \log(1+z) - z
	\quad\text{and}\quad
	r(z) = z \log(1+z).
\]
We conclude this setup by recalling the main properties of these strictly convex functions and their Legendre transforms $h^*$ and $r^*$, as used in \cite{bgl1993}.

\begin{lemma}
	\label{lem:h-r-properties}
	The following properties hold:
	\begin{enumerate}
		\item \label{item:h-r-absolute-value} For all $z > -1$,
		\[
			h(|z|) \leq h(z)
			\quad \text{and}\quad
			r(|z|) \leq r(z).
		\]
		\item \label{item:h*-r*-quadratic} For all $p \geq 0$ and $\lambda \in [0,1]$,
		\[
			h^*(\lambda p) \leq \lambda^2 h^*(p)
			\quad\text{and}\quad
			r^*(\lambda p) \leq \lambda^2 r^*(p).
		\]
		\item \label{item:h*-r*-exp-estimates} For all $p \geq 0$,
		\[
			h^*(p) = e^p - p - 1 \leq e^p
			\quad\text{and}\quad
			r^*(p) \leq e^p.
		\]
	\end{enumerate}
\end{lemma}

\subsection{Weak compactness of the fluctuations}

As a first consequence of the entropy inequality, we note that since the entropy dissipation \eqref{eq:entropy-dissipation} is nonnegative, the entropy inequality \eqref{eq:entropy-inequality} implies that the bound \eqref{eq:initial-relative-entropy} propagates forward in time.
This yields, for all $t \geq 0$,
\begin{equation}
	\label{eq:fn-relative-entropy}
	H(f_n(t)|F) \leq C^{in} \mu_n^2.
\end{equation}
The following lemma leverages this bound alongside the Young-Fenchel inequality \eqref{eq:young-fenchel} to establish a uniform $L^1$ bound on the fluctuations $(g_n)_n$.

Furthermore, by introducing a parameter $\alpha > 0$, we decompose these fluctuations into an entropy-bounded term that can be made arbitrarily small and a higher-integrability $L^{3/2}$ remainder that controls concentration.
A standard interpolation argument shows that this decomposition yields weak compactness in $L^1$ on bounded time intervals.

\begin{lemma}
	\label{lem:tau-gn-compactness}

	Let $(f_n)_n$ be a sequence of functions in $L^\infty(dt; L^1(\tau F(1-F) dvdx))$ satisfying \eqref{eq:fn-relative-entropy} and let $(g_n)_n$ be the corresponding sequence of fluctuations defined in \eqref{eq:micro-macro-decomposition}.

	The sequence $(\tau g_n)_n$ is uniformly bounded in $L^\infty(dt; L^1(F(1-F) dvdx))$ and weakly compact in $L^1([0,T] \times \Rxv; F(1-F) \, dvdxdt)$.
	Moreover, for each $t \geq 0$, the sequence $\tau g_n(t, \cdot, \cdot)$ is weakly compact in $L^1(F(1-F) dvdx)$.
\end{lemma}
\begin{proof}
	We begin by rewriting the relative entropy in terms of the function $h(z) = (1+z)\log(1+z) - z$.
	Substituting the definition of $f_n$, we obtain
	\begin{equation}
	\label{eq:relative-entropy-rewriting-with-h}
	\begin{split}
		&f_n \log\left(\frac{f_n}{F}\right) + (1-f_n) \log\left(\frac{1-f_n}{1-F}\right)\\
		&\quad = F \big(1 + \mu_n (1-F) g_n \big) \log\big(1 + \mu_n (1-F) g_n\big)
			+ (1-F) \big(1 - \mu_n F g_n \big) \log\big(1 - \mu_n F g_n\big)\\
		&\quad = F h(\mu_n (1-F) g_n) + (1-F) h(-\mu_n F g_n).
	\end{split}
	\end{equation}

	To derive a bound for the fluctuation $g_n$, we set $p = \mu_n \tau / \alpha$ and $z = \mu_n (1-F) |g_n|$ in the Young-Fenchel inequality \eqref{eq:young-fenchel}, to obtain
	\[
	    \frac{\mu^2_n}{\alpha} (1-F) |g_n| \tau
	    \leq h\big(\mu_n (1-F) |g_n|\big) + h^*\left( \frac{ \mu_n \tau}{\alpha} \right).
	\]
	Given an $\alpha > 0$, let $n_0(\alpha)$ be an integer such that $\mu_n \leq \alpha$ for every $n \geq n_0(\alpha)$.
	Applying the scaling inequality from Lemma~\ref{lem:h-r-properties}(\ref{item:h*-r*-quadratic}) and the absolute value inequality from Lemma~\ref{lem:h-r-properties}(\ref{item:h-r-absolute-value}), we find that for any $n \geq n_0(\alpha)$,
	\begin{equation*}
	    (1-F) |g_n| \tau
	    \leq \frac{\alpha}{\mu^2_n} h\big(\mu_n (1-F) g_n\big) + \frac{1}{\alpha} h^*( \tau ).
	\end{equation*}
	Similarly, choosing $p = \mu_n \tau / \alpha $ and $z = \mu_n F |g_n|$, we obtain that
	\begin{equation*}
	    F |g_n| \tau
	    \leq \frac{\alpha}{\mu^2_n} h\big(- \mu_n F g_n\big) + \frac{1}{\alpha} h^*( \tau ).
	\end{equation*}
	Summing the two inequalities above, we deduce that for every $n \geq n_0(\alpha)$,
	\begin{equation}
	    \label{eq:pointwise-bound-g_n-tau}
	    \begin{split}
	    \tau |g_n|
	    &\leq \alpha \; \frac{1}{\mu_n^2} \Big[
	        h\big(\mu_n (1-F) g_n\big) + h\big(- \mu_n F g_n\big)
	    \Big]
	    + \frac{2}{\alpha} h^*(\tau)\\
	    &= \alpha A_n + \frac{1}{\alpha} B.
	    \end{split}
	\end{equation}

	First, we establish the uniform boundedness of $(\tau g_n)_n$ in $L^\infty(dt; L^1(F(1-F) dvdx))$.
	For a fixed $n$, observe that the boundedness of $\tau g_n$ follows directly from that of $\tau f_n$.
	Setting $\alpha = 1$ in \eqref{eq:pointwise-bound-g_n-tau}, multiplying this inequality by $F(1-F)$, then integrating in the phase space $\Rxv$ we deduce that, for every $n$ such that $n \geq n_0(1)$,
	\begin{equation*}
		\int_{\mathbb{T}^N} \vmean{|g_n(t)| \tau} \, dx
		\leq \int_{\mathbb{T}^N} \vmean{A_n} \, dx
		+ \int_{\mathbb{T}^N} \vmean{B} \, dx.
	\end{equation*}

	To estimate the term $A_n$, we use that $h(z) \geq 0$ and that $F, 1 - F \leq 1$.
	Then, applying the identity \eqref{eq:relative-entropy-rewriting-with-h} we obtain that
	\begin{align*}
		\int_{\mathbb{T}^N} \vmean{A_n} \, dx
		&= \frac{1}{\mu_n^2} \iint_{\Rxv} \big[F(1-F) h(\mu_n (1-F) g_n) + F(1-F) h(- \mu_n F g_n) \big] \, dxdv\\
		&\leq \frac{1}{\mu_n^2} \iint_{\Rxv} \big[F h(\mu_n (1-F) g_n) + (1-F) h(- \mu_n F g_n) \big] \, dxdv\\
		&= \frac{1}{\mu_n^2} H(f_n | F).
	\end{align*}
	From the statement, this quantity is bounded by a constant $C^{in}$.

	To estimate the $B$ term, we use the exponential estimate for $h^*$ in Lemma~\ref{lem:h-r-properties}(\ref{item:h*-r*-exp-estimates}) to obtain that
	\begin{equation*}
		\int_{\mathbb{T}^N} \vmean{B} \, dx
		= 2 \iint_{\Rxv} h^*(\tau) F(1-F) \, dxdv
		\leq 2 \iint_{\Rxv} e^{\frac{1}{4} (1+|v|^2)} F(1-F) \, dxdv.
	\end{equation*}
	Since $F(v) \leq e^{1 - |v|^2/2}$, this last integral is finite and this, together with the estimate for $A_n$, shows that the sequence $(\tau g_n)_n$ is uniformly bounded in $L^\infty(dt; L^1(F(1-F) \, dvdx))$.

	Next, we establish the weak compactness of the sequence $(\tau g_n(t))_n$ in $L^1([0,T] \times \Rxv; d\nu)$, where $d\nu = F(1-F) \, dvdxdt$, by showing that the inequality \eqref{eq:pointwise-bound-g_n-tau} provides a decomposition of the form $\alpha L^1 + (1/\alpha) L^{3/2}$ for the term $\tau g_n$.
	Indeed, we have proven above that the sequence $(A_n)_n$ is uniformly bounded in $L^\infty(dt; L^1(F(1-F) \, dvdx))$.
	Observing that
	\[
		\begin{split}
			\int_{\mathbb{T}^N} \vmean{h^*(\tau)^{3/2}} \, dx
			&\leq \iint_{\Rxv} e^{\frac{3}{2} \tau} F(1-F) \, dxdv \\
			&\leq \iint_{\Rxv} e^{3/8 (1+|v|^2)} e^{1 - |v|^2/2} \, dxdv \\
			&= e^{1 + 3/8} \iint_{\Rxv} e^{-|v|^2/8} \, dxdv,
		\end{split}
	\]
	we obtain that $B$ is bounded in $L^\infty(dt; L^{3/2}(F(1-F)\, dvdx))$.

	This $\alpha L^1 + (1/\alpha) L^{3/2}$ decomposition implies the weak compactness of the sequence $(\tau g_n)_n$ in the global space $L^1([0,T] \times \Rxv; d\nu)$.
	While this implication is standard, we provide the details below for the sake of completeness.

	Let $C > 0$ be a constant satisfying the uniform bounds
	\begin{equation}
		\label{eq:An-Bn-bounds}
		\| A_n \|_{L^1([0,T] \times \Rxv; d\nu)}
		\leq C
		\quad\text{and}\quad
		\| B \|_{L^{3/2}([0,T] \times \Rxv; d\nu)}
		\leq C.
	\end{equation}
	From the uniform boundedness of $(\tau g_n)_n$ in $L^\infty(dt; L^1(F(1-F) \, dvdx))$ we can deduce that this sequence is uniformly bounded in $L^1([0,T] \times \Rxv; d\nu)$.
	Since the measure $\nu$ is finite, the Dunford-Pettis theorem ensures that $(\tau g_n)_n$ is weakly compact in $L^1(d\nu)$ provided the sequence is equi-integrable.

	Given $\varepsilon > 0$ we choose $\alpha = \varepsilon / (2C)$ in \eqref{eq:pointwise-bound-g_n-tau} and set $\delta_{n_0} = \left(\varepsilon / (2C) \right)^6$.
	Let $E \subset [0,T] \times \Rxv$ be a measurable set with $\nu(E) < \delta_{n_0}$.
	Integrating the decomposition \eqref{eq:An-Bn-bounds} over $E$ against the measure $d\nu$ and applying H\"{o}lder's inequality to the second term we obtain, for every $n \geq n_0(\alpha)$,
	\begin{align*}
		\int_E |\tau g_n| \, d\nu
		&\leq \alpha \int_E |A_n| \, d\nu
		+ \frac{1}{\alpha} \int_E |B| \, d\nu\\
		&\leq \alpha \| A_n \|_{L^1([0,T] \times \Rxv;\, d\nu)}
		+ \frac{1}{\alpha} \| B \|_{L^{3/2}([0,T] \times \Rxv;\, d\nu)} \nu(E)^{1/3}\\
		&\leq \alpha C
		+ \frac{1}{\alpha} C \delta_{n_0}^{1/3}\\
		&= \frac{\varepsilon}{2} + \frac{\varepsilon}{2} = \varepsilon.
	\end{align*}

	From the absolute continuity of the integral, we have for each $i = 1, 2, \dots, n_0 - 1$ a $\delta_i > 0$ such that if $\nu(E) < \delta_i$, then $\int_E |\tau g_i| \, d\nu \leq \varepsilon$.
	Taking $\delta = \min\{\delta_1, \delta_2, \dots, \delta_{n_0}\}$ we show that the sequence $(\tau g_n)_n$ is equi-integrable, and thus weakly compact in $L^1(d\nu)$.

	To establish weak compactness for fixed times, observe that, for any fixed $t \geq 0$, the decomposition \eqref{eq:pointwise-bound-g_n-tau} takes the form
	\[
		\tau |g_n(t)| \leq \alpha A_n(t) + \frac{1}{\alpha} B(t),
	\]
	Moreover, the uniform estimates \eqref{eq:An-Bn-bounds} imply the spatial bounds
	\[
		\| A_n(t) \|_{L^1(F(1-F) dvdx)}
		\leq C
		\quad\text{and}\quad
		\| B(t) \|_{L^{3/2}(F(1-F) dvdx)}
		\leq C.
	\]
	The same argument as before can be used to show that the sequence is equi-integrable and thus, by the Dunford-Pettis theorem, $(\tau g_n(t))_n$ is weakly compact in $L^1(\Rxv; F(1-F) \, dv dx)$.
\end{proof}

\subsection{Renormalization and compactness of the collision product}
\label{subsec:renorm-comp-collision-product}

Establishing a hydrodynamic limit by this weak-convergence approach requires controlling nonlinear quantities that emerge when expressing the collision integral \eqref{eq:collision-integral} in terms of the fluctuations $(g_n)_n$, particularly when proving the vanishing of conservation defects.
Although Lemma \ref{lem:tau-gn-compactness} guarantees that this sequence of fluctuations is uniformly bounded and weakly compact in weighted $L^1$ spaces, this result does not provide the necessary control.

Moreover, since the fluctuations are related to $f_n$ by
\[
	g_n = \frac{1}{\mu_n} \frac{f_n - F}{F(1-F)},
\]
the Pauli exclusion bounds ($0 \leq f_n \leq 1$), which ensure that the Boltzmann-Fermi-Dirac collision integral is well-defined without renormalization, are lost for $g_n$ due to the $1 / \mu_n$ scaling.

To overcome this lack of a priori control of the fluctuations, we introduce the renormalization factor $N_n$, defined by
\begin{equation}
	\label{eq:3-4-normalization}
	N_n = 1 + \mu_n \left( \frac{3}{4} - F \right) g_n.
\end{equation}
We also introduce the auxiliary functions $P_n$ and $M_n$, defined as
\[
    P_n = \frac{f_n}{F} = 1 + \mu_n (1-F) g_n
    \quad \text{and} \quad
    M_n = \frac{1 - f_n}{1 - F} = 1 - \mu_n F g_n.
\]

The factor $N_n$ provides the necessary non-linear saturation to control the large-fluctuation regime.
By dividing by this factor, we ensure that the resulting renormalized quantities remain uniformly bounded while preserving their behavior in regions of bounded fluctuations.
Moreover, because the denominator $N_n$ is strictly bounded away from zero, this maintains the natural bounds of these quantities.
Rigorously, we have the following estimates.

\begin{lemma}
    \label{lem:Nn-Pn-bounds}
    There exists a constant $M_{\mathrm{ren}} > 0$ such that for every $n \in \mathbb{N}$,
    \[
        \frac{1}{4} \leq N_n \leq \frac{3}{4F}, \qquad
        \frac{P_n}{N_n} \leq \frac{4}{3}, \qquad \text{and} \qquad
        \left| \frac{\mu_n g_n}{N_n} \right| \leq M_{\mathrm{ren}}.
    \]
\end{lemma}
\begin{proof}
	For the first inequality, we start by establishing an algebraic identity for $N_n$ in terms of $P_n$.
	By factoring out the term $(1-F)$ and substituting the identity $\mu_n(1-F)g_n = P_n - 1$, we rewrite $N_n$ as
	\[
		N_n = \frac{3 - 4F}{4(1-F)} P_n + \frac{1}{4(1-F)}
	\]
	Since $0 \leq f_n \leq 1$, it follows that $0 \leq P_n \leq 1/F$.
	Given that $(3 - 4F)/(4(1-F)) \geq 0$, the lower bound for $N_n$ follows from the non-negativity of $P_n$ and the fact that $1 - F \leq 1$.
	The upper bound follows directly from $P_n \leq 1/F$.

	For the second inequality, we observe that $P_n \geq 0$.
	If $P_n = 0$ then the inequality is trivially satisfied, so  we consider points where $P_n > 0$ and we bound the inverse ratio $N_n / P_n$.
	Dividing the expression for $N_n$ by $P_n$ and applying the bound $1 / P_n \geq F$ yields
	\[
		\frac{N_n}{P_n}
		= \frac{3 - 4F}{4(1-F)} + \frac{1}{4(1-F)} \frac{1}{P_n}
		\geq \frac{3 - 4F}{4(1-F)} + \frac{F}{4(1-F)}
	\]
	and this last expression evaluates directly to $3/4$.

	Finally, to prove the last inequality, we expand the definition of $N_n$ to obtain
	\[
		\frac{\mu_n g_n}{N_n}
		= \frac{1}{3/4 - F}
		\left(
			\frac{\mu_n (3/4 - F) g_n}{1 + \mu_n (3/4 - F) g_n}
		\right).
	\]
	Consider the function $E(x) = x/(1+x)$, which is strictly increasing for $x > -1$.
	Since $(1-F) \mu_n g_n \geq -1$ we have that $\mu_n (3/4 - F) g_n \geq - (3/4 - F)/(1-F)$.
	Applying $E(x)$ to both sides of this inequality yields
	\[
		- 4 \left(
			\frac{3}{4} - F
		\right)
		\leq
		\frac{\mu_n (3/4 - F) g_n}{1 + \mu_n (3/4 - F) g_n}.
	\]
	Dividing this expression by $(3/4 - F)$ we obtain that $\mu_n g_n / N_n$ is lower bounded by $-4$.

	On the other hand, using that $E(x) \leq 1$ and that $(3/4 - F)^{-1} \leq 4(e+1)/(3-e)$, we obtain an upper bound on $\mu_n g_n / N_n$.
	Taking the maximum absolute value of these lower and upper bounds yields the constant $M_{\mathrm{ren}}$, concluding the proof.
\end{proof}

We now proceed to deduce a compactness result for the collision product in the integrand of the collision integral \eqref{eq:collision-integral}.
Combining the definitions of $P_n$ and $M_n$ with the balance identity
\begin{equation}
\label{eq:equilibrium-balance-identity}
	F' F'_* (1-F)(1-F_*) = F F_* (1-F')(1-F'_*),
\end{equation}
which equates the forward and reverse collision probabilities at the equilibrium $F$, we deduce the following factorization of the collision product
\begin{equation*}
\begin{split}
	&b(v-v_*, \omega) [f'_n f'_{n,*} (1-f_n) (1-f_{n,*}) - f_n f_{n,*} (1-f'_n) (1 - f'_{n,*})] \\
	&\quad = b(v-v_*, \omega) F F_* (1-F') (1-F'_*)
		\times
		\big[
			P'_n P'_{n,*} M_n M_{n,*} - P_n P_{n,*} M'_n M'_{n,*}
		\big].
\end{split}
\end{equation*}

This factorization suggests splitting the collision integral \eqref{eq:collision-integral} into a finite equilibrium measure $\Lambda$, defined on the space $\mathbb{R}^N_v \times \mathbb{R}^N_{v_*} \times \mathbb{S}^{N-1}_{\omega}$, and a residual collision product.
To ensure compatibility with the natural scaling of the entropy dissipation in \eqref{eq:entropy-inequality}, we define the scaled collision product $q_n$ by normalizing the term in square brackets by $\mu_n \eta_n$, yielding
\begin{align}
	\Lambda(dvdv_*d\omega) &= b(v-v_*, \omega) F F_* (1-F') (1-F'_*) dv dv_* d\omega,
	\label{eq:measure-lambda-def}\\
	q_n &= \frac{P'_n P'_{n,*} M_n M_{n,*} -  P_n P_{n,*} M'_n M'_{n,*}}{\mu_n \eta_n}.
	\label{eq:scaled-collision-product}
\end{align}
Observe that the double bracketed mean \eqref{eq:vv-mean} coincides precisely to the integral against the measure $\Lambda$.

We now investigate the compactness properties of the scaled collision product $q_n$.
Unlike the fluctuations $g_n$, the term $q_n$ does not satisfy the necessary uniform bounds to establish compactness in standard spaces directly.
Instead, we analyze the compactness properties of its renormalized version $q_n / N_n$.

Relying on the boundedness of the entropy dissipation $R(f_n)$ provided by the entropy inequality \eqref{eq:entropy-inequality}, we apply the Young-Fenchel inequality followed by an interpolation argument.
This establishes the weak compactness of the sequence $(q_n / N_n)_n$ using a strategy analogous to the proof of Lemma \ref{lem:tau-gn-compactness}, yielding the following result.

\begin{lemma}
	\label{lem:collision-product-compactness}
	The sequence $( \tau q_n / N_n )_n$ is weakly compact in $L^1(dtdx\Lambda)$.
	Moreover, the sequence $( q_n / N_n N_{n,*} N'_n N'_{n,*} )_n$ is uniformly bounded in $L^2(dtdx\Lambda)$.
\end{lemma}
\begin{proof}
	Using the equilibrium balance identity \eqref{eq:equilibrium-balance-identity} and factoring out the equilibrium weight $F F_* (1-F') (1-F'_*)$, we obtain that
	\begin{equation*}
	\label{eq:rewriting-entropy-dissipation}
	\begin{split}
		\big[ f'_n f'_{n,*} &(1-f_n) (1-f_{n,*}) - f_n f_{n,*} (1-f'_n)(1-f'_{n,*})\big] \log\left(\frac{f'_n f'_{n,*} (1-f_n) (1-f_{n,*})}{f_n f_{n,*} (1-f'_n)(1-f'_{n,*})}\right)\\
		&\quad = F F_* (1-F') (1-F'_*) \big[ P'_n P'_{n,*} M_n M_{n,*} -  P_n P_{n,*} M'_n M'_{n,*}\big] \log\left(\frac{P'_n P'_{n,*} M_n M_{n,*}}{P_n P_{n,*} M'_n M'_{n,*}}\right).
	\end{split}
	\end{equation*}
	We substitute this relation into the entropy dissipation \eqref{eq:entropy-dissipation}.
	Then, applying the definition of the scaled collision product \eqref{eq:scaled-collision-product} and expressing the result in terms of $r$, we have
	\begin{equation*}
	\begin{split}
		R(f_n) &= \frac{1}{4} \int_{\mathbb{T}^N} \vvmean{\big[ P'_n P'_{n,*} M_n M_{n,*} -  P_n P_{n,*} M'_n M'_{n,*}\big] \log\left(\frac{P'_n P'_{n,*} M_n M_{n,*}}{P_n P_{n,*} M'_n M'_{n,*}}\right)} \, dx\\
		&= \frac{1}{4} \int_{\mathbb{T}^N} \vvmean{\frac{\mu_n \eta_n q_n}{P_n P_{n,*} M'_n M'_{n,*}} \log\left(1 + \frac{\mu_n \eta_n q_n}{P_n P_{n,*} M'_n M'_{n,*}}\right) P_n P_{n,*} M'_n M'_{n,*}} \, dx \\
		&= \frac{1}{4} \int_{\mathbb{T}^N} \vvmean{r\left( \frac{\mu_n \eta_n q_n}{P_n P_{n,*} M'_n M'_{n,*}} \right) P_n P_{n,*} M'_n M'_{n,*}} \, dx.
	\end{split}
	\end{equation*}

	Consider the entropy inequality \eqref{eq:entropy-inequality}.
	Since the function $h$ is nonnegative, the decomposition \eqref{eq:relative-entropy-rewriting-with-h} implies that the relative entropy $H(f_n | F)$ is nonnegative.
	Dropping this term from the left-hand side of \eqref{eq:entropy-inequality}, substituting the above expression for $R(f_n)$ and applying the initial relative entropy bound \eqref{eq:initial-relative-entropy} yields
	\begin{equation}
	\label{eq:r-product-of-p-bound}
		\frac{1}{\mu_n^2 \eta_n^2 }\int_0^\infty \int_{\mathbb{T}^N} \vvmean{r\left( \frac{\mu_n \eta_n q_n}{P_n P_{n,*} M'_n M'_{n,*}} \right) P_n P_{n,*} M'_n M'_{n,*}} \, dx dt
		\leq 4 C^{in}.
	\end{equation}

	We aim to bound $( \tau q_n / N_n)_n$ using the Young-Fenchel inequality \eqref{eq:young-fenchel}.
	Fix $\alpha > 0$.
	We rewrite the quantity $\tau |q_n| / N_n$ to isolate the specific argument of $r$ in \eqref{eq:r-product-of-p-bound}, yielding
	\[
		\tau \frac{|q_n|}{N_n}
		= \frac{\alpha}{\mu_n^2 \eta_n^2}
		\bigg(
			\frac{\mu_n \eta_n |q_n|}{P_n P_{n,*} M'_n M'_{n,*}}
		\bigg)
		\bigg( \frac{\mu_n \eta_n \tau}{\alpha} \bigg) \frac{P_n P_{n,*} M'_n M'_{n,*}}{N_n}.
	\]
	Choose $n_0(\alpha) > 0$ such that $\mu_n \eta_n \leq \alpha$ for every $n \geq n_0(\alpha)$.
	Applying the Young-Fenchel inequality \eqref{eq:young-fenchel} to the two terms in parentheses, we obtain
	\begin{equation}
		\label{eq:tau-qn-Nn-young-fenchel}
		\tau \frac{|q_n|}{N_n}
		\leq \frac{\alpha}{\mu_n^2 \eta_n^2} r\left( \frac{\mu_n \eta_n |q_n|}{P_n P_{n,*} M'_n M'_{n,*}} \right) \frac{P_n P_{n,*} M'_n M'_{n,*}}{N_n}
		+ \frac{\alpha}{\mu_n^2 \eta_n^2} r^*\left( \frac{\mu_n \eta_n \tau}{\alpha} \right) \frac{P_n P_{n,*} M'_n M'_{n,*}}{N_n}.
	\end{equation}

	We estimate each of these two terms separately.
	For the first term, we use the bound $N_n \geq 1/4$ from Lemma \ref{lem:Nn-Pn-bounds} and the absolute value bound from Lemma \ref{lem:h-r-properties}~(\ref{item:h-r-absolute-value}) to deduce that
	\begin{equation}
	\label{eq:tau-qn-Nn-young-fenchel-term1}
		\frac{1}{\mu_n^2 \eta_n^2} r\left( \frac{\mu_n \eta_n |q_n|}{P_n P_{n,*} M'_n M'_{n,*}} \right) \frac{P_n P_{n,*} M'_n M'_{n,*}}{N_n}
		\leq 4 \frac{1}{\mu_n^2 \eta_n^2} r\left( \frac{\mu_n \eta_n q_n}{P_n P_{n,*} M'_n M'_{n,*}} \right) P_n P_{n,*} M'_n M'_{n,*},
	\end{equation}
	which is bounded in $L^1(\Lambda dtdx)$ from \eqref{eq:r-product-of-p-bound}.
	Denote by $A_n$ the function on the right-hand side.

	For the second term on the right-hand side of \eqref{eq:tau-qn-Nn-young-fenchel}, we apply the quadratic homogeneous estimate for $r^*$ from Lemma \ref{lem:h-r-properties}~(\ref{item:h*-r*-quadratic}) with $\lambda = (\mu_n \eta_n)/\alpha$.
	Then, applying the bound $P_n/N_n \leq 4/3$ from Lemma \ref{lem:Nn-Pn-bounds} and noting that
	\[
		M_n = \frac{1 - f_n}{1 - F} \leq \frac{1}{1 - F} \leq 1 + e,
	\]
	with an identical bound holding for $M'_{n,*}$,	we have that
	\begin{equation}
	\label{eq:tau-qn-Nn-young-fenchel-term2}
	\begin{split}
		\frac{\alpha}{\mu_n^2 \eta_n^2} r^*\left( \frac{\mu_n \eta_n \tau}{\alpha} \right) \frac{P_n P_{n,*} M'_n M'_{n,*}}{N_n}
		&\leq \frac{1}{\alpha} r^*( \tau ) \frac{P_n P_{n,*} M'_n M'_{n,*}}{N_n}\\
		&\leq \frac{(1+e)^2}{3} \frac{1}{\alpha} r^*(\tau) P_{n,*}\\
		&= \frac{(1+e)^2}{3} \frac{1}{\alpha} r^*(\tau) + \frac{(1+e)^2}{3} \frac{1}{\alpha} r^*(\tau) \mu_n (1-F_*) g_{n,*}\\
		&\coloneq \frac{1}{\alpha} B + \frac{\mu_n}{\alpha} C_n.
	\end{split}
	\end{equation}

	Let us bound each of the terms $B$ and $C_n$.
	Observe that from \eqref{eq:collision-kernel-bound} we have, for every $(v,v_*,\omega) \in \mathbb{R}^N_v \times \mathbb{R}^N_{v_*} \times \mathbb{S}^{N-1}$,
	\begin{equation}
		\label{eq:b-v-v*-estimate}
		b(v-v_*, \omega)
		\leq C_b (1 + |v|) (1 + |v_*|).
	\end{equation}
	This bound, together with the exponential bound from Lemma \ref{lem:h-r-properties}~(\ref{item:h*-r*-exp-estimates}) and the fact that $1 - F'$ and $1 - F'_*$ are bounded by $1$, implies that
	\begin{equation*}
		\vvmean{r^*(\tau)^{3/2}}
		\leq 4 \pi C_b e^{3/8}
		\left(
			\int_{\mathbb{R}^N} (1+|v|) e^{\frac{3}{8}|v|^2} F \, dv
		\right)
		\left(
			\int_{\mathbb{R}^N} (1 + |v_*|) F_* \, dv_*
		\right).
	\end{equation*}
	Since $F(v) \leq e^{1 - |v|^2/2}$, the integrals in the right-hand side are finite, which implies the term $B$ is bounded in $L^{3/2}(\Lambda dtdx)$.

	The term $C_n$ is bounded in $L^1(\Lambda dtdx)$ since by using the same bounds as above, we obtain
	\begin{equation*}
		\vvmean{r^*(\tau) (1-F_*) g_{n,*}}
		\leq 4 \pi C_b e^{3/8}
		\left(
			\int_{\mathbb{R}^N} (1+|v|) e^{\frac{3}{8} |v|^2} F \, dv
		\right)
		\left(
			\int_{\mathbb{R}^N} (1 + |v_*|) g_{n,*} F_* (1-F_*) \, dv_*
		\right).
	\end{equation*}
	The first integral is finite, since $F(v) \leq e^{1 - |v|^2/2}$.
	Using that $1 + |v| \leq 1 + 2\tau \leq 6 \tau$ we deduce that there exists a constant $C > 0$ such that, for every $T>0$,
	\[
		\int_0^T \int_{\mathbb{T}^N} \vvmean{r^*(\tau) (1-F_*) g_{n,*}} \, dxdt
		\leq C \int_0^T \int_{\mathbb{T}^N} \int_{\mathbb{R}^N} \tau g_{n} F (1-F) \, dvdxdt,
	\]
	which implies $(C_n)_n$ is uniformly bounded in $L^1(\Lambda dtdx)$, from the result of Lemma \ref{lem:tau-gn-compactness}.

	Combining \eqref{eq:tau-qn-Nn-young-fenchel-term1}, \eqref{eq:tau-qn-Nn-young-fenchel-term2} with the bounds for $(A_n)_n$, $B$ and $(C_n)_n$, we find that for every $\alpha > 0$, there exists an integer $n_0(\alpha)$ such that, for every $n \geq n_0(\alpha)$,
	\begin{equation}
	\label{eq:q_n-N_n-estimate-An-B-Cn}
		\tau \frac{|q_n|}{N_n}
		\leq \alpha A_n + \frac{1}{\alpha} B + \frac{\mu_n}{\alpha} C_n.
	\end{equation}
	Here, the sequences $(A_n)_n$ and $(C_n)_n$ are uniformly bounded in $L^1(\Lambda dtdx)$, while $B$ is bounded in $L^{3/2}(\Lambda dtdx)$.

 	This implies $( \tau q_n / N_n )_n$ is weakly compact in $L^1(\Lambda dtdx)$ by an argument analogous to the proof of Lemma \ref{lem:tau-gn-compactness}, adapted to account for the additional term involving $\mu_n$.
 	We now present this adaptation in detail.

 	First, we establish that the sequence $( \tau q_n / N_n )_n$ is uniformly bounded in $L^1(\Lambda dtdx)$.
 	Recalling the definition of $n_0(\alpha)$, observe that for $\alpha = \max_n (\mu_n \eta_n)$, we may take $n_0 = 1$.
 	Therefore, the inequality \eqref{eq:q_n-N_n-estimate-An-B-Cn} holds for all $n \in \mathbb{N}$.
 	Since $\Lambda dtdx$ is a finite measure on the space $\mathbb{R}^N_v \times \mathbb{R}^N_{v_*} \times \mathbb{S}^{N-1}_{\omega} \times [0,T] \times \mathbb{T}^N_x$, we have the embedding $L^{3/2}(\Lambda dtdx) \hookrightarrow L^1(\Lambda dtdx)$.
 	Thus, the uniform boundedness follows immediately from the bounds on $(A_n)_n$, $B$ and $(C_n)_n$.

 	Next, for any $\alpha > 0$, let $n_1(\alpha) \geq n_0(\alpha)$ be chosen such that $\mu_n \leq \alpha^2$.
 	The estimate \eqref{eq:q_n-N_n-estimate-An-B-Cn} then implies that for all $n \geq n_1(\alpha)$,
	\[
 		\tau \frac{|q_n|}{N_n}
		\leq \alpha (A_n + C_n) + \frac{1}{\alpha} B.
 	\]
 	Since $(A_n + C_n)_n$ is uniformly bounded in $L^1(\Lambda dtdx)$ and $B$ is bounded in $L^{3/2}(\Lambda dtdx)$, the desired weak compactness follows from an argument analogous to the one reproduced at the end of the proof of Lemma \ref{lem:tau-gn-compactness}.

 	For the second statement, we return to the entropy dissipation bound \eqref{eq:r-product-of-p-bound}.
 	Consider the inequality $r(z) \geq z^2 / (1 + \frac{1}{2} z)$, which is valid for every $z > -1$.
	Substituting $z = \mu_n \eta_n q_n / (P_n P_{n,*} M'_n M'_{n,*})$, multiplying both sides by $P_n P_{n,*} M'_n M'_{n,*} / (\mu_n^2 \eta_n^2)$ and recalling the definition of $q_n$ to simplify the resulting denominator, we obtain
	\[
		\frac{1}{\mu_n^2 \eta_n^2}
		r\left(
			\frac{\mu_n \eta_n q_n}{P_n P_{n,*} M'_n M'_{n,*}}
		\right)
		P_n P_{n,*} M'_n M'_{n,*}
		\geq
		2 \frac{q_n^2}{P'_n P'_{n,*} M_n M_{n,*} + P_n P_{n,*} M'_n M'_{n,*}}.
	\]

	Next, we bound the denominator from above.
	Using the bounds for $P_n / N_n$ and $1 / N_n$ from Lemma \ref{lem:Nn-Pn-bounds}, together with the estimate $M_n \leq 1 - F \leq 1 + e$, we obtain
	\[
		\begin{split}
			\frac{P'_n P'_{n,*} M_n M_{n,*} + P_n P_{n,*} M'_n M'_{n,*}}{N_n N_{n,*} N'_n N'_{n,*}}
			&= \frac{P'_n P'_{n,*}}{N'_n N'_{n,*}}
			\frac{M_n M_{n,*}}{N_n N_{n,*}}
			+ \frac{P_n P_{n,*}}{N_n N_{n,*}}
			\frac{M'_n M'_{n,*}}{N'_n N'_{n,*}} \\
			&\leq \frac{512}{9} (1 + e)^2.
		\end{split}
	\]

	Multiplying these inequalities and noting that the lower bound for $N_n$ from Lemma \ref{lem:Nn-Pn-bounds} implies $1 / (N_n N_{n,*} N'_n N'_{n,*}) \leq 256$, we deduce that these exists a constant $C' > 0$ such that
	\[
		\left(
			\frac{q_n}{N_n N_{n,*} N'_n N'_{n,*}}
		\right)^2
		\leq C' \frac{1}{\mu_n^2 \eta_n^2}
			r\left(
				\frac{\mu_n \eta_n q_n}{P_n P_{n,*} M'_n M'_{n,*}}
			\right)
			P_n P_{n,*} M'_n M'_{n,*}.
	\]

	Integrating this inequality with respect to $dtdx\Lambda$ and applying the entropy dissipation bound \eqref{eq:r-product-of-p-bound} to the right-hand side, we conclude that
	\[
		\int_0^T \int_{\mathbb{T}^N}
		\vvmean{
			\left(
				\frac{q_n}{N_n N_{n,*} N'_n N'_{n,*}}
			\right)^2
		} \, dxdt
		\leq 4 C' C^{in}.
	\]
	This establishes the uniform $L^2(dtdx\Lambda)$ bound and completes the proof.
\end{proof}

\subsection{Bounds on renormalized quadratic fluctuations}

As discussed in Section \ref{subsec:renorm-comp-collision-product}, renormalization is introduced to control the nonlinearities that emerge throughout the asymptotic analysis, particularly in the vanishing of conservation defects.
Although Lemma \ref{lem:collision-product-compactness} established bounds for the renormalized scaled collision product, we now tackle the nonlinearities arising from the pointwise product of the fluctuations $g_n$, specifically quadratic terms of the form $g_n^2$.

While the uniform boundedness of the renormalized $g_n^2 / N_n$ is shown to be a direct consequence of the relative entropy bound \eqref{eq:fn-relative-entropy}, uniform boundedness for the weighted sequence $\tau g_n^2 / N_n$ seems out of reach with the current techniques.
However, following the approach of \cite{bgl1993}, we apply the Young-Fenchel inequality to a convex function adapted to the renormalization factor $N_n$.
By estimating the residual factors generated by this inequality, we demonstrate that the norm of this weighted sequence grows at most logarithmically.
Even though this is weaker than a uniform bound, this logarithmic control is sufficient to close the estimates in the remainder of the paper.

\begin{lemma}
	\label{lem:gn2-Nn-bounds}

	In the space $L^\infty(dt; L^1(F(1-F)dvdx))$, we have that
	\[
		\frac{g_n^2}{N_n} = \mathcal{O}(1)
		\quad\text{and}\quad
		\tau \frac{g_n^2}{N_n} = \mathcal{O}\left(\log\left(\frac{1}{\mu_n}\right)\right).
	\]
\end{lemma}
\begin{proof}
	Following the approach established for the classical Boltzmann equation in \cite{bgl1993}, we introduce the function $s(z)$, defined for all $z > -3$ by
	\begin{equation}
		\label{eq:z-function}
		s(z) = \frac{1}{2} \frac{z^2}{1 + \frac{1}{3}z}.
	\end{equation}
	We further define an auxiliary renormalization factor $\mathcal{N}_n$ by
	\[
		\mathcal{N}_n = 1 + \frac{1}{3} (1-F) \mu_n g_n.
	\]

	We start by estimating the ratio $\mathcal{N}_n / N_n$.
	Expanding the definition of $\mathcal{N}_n$ we rewrite
	\begin{equation*}
		\frac{\mathcal{N}_n}{N_n}
		= \frac{1}{N_n}
			+ \frac{1}{3} (1-F)
			\frac{\mu_n g_n}{N_n}
	\end{equation*}
	Applying the triangle inequality, the fact that $|1-F| \leq 1$, and the bounds for $1/N_n$ and $\mu_n g_n/N_n$ established in Lemma \ref{lem:Nn-Pn-bounds}, we conclude that $|\mathcal{N}_n/N_n|$ is uniformly bounded.

	This bound, together with the identity $\mu_n^2 g_n^2 / \mathcal{N}_n = 2 (1-F)^{-2} s( \mu_n (1-F) g_n )$ and the estimate $1/(1-F) \leq 4$ implies that there exists a constant $C_1 > 0$ such that
	\begin{equation}
		\label{eq:mu_n^2-g_n^2-N_n-estimate1}
		\mu_n^2 \frac{g_n^2}{N_n}
		= \mu_n^2 \frac{\mathcal{N}_n}{N_n} \frac{g_n^2}{\mathcal{N}_n}
		\leq C_1 s( \mu_n (1-F) g_n ).
	\end{equation}

	We observe that $s(z) \leq h(z)$ for every $z \geq -1$, where $h$ is the function defined in Lemma \ref{lem:h-r-properties} (cf. the proof of Corollary 3.2 in \cite{bgl1993}).
	Since $s(z)$ is a positive function for $z \geq -1$, we may combine estimate \eqref{eq:mu_n^2-g_n^2-N_n-estimate1} with this upper bound to obtain that
	\begin{equation*}
	\begin{split}
		\frac{1}{C_1} \mu_n^2 \frac{g_n^2}{N_n}
		&\leq
			s( \mu_n (1-F) g_n )
			+ s( -\mu_n F g_n ) \\
		&\leq
			h( \mu_n (1-F) g_n )
			+ h( -\mu_n F g_n ).
	\end{split}
	\end{equation*}
	We now multiply both sides by $F(1-F)$.
	Using the bound $1 - F \leq 1$ on the first term of the right-hand side and $F \leq 1$ on the second one, we deduce that
	\[
		F (1-F) \frac{g_n^2}{N_n}
		\leq \frac{C_1}{\mu_n^2}
		\left[
			F h( \mu_n (1-F) g_n )
			+ (1-F) h( -\mu_n F g_n )
		\right].
	\]
	The first desired estimate then follows from integrating this inequality over $(x,v) \in \mathbb{T}^N \times \mathbb{R}^N$ and applying \eqref{eq:fn-relative-entropy}.

	For the second estimate, we rely on the inequality
	\begin{equation*}
		\frac{h'(w)}{s'(w)} s(z)
		\leq h(z) + \frac{h'(w)}{s'(w)} s(w) - h(w),
	\end{equation*}
	for every $z, w \geq 0$, which follows from the Young-Fenchel inequality applied to the function $f = h \circ s^{-1}$ (cf. proof of Proposition 3.3 in \cite{bgl1993}).
	This inequality, a priori only valid for positive values of $z$, can be extended if we notice that, for every $z \in ]0,1]$,
	\[
		\frac{s(-z)}{s(z)}
		= \frac{3+z}{3-z}.
	\]
	The expression on the right-hand side has a maximum value of $2$ in this region.
	This implies $s(z) \leq 2 s(|z|)$ for every $z > -1$, since this inequality is trivially satisfied for $z \geq 0$.
	Multiplying by $h'(w)/s'(w) \geq 0$ and applying the Young-Fenchel inequality, we estimate $h(|z|)$ on the right-hand side using Lemma~\ref{lem:h-r-properties}(\ref{item:h-r-absolute-value}) to obtain that, for every $z \geq -1$, and $w \geq 0$,
	\begin{equation}
		\label{eq:young-inequality-h-s}
		\frac{h'(w)}{s'(w)} s(z)
		\leq 2 h(z)
		+ 2 \left(
			\frac{h'(w)}{s'(w)} s(w) - h(w)
		\right).
	\end{equation}

	We study the term $h' / s'$ on the left-hand side.
	Computing this directly from the definitions of $h$ and $s$, we find that for all $w > 0$,
	\begin{equation}
		\label{eq:h'-s'-formula}
		\frac{h'(w)}{s'(w)}
		= \frac{2 (w+3)^2 \log(1+w)}{3 w (w+6)}.
	\end{equation}
	Because this function is strictly increasing for $w > 0$ and its range is the interval $(1, +\infty)$, for each $n \in \mathbb{N}$ and $v \in \mathbb{R}^N$, there exists a unique solution $w = w_n(v) > 0$ to the equation
	\begin{equation}
		\label{eq:definition-wn}
		\frac{h'(w)}{s'(w)} = 1 + \log\left(
			1 + \mu_n \exp\left(\frac{1}{8} |v|^2\right)
		\right).
	\end{equation}

	By substituting $z = \mu_n (1-F) g_n$ and $w = w_n(v)$ in \eqref{eq:young-inequality-h-s}, then applying inequality \eqref{eq:mu_n^2-g_n^2-N_n-estimate1} to the left-hand side, we find that there exists a constant $C_2 > 0$ such that
	\begin{equation*}
		C_2 \frac{h'(w_n(v))}{s'(w_n(v))} \mu_n^2 \frac{g_n^2}{N_n}
		\leq h(\mu_n (1-F) g_n)
			+ \left(
				\frac{h'(w_n(v))}{s'(w_n(v))} s(w_n(v))
				- h(w_n(v))
			\right).
	\end{equation*}

	Since $h$ is a nonnegative function, we bound the right-hand side above by adding a term $h(- \mu_n F g_n)$.
	Multiplying both sides by $F(1-F)$ and applying the bounds $1 - F \leq 1$ and $F \leq 1$ to the terms involving $h(\mu_n (1-F) g_n)$ and $h(- \mu_n F g_n)$, respectively, yields
	\begin{equation}
	\label{eq:mu_n^2-g_n^2-N_n-estimate2}
	\begin{split}
		C_2 \frac{h'(w_n(v))}{s'(w_n(v))} \mu_n^2 F(1-F) \frac{g_n^2}{N_n}
		&\leq F h(\mu_n (1-F) g_n)
			+ (1-F) h(- \mu_n F g_n)\\
			&\quad + F(1-F) \left(
				\frac{h'(w_n(v))}{s'(w_n(v))} s(w_n(v))
				- h(w_n(v))
			\right).
	\end{split}
	\end{equation}

	We estimate the left-hand side of \eqref{eq:mu_n^2-g_n^2-N_n-estimate2} from below.
	Regarding the choice of $w_n$ made previously in \eqref{eq:young-inequality-h-s}, it was shown in the proof of Proposition 3.3 of \cite{bgl1993} that
	\begin{equation}
		\label{eq:zeta_n-h'-s'-from-bgl}
		\zeta_n \frac{1}{8} |v|^2 \leq \frac{h'(w_n(v))}{s'(w_n(v))},
		\quad \text{for every } v \in \mathbb{R}^N,
	\end{equation}
	where $\zeta = \zeta_n$ is the solution to the algebraic equation
	\[
		1
		- \zeta \log(\zeta)
		- (1 - \zeta) \log(1 - \zeta)
		+ \zeta \log \mu_n = 0.
	\]
	This equation can be interpreted as the intersection of the graph of $f(\zeta) = \zeta \log(\zeta) + (1 - \zeta) \log(1 - \zeta)$ with the graph of $g(\zeta) = 1 +\zeta \log \mu_n$, as illustrated in Fig. \ref{fig:triangle-graph}.
	In this representation, we have that $\tan\theta = [\log(1/\mu_n)]^{-1}$.
	As $\mu_n\to 0$, the angle $\theta$ tends to $0$, and consequently $\zeta_n \to 0$.

	\begin{figure}[htp]
	\begin{center}
		\includegraphics[width=0.5\textwidth]{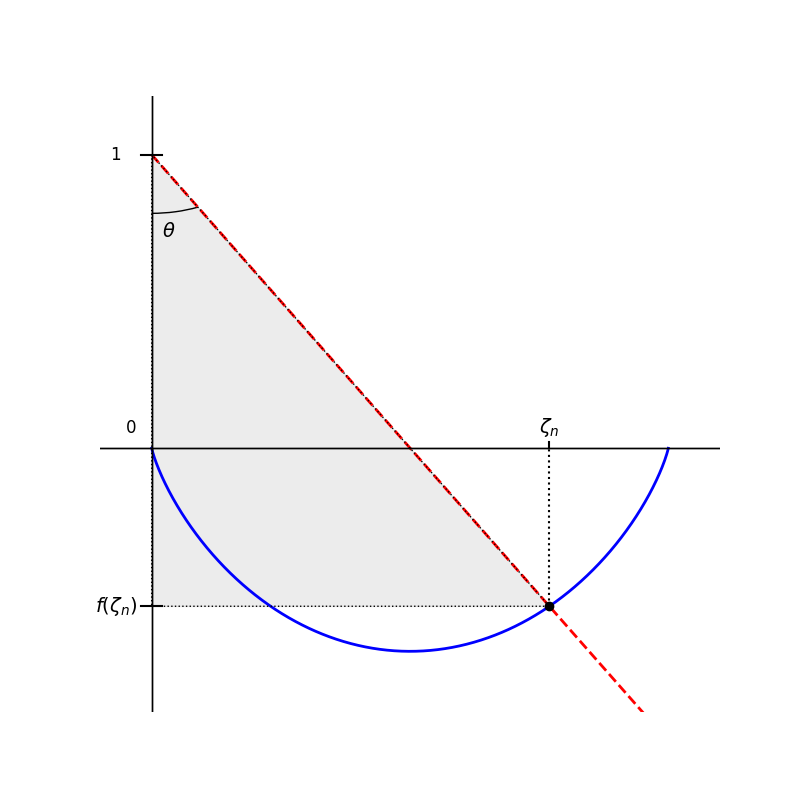}
		\caption{Intersection of the functions $f$ and $g$}\label{fig:triangle-graph}
	\end{center}
	\end{figure}

	Furthermore, consider the triangle formed by the intersection point $(\zeta_n, f(\zeta_n))$ and the points $(0,1)$ and $(0, f(\zeta_n))$ on the y-axis.
	We observe that
	\[
		\frac{1}{\log(1/\mu_n)}
		= \tan\theta
		= \frac{\zeta_n}{1 + f(\zeta_n)}.
	\]
	Since $- \log(2) \leq f(\zeta) \leq 0$ for every $0 \leq \zeta \leq 1$, we deduce the bounds
	\begin{equation*}
		\zeta_n \leq \frac{1}{\log(1/\mu_n)} \leq \frac{\zeta_n}{1 - \log(2)}.
	\end{equation*}

	Combining this estimate with \eqref{eq:zeta_n-h'-s'-from-bgl} allows us to bound the quotient $h'(w_n(v)) / s'(w_n(v))$ from below.
	Substituting this bound into \eqref{eq:mu_n^2-g_n^2-N_n-estimate2} and dividing both sides by $\mu_n^2$, we deduce that there exists a constant $C_3 > 0$ such that
	\begin{equation}
	\label{eq:mu_n^2-g_n^2-N_n-estimate3}
	\begin{split}
		\frac{1}{\log(1/\mu_n)} |v|^2 F(1-F) \frac{g_n^2}{N_n}
		&\leq \frac{C_3}{\mu_n^2}\big[
				F h(\mu_n (1-F) g_n)
				+ (1-F) h(- \mu_n F g_n)
			\big]\\
		&\quad + F(1-F) \frac{C_3}{\mu_n^2} \left[
			\frac{h'(w_n(v))}{s'(w_n(v))} s(w_n(v))
			- h(w_n(v))
		\right].
	\end{split}
	\end{equation}

	We now prove that the right-hand side of \eqref{eq:mu_n^2-g_n^2-N_n-estimate3} is bounded in $L^\infty(dt; L^1(dvdx))$ by studying the two terms of \eqref{eq:mu_n^2-g_n^2-N_n-estimate3} separately.
	The first term in square brackets equals the relative quantum entropy density, from the identity \eqref{eq:relative-entropy-rewriting-with-h}.
	Therefore, its integral over $(x,v) \in \Rxv$ is given by
	\[
		\frac{1}{\mu_n^2}
		\int_{\mathbb{T}^N}
		\int_{\mathbb{R}^N}
			F h(\mu_n (1-F) g_n)
			+ (1-F) h(- \mu_n F g_n)
			\, dxdv
		= \frac{1}{\mu_n^2} H(f_n | F).
	\]
	Using the entropy bound \eqref{eq:fn-relative-entropy}, we conclude that this quantity is bounded by $C^{in}$.

	To estimate the second term of \eqref{eq:mu_n^2-g_n^2-N_n-estimate3}, we return to the proof of Proposition 3.3 in \cite{bgl1993}, which establishes that there exists a constant $C_4 > 0$ such that for all $w \geq 0$,
	\begin{equation}
	\label{eq:h'-s'-s-w-estimate}
		\frac{h'(w)}{s'(w)} s(w) - h(w)
		\leq C_4 (1 + w)
		\left(
			\frac{h'(w)}{s'(w)} - 1
		\right)^2.
	\end{equation}
	Choosing $w = w_n(v)$, we estimate the two factors on the right-hand side separately.
	First, using the definition of $w_n(v)$ in \eqref{eq:definition-wn} and the inequality $\log(1+x) \leq x$, we bound the squared difference
	\begin{equation} \label{eq:squared-h'-s'-estimate}
	\begin{split}
		\left(
			\frac{h'(w_n(v))}{s'(w_n(v))} - 1
		\right)^2
		= \left[
			\log\left(
				1 + \mu_n \exp\left(\frac{1}{8} |v|^2\right)
			\right)
		\right]^2
		\leq \mu_n^2 \exp\left(\frac{1}{4} |v|^2\right).
	\end{split}
	\end{equation}

	Second, let $n_0 \in \mathbb{N}$ be such that $\mu_n \leq 1$ for every $n \geq n_0$.
	From \eqref{eq:definition-wn} and the trivial bound $1 \leq \exp(|v|^2/8)$, we find that for every $n \geq n_0$,
	\begin{equation*}
	\begin{split}
		\frac{h'(w_n(v))}{s'(w_n(v))}
		&\leq 1
		+ \log\left(
			(1 + \mu_n)
			\exp\left(\frac{1}{8} |v|^2\right)
		\right)\\
		&\leq 1 + \log(2)
		+ \frac{1}{8} |v|^2\\
		&\leq 2 + \frac{1}{8} |v|^2.
	\end{split}
	\end{equation*}
	On the other hand, since $(w+3)^2 / (w(w+6)) \geq 1$ for all $w > 0$, equation \eqref{eq:h'-s'-formula} yields the lower bound $h'(w)/s'(w) \geq 2 \log(1 + w)/3$.
	Combining this lower bound evaluated at $w = w_n(v)$ with the upper bound above, we can estimate the second factor of \eqref{eq:h'-s'-s-w-estimate} to obtain
	\begin{equation*}
		1 + w_n(v)
		\leq e^3 \exp\left( \frac{3}{16} |v|^2 \right).
	\end{equation*}

	Using this result alongside \eqref{eq:squared-h'-s'-estimate}, we deduce from \eqref{eq:h'-s'-s-w-estimate} that
	\[
		\frac{h'(w_n(v))}{s'(w_n(v))} s(w_n(v)) - h(w_n(v))
		\leq C_4 e^3 \mu_n^2 \exp\left(\frac{7}{16} |v|^2\right).
	\]
	Multiplying this estimate by $F(1-F)/\mu_n^2$, integrating over $(x,v) \in \Rxv$ and applying the bound $F(v) \leq e \exp(-|v|^2/2)$ yields
	\begin{equation*}
		\frac{1}{\mu_n^2}
		\int_{\mathbb{T}^N}
		\int_{\mathbb{R}^N}
			\left[
				\frac{h'\big(w_n(v)\big)}{s'\big(w_n(v)\big)} s\big(w_n(v)\big)
				- h\big(w_n(v)\big)
			\right]
			F(1-F) \, dv dx
		\leq C_4 e^4 \int_{\mathbb{R}^N}
			\exp\left(-\frac{1}{16} |v|^2\right)\, dv.
	\end{equation*}
	Hence the second term in \eqref{eq:mu_n^2-g_n^2-N_n-estimate3} is uniformly bounded, which implies that $|v|^2 g_n^2 / N_n$ is of order $\log(1/\mu_n)$ in $L^\infty(dt; L^1(F(1-F)dvdx))$.
	Since we have already established that $g_n^2 / N_n$ is uniformly bounded in this space, the desired estimate follows.
\end{proof}

\subsection{Lower semi-continuity and $L^2$ bounds}
\label{subsec:lower-semi-continuity}

The final consequence of the entropy inequality \eqref{eq:entropy-inequality} that we study is that, although the weak compactness results established in Lemmas \ref{lem:tau-gn-compactness} and \ref{lem:collision-product-compactness} operate in $L^1$ spaces, any resulting weak limit also belongs to $L^2$.

Formally, this follows because the function $h$ defined in the introduction to Section \ref{sec:fluctuation-theory} satisfies $h(z) \approx z^2$ as $z \to 0$.
Considering the decomposition \eqref{eq:relative-entropy-rewriting-with-h}, we deduce that in the limit as $\mu_n \to 0$, the integrand of the relative entropy $H(f_n | F)$ satisfies
\[
	F h(\mu_n (1-F) g_n) + (1-F) h(-\mu_n F g_n)
	\approx \mu_n^2 F(1-F) g^2_n.
\]
This implies that, asymptotically, the bound \eqref{eq:fn-relative-entropy} provides control on the $L^2$ norm of the sequence $(g_n)_n$ in a weighted space.
Similarly, using that $r(z) \approx z^2$ as $z \to 0$ in \eqref{eq:r-product-of-p-bound} yields an $L^2$ bound on the sequence $q_n (P_n P_{n,*} M'_n M'_{n,*})^{-1/2}$, which asymptotically corresponds to an $L^2$ bound on $q_n$, since $P_n P_{n,*} M'_n M'_{n,*} \approx 1$ as $\mu_n \to 0$.

These formal observations are made rigorous in the following result:
\begin{lemma}
	\label{lem:lower-semi-continuity-entropy}
	The following lower semi-continuity properties hold:
	\begin{enumerate}[label=\emph{\roman*})]
		\item Let $T>0$.
		If $\tau g_n \weakto \tau g$ in $L^1([0,T] \times \Rxv; dtdx F(1-F) dv)$ then $g \in L^\infty(dt; L^2(F(1-F) dvdx))$ and, for almost every $t \in (0,T)$,
		\begin{equation}
			\label{eq:g-L2-bounded-entropy}
			\frac{1}{2} \int_{\mathbb{T}^N} \vmean{g(t,x, \cdot)^2} \, dx
			\leq \liminf_n \frac{1}{\mu_n^2} H(f_n(t)|F).
		\end{equation}
		\item If $\tau q_n / N_n \weakto \tau q$ in $L^1(dtdx \Lambda)$ then $q \in L^2(dtdx \Lambda)$ and we have, for every $t \geq 0$,
		\[
			\frac{1}{4} \int_0^t \int_{\mathbb{T}^N} \vvmean{q(s,x, \cdot)^2} \, dxds
				\leq \liminf_n \frac{1}{(\mu_n \eta_n)^2} \int_0^t R(P_n)(s) \, ds.
		\]
	\end{enumerate}
\end{lemma}
\begin{remark}
	A result traditionally established alongside the above statement is that the limit collision product $q$ also satisfies the collision symmetries:
	\begin{equation*}
			\begin{split}
		q(t,x,v,v_*,\omega)
		&= q(t,x,v_*,v,\omega)\\
		&= q(t,x,v,v_*,-\omega)\\
		&= - q(t,x,v',v'_*,\omega),
	\end{split}
	\end{equation*}
	where we have used the standard notation $v' \equiv v'(v,v_*,\omega)$ and $v'_* \equiv v'_*(v,v_*,\omega)$ for the post-collisional velocities.
	In \cite{bgl1993}, this property was used to prove the Boussinesq and incompressibility relations in the Stokes limit, prior to the method of vanishing conservation defects.
	In our approach, we establish those relations directly through the vanishing of these defects, rendering this explicit symmetry check optional.

	Nevertheless, this property may be easily proven by introducing the auxiliary sequence $Q_n = q_n / (N_n N_{n,*} N'_n N'_{n,*})$, which possesses the same symmetries as $q_n$.
	Because the factor $1/(N_{n,*} N'_n N'_{n,*})$ is uniformly bounded and, up to a subsequence, converges pointwise almost everywhere to $1$, the product limit theorem implies that $Q_n$ converges weakly in $L^1_{loc}(dt; L^1(dx \Lambda))$ to $q$, and the result follows.
\end{remark}

\begin{proof}[Proof of Lemma \ref{lem:lower-semi-continuity-entropy}]

	In view of the decomposition \eqref{eq:relative-entropy-rewriting-with-h}, let
	\[
		\eta_F(z) = F h((1-F) z) + (1-F) h(- F z)
	\]
	be the density corresponding to relative entropy $H(f | F)$.
	To establish the first result, we employ a truncation argument on the limiting function $g$.
	For every $\lambda \geq 0$, let
	\[
		E_\lambda
		= \{
			(t,x,v) \in [0,T] \times \Rxv
			\quad \text{s.t.} \quad
			|g(t,x,v)| \leq \lambda
		\}.
	\]
	We will work on this set throughout the proof and eventually let $\lambda \to \infty$.

	The proof proceeds in two steps.
	We begin by showing that if we evaluate the entropy integral on the right-hand side of \eqref{eq:g-L2-bounded-entropy} on the fixed limit $g$, rather than the sequence $g_n$, we obtain full convergence to the squared norm.
	That is, we will show that, for almost every $t \in (0,T)$,
	\begin{equation}
		\label{eq:entropy-limit-claim-1}
		\lim_{n \to \infty} \frac{1}{\mu_n^2} \int_{E_\lambda} \eta_F(\mu_n g) \, dxdv
		= \frac{1}{2} \int_{E_\lambda} g^2 F(1-F) \, dxdv.
	\end{equation}
	To prove this, we use that, for every $z \geq -1$,
	\begin{equation}
		\label{eq:z^2-h-inequality}
		0 \leq \frac{1}{2} z^2 - h(z) \leq \frac{1}{6} |z|^3.
	\end{equation}

	Fix $\lambda > 0$.
	Since $\mu_n \to 0$ as $n \to \infty$, for every $n \in \mathbb{N}$, let $n_0(\lambda)$ be an integer such that $|\mu_n| \leq 1/(2\lambda)$ for every $n \geq n_0(\lambda)$.
	Observe that, by definition, for every $n \geq n_0(\lambda)$, the terms $z_1 = \mu_n (1-F) g \mathbbm{1}_{E_\lambda}$ and $z_2 = - \mu_n F g \mathbbm{1}_{E_\lambda}$ both satisfy
	\begin{equation}
		\label{eq:z_i-correct-domain}
		|z_i| \leq |\mu_n| \lambda \leq \frac{1}{2},
	\end{equation}
	and therefore both terms $z_1, z_2$ lie in the domain of validity of the inequality.

	We multiply inequality \eqref{eq:z^2-h-inequality} for $z_1$ by $F$ and add it to the inequality \eqref{eq:z^2-h-inequality} for $z_2$ multiplied by $(1-F)$.
	Observing that $F z_1^2 + (1-F) z_2^2 = \mu_n^2 F(1-F) g^2 \mathbbm{1}_{E_\lambda}$ and using that $\eta_F(\mu_n g \mathbbm{1}_{E_\lambda}) = \eta_F(\mu_n g) \mathbbm{1}_{E_\lambda}$ we obtain, for every $n \geq n_0(\lambda)$,
	\begin{equation*}
		0 \leq \frac{1}{2} \mu_n^2 F (1 - F) g^2 \mathbbm{1}_{E_\lambda}
		- \eta_F(\mu_n g) \mathbbm{1}_{E_\lambda}
		\leq \frac{1}{6} \mu_n^3 F ( 1 - F )
		\left[
			F^2 + (1-F)^2
		\right] |g|^3 \mathbbm{1}_{E_\lambda}.
	\end{equation*}
	We simplify the upper bound using that $F^2 + (1-F)^2 \leq 1$ and that $|g| \leq \lambda$ in $E_\lambda$.
	Dividing the entire expression by $\mu_n^2$ we conclude that for every $n \geq n_0(\lambda)$,
	\begin{equation*}
		0 \leq \frac{1}{2}  F (1 - F) g^2 \mathbbm{1}_{E_\lambda}
		- \frac{1}{\mu_n^2} \eta_F(\mu_n g) \mathbbm{1}_{E_\lambda}
		\leq \frac{1}{6} \mu_n F (1 - F) \lambda^3.
	\end{equation*}

	We now integrate the whole expression over $(x,v) \in \Rxv$.
	For a fixed $\lambda > 0$, passing to the limit as $n \to \infty$, we obtain \eqref{eq:entropy-limit-claim-1}.

	Having established the convergence for a fixed $g$, we now exploit the convexity of $h$ to deduce the lower bound result \eqref{eq:g-L2-bounded-entropy} for the sequence $(g_n)_n$.
	Since $h$ is a differentiable convex function throughout its entire domain, its graph must lie above every tangent line.
	That is, for every $x, y \geq -1$, we have
	\[
		h(x) + h'(x) (y-x) \leq h(y).
	\]

	Consider the points $x_1 = \mu_n (1-F) g \mathbbm{1}_{E_\lambda}$, $y_1 = \mu_n (1-F) g_n \mathbbm{1}_{E_\lambda}$ and $x_2 = -\mu_n F g \mathbbm{1}_{E_\lambda}$, $y_2 = - \mu_n F g_n \mathbbm{1}_{E_\lambda}$.
	For $n \geq n_0(\lambda)$, observe that by \eqref{eq:z_i-correct-domain}, we have $|x_1|, |x_2| \geq 1/2$ and since $0 \leq f_n \leq 1$, it follows that $y_1, y_2 \geq -1$.
	Therefore, these choices of $x_i$ and $y_i$ lie within the domain of validity for the convexity inequality.

	We multiply this inequality for $(x_1, y_1)$ by $F$ and add it to the inequality for $(x_2, y_2)$ multiplied by $(1-F)$.
	Using the identities $\eta_F(\mu_n g \mathbbm{1}_{E_\lambda}) = \eta_F(\mu_n g) \mathbbm{1}_{E_\lambda}$ and $h'(z \mathbbm{1}_{E_\lambda}) = h'(z) \mathbbm{1}_{E_\lambda}$, then dividing the whole expression by $ \mu_n^2$ yields
	\begin{equation}
	\label{eq:convexity-h-psi_n-psi}
	\begin{split}
		\frac{1}{\mu_n^2} \eta_F(\mu_n g) \mathbbm{1}_{E_\lambda}
		+ \frac{1}{\mu_n} F(1-F)
			\Big[
				h'(\mu_n (1-F) g)
				- h'(- \mu_n F g)
			\Big] &(g_n - g) \mathbbm{1}_{E_\lambda} \\
		&\quad\leq
		\frac{1}{\mu_n^2}
		\eta_F(\mu_n g_n) \mathbbm{1}_{E_\lambda}.
	\end{split}
	\end{equation}

	To establish the lower bound \eqref{eq:g-L2-bounded-entropy}, we will first prove that the second term in the left-hand side tends to zero.
	To this end, we rely on the inequality
	\[
		0 \leq z - h'(z) \leq z^2,
	\]
	which holds for every $z \geq -1/2$.
	For $n \geq n_0(\lambda)$, we subtract the inequality for $z = x_2$ from the inequality for $z = x_1$.
	Using that $h'(z \mathbbm{1}_{E_\lambda}) = h'(z) \mathbbm{1}_{E_\lambda}$ and dividing the whole expression by $\mu_n$ yields
	\[
		- \mu_n F^2 g^2 \mathbbm{1}_{E_\lambda}
		\leq g \mathbbm{1}_{E_\lambda}
		- \frac{1}{\mu_n}
		\big[
			h'(\mu_n(1-F) g)
			- h'(- \mu_n F g)
		\big] \mathbbm{1}_{E_\lambda}
		\leq \mu_n (1-F)^2 g^2 \mathbbm{1}_{E_\lambda}.
	\]

	Using that in $E_\lambda$ we have $\mu_n (1-F)^2 g^2 \leq \mu_n \lambda^2$ and $- \mu_n F^2 g^2 \geq - \mu_n \lambda^2$.
	Integrating over $(x,v) \in \Rxv$ we have, for fixed $\lambda > 0$ and $t \in [0,T]$, that
	\[
		\frac{1}{\mu_n}
		\big[
			h'(\mu_n(1-F) g)
			- h'(- \mu_n F g)
		\big] \mathbbm{1}_{E_\lambda}
		\to g \mathbbm{1}_{E_\lambda}
	\]
	strongly in $L^\infty(\Rxv)$.
	Since $F(1-F) g_n \weakto F(1-F) g$ weakly in $L^1([0,T] \times \Rxv)$ it follows, by a weak-strong result, that
	\begin{equation}
		\label{eq:h'-psi_n-psi-convergence}
		\frac{1}{\mu_n}
		F(1-F)
		\big[
			h'(\mu_n(1-F) g)
			- h'(- \mu_n F g)
		\big] (g_n - g) \mathbbm{1}_{E_\lambda}
		\weakto 0
	\end{equation}
	weakly in $L^1([0,T] \times \Rxv)$.

	We integrate inequality \eqref{eq:convexity-h-psi_n-psi} over $(x,v) \in \Rxv$, estimating the right-hand side with $\mathbbm{1}_{E_\lambda} \leq 1$.
	Passing both sides to the $\liminf$ and using the convergences \eqref{eq:entropy-limit-claim-1} and \eqref{eq:h'-psi_n-psi-convergence}, we obtain
	\[
		\frac{1}{2} \int_{\Rxv} F(1-F) g^2 \mathbbm{1}_{E_\lambda} \, dxdv
		\leq \liminf_n \frac{1}{\mu_n^2} \int_{\Rxv} \eta_F(\mu_n g_n) \, dxdv.
	\]
	Applying the Beppo-Levi theorem, we can pass the left-hand side to the limit as $\lambda \to \infty$.
	This establishes \eqref{eq:g-L2-bounded-entropy} and the first assertion in the statement.

	Turning to the proof of the second assertion, we employ essentially the same general technique as in the proof of the first, adapting the spaces and relying on inequalities for $r$ instead of $h$.
	For any $\lambda > 0$, consider the set
	\[
		F_\lambda
		= \{
			(t,x,v,v_*,\omega) \in [0,T] \times
			\mathbb{T}^N_x \times \mathbb{R}^{2N}_{v,v_*} \times \mathbb{S}^{N-1}_\omega
			\quad \text{s.t.} \quad
			|q(t,x,v,v_*,\omega)| \leq \lambda
		\}.
	\]
	Since $\mu_n \eta_n \to 0$, for $\lambda > 0$ let $n_1(\lambda)$ be such that $|\mu_n \eta_n| \leq 1/2$ for every $n \geq n_1(\lambda)$.
	Consider the inequality, valid for every $z \geq -1$,
	\[
		0 \leq z^2 - r(z) \leq \frac{1}{2} |z|^3.
	\]
	Fix $\lambda > 0$.
	For every $n \geq n_1(\lambda)$, we choose $z = \mu_n \eta_n q \mathbbm{1}_{F_\lambda}$ in this inequality (noting that $|z| \leq 1/2$).
	Using that $r(\mu_n \eta_n q \mathbbm{1}_{F_\lambda}) = r(\mu_n \eta_n q) \mathbbm{1}_{F_\lambda}$, estimating the right-hand side with $q^3 \mathbbm{1}_{F_\lambda} \leq \lambda^3$ and dividing the whole expression by $(\mu_n \eta_n)^2$, we obtain
	\[
		0 \leq q^2 \mathbbm{1}_{F_\lambda}
		- \frac{1}{(\mu_n \eta_n)^2} r(\mu_n \eta_n q) \mathbbm{1}_{F_\lambda}
		\leq \frac{1}{2} \mu_n \eta_n \lambda^3.
	\]
	This uniform estimate then implies the convergence
	\begin{equation}
		\label{eq:r-conv-to-q2}
		\frac{1}{(\mu_n \eta_n)^2} r(\mu_n \eta_n q) \mathbbm{1}_{F_\lambda}
		\to q^2 \mathbbm{1}_{F_\lambda}
	\end{equation}
	in $L^\infty(dtdx \Lambda)$, with $\Lambda$ as defined in \eqref{eq:measure-lambda-def}.

	Next, we use that, for every $z > -1$,
	\[
		0 \leq 2z - r'(z)
		\leq \frac{3 z^2}{1+z}.
	\]
	For $n \geq n_1(\lambda)$, we substitute once again $z = \mu_n \eta_n q \mathbbm{1}_{F_\lambda}$ in the above inequality. Using that $r'(\mu_n \eta_n q \mathbbm{1}_{F_\lambda}) = r'(\mu_n \eta_n q) \mathbbm{1}_{F_\lambda}$ and dividing the whole expression by $\mu_n \eta_n$ we obtain,
	\[
		0 \leq 2q \mathbbm{1}_{F_\lambda}
		- \frac{1}{\mu_n \eta_n} r'(\mu_n \eta_n q) \mathbbm{1}_{F_\lambda}
		\leq 3 \frac{\mu_n \eta_n q^2}{1 + \mu_n \eta_n q} \mathbbm{1}_{F_\lambda}.
	\]
	Since for $n \geq n_1(\lambda)$ we have $|\mu_n \eta_n q| \leq 1/2$ on the set $F_\lambda$, the denominator on the right-hand side satisfies $1 + \mu_n \eta_n q \geq 1/2$.
	Therefore, from the definition of $F_\lambda$, the right-hand side is bounded by $6 \mu_n \eta_n \lambda^2$, which implies the convergence
	\begin{equation}
		\label{eq:r'-conv-to-2q}
		\frac{1}{\mu_n \eta_n} r'(\mu_n \eta_n q) \mathbbm{1}_{F_\lambda}
		\to 2q \mathbbm{1}_{F_\lambda}
	\end{equation}
	in $L^\infty(dtdx \Lambda)$.

	Because $r$ is a convex differentiable function we have, for every $x, y > -1$,
	\[
		r(x) + r'(x) (y-x) \leq r(y).
	\]
	To handle potential singularities in the quotient, let $(\alpha_n)_n$ be a sequence of strictly positive numbers such that $\alpha_n \to 0$.
	We introduce the renormalization factor
	\begin{equation}
		\label{eq:renormalization-factor-absolute-value}
	    N^{abs}_n = 1 + \left(\frac{3}{4} - F\right) |\mu_n g_n|
	    \geq 1.
	\end{equation}
	For $n \geq n_1(\lambda)$, we substitute
	\[
	    x = \mu_n \eta_n q \mathbbm{1}_{F_\lambda}
	    \quad \text{and} \quad
	    y = \frac{\mu_n \eta_n q_n }{P_n P_{n,*} M'_n M'_{n,*} + \alpha_n} \mathbbm{1}_{F_\lambda}
	\]
	into the convexity inequality of $r$.
	Multiplying the result by $P_n P_{n,*} M'_n M'_{n,*} / N^{abs}_n \geq 0$, we use that $1/N^{abs}_n \leq 1$ and the trivial bound $\mathbbm{1}_{F_\lambda} \leq 1$ to find an upper bound for the right-hand side.
	Dividing the whole expression by $(\mu_n \eta_n)^2$, we obtain
	\begin{equation}
	\label{eq:r-r'-convexity-inequality}
	\begin{split}
		&\frac{1}{(\mu_n \eta_n)^2} r(\mu_n \eta_n q)
		\frac{P_n P_{n,*} M'_n M'_{n,*}}{N^{abs}_n}
		\mathbbm{1}_{F_\lambda}\\
		&\quad + \frac{1}{\mu_n \eta_n} r'(\mu_n \eta_n q)
		\left[
			\frac{q_n}{P_n P_{n,*} M'_n M'_{n,*} + \alpha_n}
			- q
		\right]
		\frac{P_n P_{n,*} M'_n M'_{n,*}}{N^{abs}_n}
		\mathbbm{1}_{F_\lambda} \\
		&\leq
		\frac{1}{(\mu_n \eta_n)^2}
		r\left(
			\frac{\mu_n \eta_n q_n}{P_n P_{n,*} M'_n M'_{n,*} + \alpha_n}
		\right)
		P_n P_{n,*} M'_n M'_{n,*}.
	\end{split}
	\end{equation}

	To establish the lower semi-continuity of the collision product, we study the $L^1(dtdx\Lambda)$ limit of each term on the left-hand side of the preceding inequality.
	For the first term, we start by studying the limit of $P_n P_{n,*} M'_n M'_{n,*} / N^{abs}_n $.
	Using the bounds $M'_n, M'_{n,*} \leq 1+e$ and the estimate $P_n/N_n \leq 4/3$ from Lemma \ref{lem:Nn-Pn-bounds}, we obtain the domination
	\[
		\frac{P_n P_{n,*} M'_n M'_{n,*}}{N^{abs}_n}
		\leq (1+e)^2 \frac{P_n}{N_n} P_{n,*}
		\leq \frac{4}{3} (1+e)^2 P_{n,*}
		= \frac{4}{3} (1+e)^2
		\left[
			1
			+ \mu_n (1-F_*) g_{n,*}
		\right],
	\]

	Given that $(\tau g_n)_n$ is uniformly bounded in $L^\infty(dt; L^1(F(1-F)dvdx))$, we deduce that $\mu_n (1 - F_*) g_{n,*} \to 0$ in $L^\infty(dt; L^1(dx \Lambda))$.
	Therefore, the right-hand side above converges in $L^1_{loc}(dt; L^1(dx \Lambda))$.
	Furthermore, passing to a subsequence, we may assume that $\mu_n g_n \to 0$ almost everywhere in $[0,T] \times \Rxv$, which yields
	\begin{equation}
		\label{eq:P-P-M-M-convergence}
		\frac{P_n P_{n,*} M'_n M'_{n,*}}{N^{abs}_n}
		\to 1
	\end{equation}
	almost everywhere in $[0,T] \times \mathbb{T}^N_x \times \mathbb{R}^{2N}_{v,v_*} \times \mathbb{S}^{N-1}_\omega$.
	The generalized dominated convergence theorem then implies that the convergence \eqref{eq:P-P-M-M-convergence} holds in $L^1_{loc}(dt; L^1(dx \Lambda))$.
	Combining this with the convergence \eqref{eq:r-conv-to-q2}, we have that
	\begin{equation}
		\label{eq:r-convergence-first-term}
		\int_0^t
		\int_{\mathbb{T}^N}
			\vvmean{
				\frac{1}{(\mu_n \eta_n)^2} r(\mu_n \eta_n q)
				\frac{P_n P_{n,*} M'_n M'_{n,*}}{N^{abs}_n}
				\mathbbm{1}_{F_\lambda}
			}
		\, dxds
		\to
		\int_0^t
		\int_{\mathbb{T}^N}
			\vvmean{
				q^2 \mathbbm{1}_{F_\lambda}
			}
		\, dxds,
	\end{equation}
	for every $t > 0$.

	For the second term on the left-hand side of \eqref{eq:r-r'-convexity-inequality}, we write
	\begin{equation}
	\label{eq:q_n-q-times-product-Pn}
	\begin{split}
		&\left[
			\frac{q_n}{P_n P_{n,*} M'_n M'_{n,*} + \alpha_n}
			- q
		\right]
		\frac{P_n P_{n,*} M'_n M'_{n,*}}{N^{abs}_n}
		\mathbbm{1}_{F_\lambda} \\
		&\quad= \frac{q_n}{P_n P_{n,*} M'_n M'_{n,*} + \alpha_n} \frac{P_n P_{n,*} M'_n M'_{n,*}}{N^{abs}_n} \mathbbm{1}_{F_\lambda}
		- q \frac{P_n P_{n,*} M'_n M'_{n,*}}{N^{abs}_n}
		\mathbbm{1}_{F_\lambda}.
	\end{split}
	\end{equation}
	Call the first term on the right-hand side $T_1$ and the second term $T_2$.
	We rewrite $T_1$ as
	\[
		T_1 = \frac{q_n}{N_n}
		\frac{N_n}{N^{abs}_n}
		\frac{P_n P_{n,*} M'_n M'_{n,*}}{P_n P_{n,*} M'_n M'_{n,*} + \alpha_n}
		\mathbbm{1}_{F_\lambda}.
	\]
	By Lemma \ref{lem:collision-product-compactness}, the first factor converges to $q$ in $L^\infty(dt; L^1(dx \Lambda))$. The second and third factors are nonnegative, uniformly bounded by $1$, and converge almost everywhere to $1$.
	Therefore, by the product limit theorem (see, for example, the Appendix of \cite{bgl1993}), we have that $T_1 \weakto q \mathbbm{1}_{F_\lambda}$ weakly in $L^1_{loc}(dt; L^1(dx \Lambda))$.

	For $T_2$, the convergence \eqref{eq:P-P-M-M-convergence} in $L^1_{loc}(dt; L^1(dx \Lambda))$ yields that $T_2$ converges to $q \mathbbm{1}_{F_\lambda}$ strongly in $L^1_{loc}(dt; L^1(dx \Lambda))$.
	Since we previously showed that $T_1$ converges weakly to the same limit, the expression \eqref{eq:q_n-q-times-product-Pn} converges weakly to $0$ in $L^1_{loc}(dt; L^1(dx \Lambda))$.

	We now pair this weak convergence with the strong $L^\infty$ convergence of the $r'$ factor derived in \eqref{eq:r'-conv-to-2q}.
	A standard weak-strong convergence argument yields
	\[
		\int_0^t \int_{\mathbb{T}^N}
		\vvmean{
			\frac{1}{\mu_n \eta_n}
			r'(\mu_n \eta_n q)
			\left[
				\frac{q_n}{P_n P_{n,*} M'_n M'_{n,*} + \alpha_n}
				- q
			\right]
			\frac{P_n P_{n,*} M'_n M'_{n,*}}{N^{abs}_n}
			\mathbbm{1}_{F_\lambda}
		} \, dxds
		\to 0.
	\]

	We now integrate the convexity inequality \eqref{eq:r-r'-convexity-inequality} over $[0,t] \times  \mathbb{T}^N_x \times \mathbb{R}^{2N}_{v,v_*} \times \mathbb{S}^{N-1}_\omega$.
	Using the limit established above for the second term, together with the convergence \eqref{eq:r-convergence-first-term} for the first term, we take the $\liminf$ as $n \to \infty$ to obtain that
	\begin{equation}
		\label{eq:lower-semic-q2-F-lambda}
		\begin{split}
			&\int_0^t \int_{\mathbb{T}^N}
				\vvmean{
					q^2 \mathbbm{1}_{F_\lambda}
				}
			\, dxds \\
			&\leq \liminf_{n\to\infty}
			\int_0^t \int_{\mathbb{T}^N}
				\vvmean{
					\frac{1}{(\mu_n \eta_n)^2}
					r\left(
						\frac{\mu_n \eta_n q_n}{P_n P_{n,*} M'_n M'_{n,*} + \alpha_n}
					\right)
					P_n P_{n,*} M'_n M'_{n,*}
				}
			\, dxds.
		\end{split}
	\end{equation}
	Finally, observing that $r$ is convex with $r(0) = 0$, we have $r(cz) \leq c r(z) \leq r(z)$ for any $c \in (0,1)$.
	Therefore, for every integer $n$, we have
	\[
		r\left(
			\frac{\mu_n \eta_n q_n}{P_n P_{n,*} M'_n M'_{n,*} + \alpha_n}
		\right)
		\leq
		r\left(
			\frac{\mu_n \eta_n q_n}{P_n P_{n,*} M'_n M'_{n,*}}
		\right).
	\]
	Applying this inequality to \eqref{eq:lower-semic-q2-F-lambda} we recognize the expression for the time integral of $4 R(f_n)$ on the right-hand side.
	The result then follows by passing to the limit as $\lambda \to \infty$, applying the Beppo-Levi theorem on the left-hand side.
\end{proof}

\section{The linearized collision operator and the limit fluctuation}
\label{sec:linearized-operator}

Linearizing the collision integral $Q(f)$ around the equilibrium using the perturbation $f = F + \mu F(1-F) \phi$ and applying the equilibrium balance identity \eqref{eq:equilibrium-balance-identity}, we obtain the linearized collision operator.
The general convention in the literature is to flip the sign and divide the expression by $F(1-F)$, ensuring the resulting definition is positive semi-definite and self-adjoint, respectively (see, for example, \cite[Chapter II]{zakrevskiy-thesis-2015}).

Thus, for any function $\phi = \phi(v)$, we define the linearized Fermi-Dirac collision operator as
\begin{equation}
	\label{eq:linearized-op-def}
	\mathcal{L}(\phi)
	= \frac{1}{F(1-F) }\iint_{\mathbb{R}^N \times \mathbb{S}^{N-1}} b(v-v_*, \omega) F F_* (1-F')(1-F'_*) (\phi + \phi_* - \phi' - \phi'_*) \, dv_* d\omega.
\end{equation}

Notice this operator closely resembles the classical linearized Boltzmann collision operator.
In fact, replacing the weight $F F_* (1-F') (1-F'_*) / (F (1-F))$ with the classical Maxwellian $M_*$ recovers the classical operator exactly.
Accordingly, as noted in Remark 2.1 of \cite{jiang-2022}, the foundational properties established for the classical theory (proven, for example, in \cite{glassey1996}) transfer naturally to the quantum analogue.
We recall these fundamental properties in the following result.

\begin{lemma}
	\label{lem:linearized-properties}
	The linearized collision operator $\mathcal{L}$ is a positive semi-definite unbounded self-adjoint Fredholm operator on $L^2(F(1-F) dv)$ with domain $L^2(\tau F(1-F) dv)$ and null space
	\[
		\ker(\mathcal{L})
		= \operatorname{span}\{1, v_1, v_2, \dots, v_N, |v|^2\}.
	\]
\end{lemma}

Furthermore, by estimating its norm, we show that the linearized collision operator is continuous between weighted $L^1$ spaces.
This continuity plays a crucial role in Lemma \ref{lem:limit-local-maxwellian}, where we determine the limiting form of the fluctuations $(g_n)_n$.

\begin{lemma}
	\label{lem:L-continuity}
	The operator $\mathcal{L}$ defined in \eqref{eq:linearized-op-def} is a bounded operator from $L^1( \tau F(1-F) dv )$ to $L^1( F(1-F) dv)$.
\end{lemma}
\begin{proof}
	Let $\phi \in L^1(F(1-F) dv)$.
	To estimate the weighted norm, we integrate $|\mathcal{L}(\phi)|$ over $v \in \mathbb{R}^N$ against the weight $F(1-F)$.
	Successive applications of the triangle inequality yield
	\begin{align*}
		\int_{\mathbb{R}^N} F(1-F) |\mathcal{L}(\phi)| \, dv
		&\leq \iiint_{\mathbb{R}^{2N} \times \mathbb{S}^{N-1}}
			|\phi + \phi_* |
			\, F(1-F) \,  d\Lambda(v,v_*,\omega) \\
		&\quad + \iiint_{\mathbb{R}^{2N} \times \mathbb{S}^{N-1}}
			|\phi' + \phi'_*|
			\, F(1-F) \,  d\Lambda(v,v_*,\omega),
	\end{align*}
	where $\Lambda$ denotes the measure defined in \eqref{eq:measure-lambda-def}.
	We denote the two integrals on the right-hand side by $I_1$ and $I_2$, respectively.

	To handle $I_2$, we apply the pre/post-collisional change of variables $(v, v_*) \leftrightarrow (v', v'_*)$.
	From the the equilibrium balance identity $F F_* (1-F') (1-F'_*) = F' F'_*(1-F)(1-F_*)$
	we deduce that the collision measure $\Lambda$ is invariant by this transformation, and therefore $I_2$ rewrites as
	\begin{align*}
		I_2 &= \iiint_{\mathbb{R}^{2N} \times \mathbb{S}^{N-1}}   |\phi' + \phi'_*| F(1-F) \, d\Lambda(v,v_*,\omega) \\
		& = \iiint_{\mathbb{R}^{2N} \times \mathbb{S}^{N-1}}  |\phi + \phi_*| F'(1-F') \, d\Lambda(v', v'_*, \omega) \\
		&= \iiint_{\mathbb{R}^{2N} \times \mathbb{S}^{N-1}} |\phi + \phi_*| F'(1-F') \, d\Lambda(v,v_*,\omega).
	\end{align*}
	Applying the bounds $F(1-F) \leq 1$ and $F'(1-F') \leq 1$ to $I_1$ and $I_2$, respectively, and expanding the definition of $\Lambda$, we obtain the following estimate.
	Note that we have also used the inequalities $1 - F' \leq 1$ and $1 - F'_* \leq 1$ to conclude that
	\begin{equation*}
	\begin{split}
		\int_{\mathbb{R}^N} F(1-F) |\mathcal{L}(\phi)| \, dv
		&\leq 2 \iiint_{\mathbb{R}^{2N} \times \mathbb{S}^{N-1}} |\phi + \phi_*|  \, d\Lambda(v,v_*,\omega) \\
		&\leq 2 \iiint_{\mathbb{R}^{2N} \times \mathbb{S}^{N-1}} b(|v-v_*|, \omega) |\phi + \phi_*| F F_* \, dv dv_* d\omega.
	\end{split}
	\end{equation*}

	Substituting the hard potentials bound \eqref{eq:collision-kernel-bound} on the collision kernel, applying the triangle inequality and evaluating the integral over $\omega$, we obtain
	\begin{equation*}
	\begin{split}
		\int_{\mathbb{R}^N} F(1-F) |\mathcal{L}(\phi)| \, dv
		&\leq 2C_b |\mathbb{S}^{N-1}| \iint_{\mathbb{R}^{2N}} (1 + |v|) (1 + |v_*|) (|\phi| + |\phi_*|) F F_* \, dv dv_* \\
		&= 4C_b |\mathbb{S}^{N-1}|
		\left(
			\int_{\mathbb{R}^N} (1+|v|) F \, dv
		\right)
		\int_{\mathbb{R}^N} (1+|v|) F |\phi| \, dv.
	\end{split}
	\end{equation*}
	In the final equality, we decoupled the integrals and relabeled the $v_*$ variables as $v$.
	Finally, using that $1 + |v| \leq 2 (1+ |v|^2) = 8 \tau$ together with the bound $1/(1 - F) \leq 1 + e$, we derive
	\[
		\| \mathcal{L}(\phi) \|_{L^1( F(1-F) dv)}
		\leq 32 C_b |\mathbb{S}^{N-1}| (1 + e) K_1
		\|\phi \|_{L^1( \tau F(1-F) dv )},
	\]
	where $K_1 = \int_{\mathbb{R}^N} (1+|v|) F \, dv$.
	This inequality then achieves the proof.
\end{proof}

The following result establishes the limiting form of any weakly convergent sequence of fluctuations $(g_n)$, providing a rigorous justification for the approximation \eqref{eq:F-varpi-U-theta-approx}.

\begin{lemma}
	\label{lem:limit-local-maxwellian}
	Let $T>0$.
	If $\tau g_n \weakto \tau g$ in $L^1([0,T] \times \Rxv; dtdx F(1-F) dv)$, then $g$ can be expressed as a local Fermi-Dirac distribution, that is,
	\[
		g(t,x,v) = \rho(t,x)
			+ u(t,x) \cdot v
			+ \theta(t,x) \frac{1}{2} \left( |v|^2 - K \right),
	\]
\end{lemma}
\begin{proof}

	We aim to show that $\mathcal{L}(g) = 0$.

	The core of the proof relies on the observation that, as established previously, the collision term scales globally as $\mu_n \eta_n$.
	However, by applying the micro-macro decomposition \eqref{eq:micro-macro-decomposition}, this collision product $\mu_n \eta_n q_n$ can be expanded as the sum of a linear part, $- \mu_n \mathcal{L}(g_n)$, and remainder terms of strictly higher order in $\mu_n$.
	Therefore, in the limit as $\mu_n, \eta_n \to 0$, the higher-order terms vanish faster than the linear part, implying $\mathcal{L}(g_n) \to 0$.

	In practice, it is easier to exploit the boundedness of the terms $M_n$ rather than fully expanding the term $\mu_n \eta_n q_n$ in terms of $g_n$.
	Using the relation $P_n = M_n + \mu_n g_n$, we have
	\begin{equation*}
	\begin{split}
		\mu_n \eta_n q_n
		&= P'_n P'_{n,*} M_n M_{n,*} - P_n P_{n,*} M'_n M'_{n,*} \\
		&= \mu_n \big[
			g'_n M'_{n,*} M_n M_{n,*}
			+ g'_{n,*} M'_n M_n M_{n,*}
			- g_n M_{n,*} M'_n M'_{n,*}
			- g_{n,*} M_n M'_n M'_{n,*}
		\big] \\
		&\quad + \mu^2_n \big[
			g'_n g'_{n,*} M_n M_{n,*}
			- g_n g_{n,*} M'_n M'_{n,*}.
		\big]
	\end{split}
	\end{equation*}
	Dividing the entire expression by $\mu_n$ and rearranging the linear terms to the left-hand side, we obtain the decomposition
	\begin{equation}
		\label{eq:linear-part-decomposition}
		g_n + g_{n,*} - g'_n - g'_{n,*}
		= \mu_n [g'_n g'_{n,*} M_n M_{n,*} - g_n g_{n,*} M'_n M'_{n,*}]
		- \eta_n q_n
		+ R_n,
	\end{equation}
	where $R_n$ is given by
	\begin{equation*}
	\begin{split}
		R_n &= g_n (1 - M_{n,*} M'_n M'_{n,*})
			+ g_{n,*} (1 - M_n M'_n M'_{n,*})\\
			&\quad - g'_n (1 - M'_{n,*} M_n M_{n,*})
			- g'_{n,*} (1 - M'_n M_n M_{n,*}).
	\end{split}
	\end{equation*}

	Consider the integrated renormalization factor
	\[
	    Z_n \coloneq
	    \int_{\mathbb{R}^N} \tau
	    	\left[
	    		1
	    		+ \frac{1}{3} \left(\frac{3}{4} - F\right) |\mu_n g_n|
	    	\right]  F \, dv.
	\]
	To exploit the compactness properties established in Lemmas \ref{lem:tau-gn-compactness} and \ref{lem:collision-product-compactness}, we divide the entire expression \eqref{eq:linear-part-decomposition} by $N_n Z_n$.
	Substituting into the definition \eqref{eq:linearized-op-def} of $\mathcal{L}$, we obtain that
	\begin{equation}
		\label{eq:limit-maxw-L-gn-decomposition}
		\frac{F(1-F) \mathcal{L}(g_n)}{N_n Z_n}
		= \mu_n I_1 - \eta_n I_2 + I_3
	\end{equation}
	where
	\begin{align*}
		I_1 &\coloneq
		\begin{aligned}[t]
			\frac{1}{N_n Z_n}
			\iint_{\mathbb{R}^N \times \mathbb{S}^{N-1}} b(v-v_*, \omega)
			& F' F'_* (1-F) (1-F_*) \\
			&\quad \times \big[
				g'_n g'_{n,*} M_n M_{n,*}
				- g_n g_{n,*} M'_n M'_{n,*}
			\big] \, dv_* d\omega,
		\end{aligned} \\
		I_2 &\coloneq
			\frac{1}{Z_n}
			\iint_{\mathbb{R}^N \times \mathbb{S}^{N-1}} b(v-v_*, \omega) F' F'_* (1-F) (1-F_*) \frac{q_n}{N_n} \, dv_* d\omega, \\
		I_3 &\coloneq
			\frac{1}{N_n Z_n}
			\iint_{\mathbb{R}^N \times \mathbb{S}^{N-1}} b(v-v_*, \omega) F' F'_* (1-F) (1-F_*)
				R_n \, dv_* d\omega.
	\end{align*}

	To estimate $I_1$, we start by integrating over $v \in \mathbb{R}^N$ and estimating the renormalization factor using Lemma \ref{lem:Nn-Pn-bounds}.
	This way, for every $(t, x) \in (0,\infty) \times \mathbb{T}^N$, we have,
	\begin{equation*}
	\begin{split}
		\| I_1(t,x) \|_{L^1(\mathbb{R}^N)}
		\leq \frac{4}{Z_n}
		\iiint_{\mathbb{R}^{2N} \times \mathbb{S}^{N-1}}
			&| b(v-v_*, \omega) |
			F' F'_* (1-F) (1-F_*) \\
			& \quad \times \big[
				| g'_n g'_{n,*} | | M_n M_{n,*} |
				+ | g_n g_{n,*} | | M'_n M'_{n,*} |
			\big]
			\, dvdv_*d\omega.
	\end{split}
	\end{equation*}
	We separate the integrals and use the equilibrium balance equation \eqref{eq:equilibrium-balance-identity}.
	Then, given the bound $|(1-F) M_n| = |1 - f_n| \leq 1$ and similar relations for $M_{n,*}$, $M'_n$ and $M'_{n,*}$, it follows that
	\begin{equation*}
	\begin{split}
		\| I_1(t,x) \|_{L^1(\mathbb{R}^N)}
		&\leq \frac{4}{Z_n}
		\iiint_{\mathbb{R}^{2N} \times \mathbb{S}^{N-1}}
			| b(v-v_*, \omega) |
			F' F'_*  | g'_n g'_{n,*} |
			\, dvdv_*d\omega\\
		&\quad + \frac{4}{Z_n}
		\iiint_{\mathbb{R}^{2N} \times \mathbb{S}^{N-1}}
			| b(v-v_*, \omega) |
			F F_*  | g_n g_{n,*} |
			\, dvdv_*d\omega \\
		&= \frac{8}{Z_n}
		\iiint_{\mathbb{R}^{2N} \times \mathbb{S}^{N-1}}
			| b(v-v_*, \omega) |
			F F_*  | g_n g_{n,*} |
			\, dvdv_*d\omega.
	\end{split}
	\end{equation*}
	The equality in the last line follows from an application of a pre/post-collisional change of variables $(v, v_*) \leftrightarrow (v', v'_*)$ in the first integral.
	We apply the bound \eqref{eq:collision-kernel-bound} on the collision kernel and we evaluate the integral over $\omega$ to obtain
	\begin{equation*}
		\| I_1(t,x) \|_{L^1(\mathbb{R}^N)}
		\leq \frac{8}{Z_n} C_b | \mathbb{S}^{N-1} |
			\left[
				\int_{\mathbb{R}^N} (1 + |v|) F |g_n| \, dv
			\right]^2.
	\end{equation*}

	We estimate the integrand on the right-hand side using that $1 + |v| \leq 8 \tau$ and that $(3/4 - F)^{-1} \leq 4(1+e)/(3-e)$.
	These bounds imply there exists a constant $C_1 > 0$ such that $(1 + |v|) F |g_n| \leq C_1 \tau F (3/4 - F) |g_n|$.
	Substituting this estimate back into the inequality for $\| I_1(t,x) \|_{L^1}$ and expanding the definition of $Z_n$, we deduce that there exists a constant $C_2 > 0$ such that
	\begin{equation*}
		\| I_1(t,x) \|_{L^1(\mathbb{R}^N)}
		\leq C_2
		\frac{
			\left[
				\displaystyle \int_{\mathbb{R}^N}  \left(\frac{3}{4} - F\right) |g_n|
				\tau F \, dv
			\right]^2
		}{
			\| \tau F \|_{L^1}
			+ \displaystyle \frac{1}{3} \int_{\mathbb{R}^N} \left(\frac{3}{4} - F\right) \mu_n |g_n| \tau F \, dv
		}.
	\end{equation*}

	Multiplying and dividing the numerator by $\mu_n^2 \| \tau F \|^2_{L^1}$ and factoring $\| \tau F \|_{L^1}$ out of the denominator, we can rewrite the fraction on the right-hand side in terms of the function $s(z)$, defined in \eqref{eq:z-function}, yielding
	\begin{equation}
		\label{eq:limit-maxw-I1-inequality-s}
		\begin{split}
			\| I_1(t,x) \|_{L^1(\mathbb{R}^N)}
			\leq \frac{C_2}{\mu_n^2} \| \tau F \|_{L^1}\,
			s\left(
				\int_{\mathbb{R}^N} \Big(\frac{3}{4} - F\Big) \mu_n |g_n| \frac{\tau F}{\| \tau F \|_{L^1}} \, dv
			\right).
		\end{split}
	\end{equation}
	We then apply Jensen's inequality to the convex function $s(z)$ with respect to the probability measure $\tau F / \| \tau F \|_{L^1} \, dv$.
	Estimating the integrand with $s(|z|) \leq 3z^2/(1+z)$ for every $z > -1$ and recognizing the renormalization factor $N_n$ \eqref{eq:3-4-normalization} in the denominator $1 + z$, we extract an upper bound in terms of $g_n^2 / N_n$.
	Finally, using that $(3/4 - F)^2 \leq (1-F)$, we obtain that
	\begin{equation*}
	\begin{split}
		s\left(
			\int_{\mathbb{R}^N} \left(\frac{3}{4} - F\right) \mu_n |g_n| \frac{\tau F}{\| \tau F \|_{L^1}} \, dv
		\right)
		&\leq \int_{\mathbb{R}^N} s\left(
			\left(\frac{3}{4} - F\right) \mu_n |g_n|
		\right)
		\frac{\tau F}{\| \tau F \|_{L^1}} \, dv \\
		&\leq \frac{3\mu_n^2}{\| \tau F \|_{L^1}} \int_{\mathbb{R}^N}
			\tau \frac{g_n^2}{N_n}
			F \left(\frac{3}{4} - F\right)^2
			\, dv, \\
		&\leq \frac{3\mu_n^2}{\| \tau F \|_{L^1}} \vmean{\frac{g_n^2}{N_n}}.
	\end{split}
	\end{equation*}
	Substituting this back into \eqref{eq:limit-maxw-I1-inequality-s} and integrating over $x \in \mathbb{T}^N$ we conclude, from Lemma \ref{lem:gn2-Nn-bounds}, that the term $I_1$ is of order $\log (1/\mu_n)$ in $L^\infty((0,\infty); L^1(\Rxv))$.
	Therefore, for the first term on the right-hand side of \eqref{eq:limit-maxw-L-gn-decomposition}, we have the convergence
	\begin{equation}
		\label{eq:limit-maxw-I1-conv}
		\mu_n I_1 \to 0
		\quad \text{in }
		L^\infty((0,\infty); L^1(\Rxv)).
	\end{equation}

	For the second term, we fix $T>0$.
	Integrating the expression of $I_2$ over $v \in \mathbb{R}^N$ and using the estimate $Z_n \geq \| \tau F \|_{L^1}$ we obtain, for every $(t,x) \in (0,T) \times \mathbb{T}^N$,
	\[
		\| I_2(t,x) \|_{L^1(\mathbb{R}^N)}
		\leq \frac{1}{\| \tau F \|_{L^1}}
		\iiint_{\mathbb{R}^{2N} \times \mathbb{S}^{N-1}}
			\left|
				\frac{q_n}{N_n}
			\right|
			d\Lambda(v, v_*, \omega),
	\]
	where $\Lambda$ is the measure defined in \eqref{eq:measure-lambda-def}.
	Further integrating in $(t,x)$ we deduce, as a consequence of Lemma \ref{lem:collision-product-compactness} that, for every $T>0$, the $L^1((0,T) \times \Rxv; \, dtdxdv)$ norm of $I_2$ is bounded by a constant.
	Hence, for the second term of \eqref{eq:limit-maxw-L-gn-decomposition}, we have the convergence
	\begin{equation}
		\label{eq:limit-maxw-I2-conv}
		\eta_n I_2 \to 0
		\quad \text{in }
		L^1((0,T) \times \Rxv).
	\end{equation}

	Finally, we consider $I_3$.
	Although we cannot establish suitable norm bounds to guarantee strong convergence, we can nevertheless show that this term converges weakly to zero in $L^1$, which suffices for our result.

	To establish this weak convergence, we start by showing the weak convergence of the term $R_n$.
	Fix a final time $T > 0$.
	Our first step is to show that
	\begin{equation}
		\label{eq:limit-maxw-I3-weak-convergence}
		g_n \weakto g
		\quad \text{in }
		L^1([0,T] \times \mathbb{T}^N_x \times \mathbb{R}^{2N}_{v,v_*} \times \mathbb{S}^{N-1}_\omega; \, dtdx\Lambda).
	\end{equation}
	To this end, take an arbitrary test function $\varphi \in L^\infty([0,T] \times \mathbb{T}^N_x \times \mathbb{R}^{2N}_{v,v_*} \times \mathbb{S}^{N-1}_\omega; dtdx\Lambda)$.
	From $\varphi$, we construct a new function $\psi$ on $[0,T] \times \Rxv$, defined by
	\[
		\psi(t,x,v) =
		\frac{1}{\tau} \iint_{\mathbb{R}^N \times S^{N-1}}
			b(v-v_*,\omega) \frac{F_* (1-F') (1-F'_*)}{1-F}
			\varphi \, dv_* d\omega.
	\]

	Let us show that $\psi \in L^\infty([0,T] \times \mathbb{T}^N_x \times \mathbb{R}^{N}_{v}; F(1-F) dvdxdt)$.
	Applying the estimate \eqref{eq:b-v-v*-estimate} on the collision kernel, alongside the bounds $(1-F') (1-F'_*) \leq 1$ and $1/(1-F) \leq 1+e$, we evaluate the integral over $\omega \in \mathbb{S}^{N-1}_\omega$ to obtain that
	\[
		|\psi(t,x,v)| \leq C_b
			\| \varphi \|_{L^\infty}
			\frac{1+|v|}{\tau}
			(1 + e)
			| \mathbb{S}^{N-1} |
			\int_{\mathbb{R}^N} (1 + |v_*|) F_* \, dv_*.
	\]
	Since the integral on the right-hand side is finite and $(1 + |v|)\leq 8 \tau$, this establishes the required bound.

	Taking $\psi$ as a test function for the weak convergence $\tau g_n \weakto \tau g$ in $L^1(F(1-F) dvdxdt)$, we obtain that
	\begin{align*}
		\int_{[0,T] \times \mathbb{T}^N \times \mathbb{R}^{2N} \times \mathbb{S}^{N-1}}
			g_n \varphi
			\, dtdxd\Lambda
		&= \int_{[0,T] \times \mathbb{T}^N \times \mathbb{R}^{N}}
			\tau g_n \psi F (1-F)
			\, dtdxdv \\
		&\to \int_{[0,T] \times \mathbb{T}^N \times \mathbb{R}^{N}}
			\tau g \psi F (1-F)
			\, dtdxdv \\
		&= \int_{[0,T] \times \mathbb{T}^N \times \mathbb{R}^{2N} \times \mathbb{S}^{N-1}}
			g \varphi
			\, dtdxd\Lambda
	\end{align*}
	which establishes \eqref{eq:limit-maxw-I3-weak-convergence}.

	Next, recall that $\mu_n g_n \to 0$ almost everywhere in $[0,T] \times \Rxv$, which immediately implies that $M_{n,*} \to 1$, $M'_{n} \to 1$ and $M'_{n,*} \to 1$ almost everywhere.
	Furthermore, the definition of these terms guarantees that
	$0 \leq M_{n,*}, M'_n, M'_{n,*} \leq 1/(1-F) \leq 1+e$.

	Therefore, for almost every point $(t,x,v,v_*)$, we have $1 - M_{n,*} M'_n M'_{n,*} \to 0$ alongside the uniform bound
	\[
		| 1 - M_{n,*} M'_n M'_{n,*} | \leq 1 + (1+e)^3.
	\]
	Since the measure $dtdx\Lambda$ is finite, the product limit theorem then yields
	\[
		g_n (1 - M_{n,*} M'_n M'_{n,*}) \weakto 0
		\quad \text{in }
		L^1([0,T] \times \mathbb{T}^N_x \times \mathbb{R}^{2N}_{v,v_*} \times \mathbb{S}^{N-1}_\omega; \, dtdx\Lambda).
	\]

	We now exploit the standard symmetries of the Boltzmann collision operator in order to obtain the convergences of the other terms of $R_n$.
	Recall that the measure $dtdx\Lambda$ is invariant under the exchange of colliding particles $(v, v_*) \mapsto (v_*, v)$ and the pre/post-collisional change of variables $(v, v_*) \mapsto (v', v'_*)$.

	By the definition of weak convergence in $L^1$, this limit holds when integrated against any test function $\varphi \in L^\infty$.
	Applying these transformations to $\varphi$ and using the invariance of the measure transfers the changes of variables directly to the sequence itself.

	Thus, exchanging $v \leftrightarrow v_*$ in the first limit immediately yields
	\[
		g_{n,*} (1 - M_{n} M'_n M'_{n,*}) \weakto 0
		\quad \text{in }
		L^1([0,T] \times \mathbb{T}^N_x \times \mathbb{R}^{2N}_{v,v_*} \times \mathbb{S}^{N-1}_\omega; \, dtdx\Lambda).
	\]
	Similarly, applying the pre/post-collisional change of variables $(v, v_*, v', v'_*) \mapsto (v', v'_*, v, v_*)$ to the initial limit gives $g'_n (1 - M'_{n,*} M_n M_{n,*}) \weakto 0$ and exchanging $v \leftrightarrow v_*$ in this third limit gives $g'_{n,*} (1 - M'_n M_{n,*} M_n) \weakto 0$.
	Combining these four weak limits yields the convergence
	\begin{equation}
		\label{eq:limit-maxw-Rn-convergence}
		R_n \weakto 0
		\quad \text{in }
		L^1([0,T] \times \mathbb{T}^N_x \times \mathbb{R}^{2N}_{v,v_*} \times \mathbb{S}^{N-1}_\omega; \, dtdx\Lambda).
	\end{equation}

	We now show that this convergence implies that $R_n / (N_n Z_n) \weakto 0$.
	From Lemma \ref{lem:tau-gn-compactness}, the sequence $(\tau g_n)_n$ is uniformly bounded in $L^\infty(dt; L^1(F(1-F) dvdx))$.
	By passing to a subsequence, we may assume that $\mu_n g_n(t,x,v) \to 0$ for almost every $(t,x,v) \in [0,T] \times \Rxv$ and that $\mu_n \| \tau g_n(t,x) \|_{L^1(F(1-F) dv)} \to 0$ for almost every $(t,x) \in [0,T] \times \mathbb{T}^N$.

	Furthermore, since $3/4 - F \leq 1 - F$ we have from the expression for $Z_n$ that
	\[
		\| \tau F \|_{L^1}
		\leq Z_n(t,x)
		\leq \| \tau F \|_{L^1}
		+ \frac{1}{3} \mu_n
		\left\| \tau g_n(t,x) \right\|_{L^1(F(1-F) dv)}
	\]
	and therefore $Z_n \to \| \tau F \|_{L^1}$ for almost every $(t,x) \in [0,T] \times \mathbb{T}^N$.
	Given the bounds $1/N_n \leq 1$ and $1/Z_n \leq 1/\| \tau F \|_{L^1}$, alongside the fact that the measure $dtdx\Lambda$ is finite, the product limit theorem yields
	\begin{equation}
		\label{eq:limit-maxw-Nn-Zn-Rn-convergence}
		\frac{1}{N_n Z_n}
		R_n
		\weakto 0
		\quad \text{in }
		L^1([0,T] \times \mathbb{T}^N_x \times \mathbb{R}^{2N}_{v,v_*} \times \mathbb{S}^{N-1}_\omega; \, dtdx\Lambda).
	\end{equation}

	Let $\varphi \in L^\infty([0,T] \times \mathbb{T}^N_x \times \mathbb{R}^{N}_v)$ be an arbitrary test function.
	Testing the above convergence against $(t,x,v,v_*,\omega) \mapsto \varphi(t,x,v)$, using the equilibrium balance identity \eqref{eq:equilibrium-balance-identity} and recognizing the inner integral as $I_3$, we obtain
	\begin{equation*}
		I_3 \weakto 0
		\quad \text{in }
		L^1([0,T] \times \mathbb{T}^N_x \times \mathbb{R}^{N}_v; \, dtdxdv).
	\end{equation*}
	Substituting the above limit, alongside \eqref{eq:limit-maxw-I1-conv} and \eqref{eq:limit-maxw-I2-conv}, into the decomposition \eqref{eq:limit-maxw-L-gn-decomposition}, we obtain that for all $T>0$,
	\[
		\frac{\mathcal{L}(g_n)}{N_n Z_n} \weakto 0
		\text{ in }
		L^1([0,T] \times \Rxv; \, F(1-F) dvdxdt).
	\]

	On the other hand, the strong continuity of $\mathcal{L}$ given by Lemma \ref{lem:L-continuity} implies its weak continuity.
	Therefore, $\mathcal{L}(g_n) \weakto \mathcal{L}(g)$ in $L^1([0,T] \times \Rxv; \, F(1-F) dvdxdt)$.
	Using the same bounds and convergences for $N_n$ and $Z_n$ as those used to establish \eqref{eq:limit-maxw-Nn-Zn-Rn-convergence}, along with the fact that the measure $F(1-F) dvdxdt$ is finite, we can apply product limit theorem.
	This yields
	\[
		\frac{\mathcal{L}(g_n)}{N_n Z_n} \weakto \frac{1}{\| \tau F\|_{L^1}} \mathcal{L}(g)
		\quad \text{in }
		L^1([0,T] \times \Rxv; \, F(1-F) dvdxdt).
	\]
	By uniqueness of the weak limit, we have $\mathcal{L}(g) = 0$, which implies $g$ is a local Fermi-Dirac distribution by Lemma \ref{lem:linearized-properties}.
\end{proof}

\section{Vanishing of conservation defects}
\label{sec:vanishing-conservation-defects}

With the explicit form of the limiting function $g$ established in Lemma \ref{lem:limit-local-maxwellian}, we aim to derive the hydrodynamic equations governing the macroscopic parameters $\rho$, $u$ and $\theta$.
Formally, this is done by multiplying the BFD equation for $f_n$ by a collision invariant $\zeta \in \{1, v_1, \dots, v_N, |v|^2\}$ and integrating over velocity space.
By the standard symmetry properties of the collision integral, the velocity integral of $\zeta Q(f_n)$ on the right-hand side vanishes.
Substituting $f_n$ using the decomposition \eqref{eq:micro-macro-decomposition}, we drop the velocity integrals of $F$, as they are constant in space and time.
The resulting equation is the formal conservation law
\begin{equation}
	\label{eq:formal-local-conserv}
	\partial_t \vmean{\zeta g_n}
	+ \Div_x \vmean{\zeta v g_n} = 0.
\end{equation}

Taking the formal limit $g_n \weakto g$, we obtain the local conservation law for the limiting distribution.
Substituting the explicit form of $g$ and making appropriate choices for $\zeta$ then produces the desired hydrodynamic equations.

As in classical Boltzmann theory, physical a priori estimates are insufficient to justify all the conservation laws \eqref{eq:formal-local-conserv} rigorously.
While these estimates guarantee uniform $L^1$ bounds controlling up to the second moment $|v|^2 f_n$, they fall short of bounding the higher moments required by the flux terms.
Most notably, when $\zeta = |v|^2$, the flux term $f_n |v|^2 v$ cannot be controlled a priori.

It is worth noting that for certain restricted regimes, the available bounds suffice to deduce \eqref{eq:formal-local-conserv} rigorously.
For example, assuming the collision kernel $b$ is integrable (as in Dolbeault \cite{dolbeault-1994}), one can show that the second moment of the collision integral, $(1+|v|^2) Q(f_n)$, is integrable in velocity space for each $n$.
This justifies \eqref{eq:formal-local-conserv} for $\zeta \in \operatorname{span}\{1, v_1, \dots, v_N\}$, which is enough to obtain the isothermal Stokes limit and bypass the technical scaling \eqref{eq:stokes-limit-technical-scaling}.

However, for more general kernels (including hard spheres), one can generally only establish local integrability on $(1+|v|^2) Q(f_n)$.
This weaker property is insufficient to guarantee that the velocity integral of $\zeta Q(f_n)$ vanishes.
Coupled with the lack of a priori control over the higher-moment fluxes, this poses a major obstacle.

To circumvent this, we test the conservation laws using the renormalized formulation of the equations.
This renormalization procedure, however, breaks the symmetries of the collision operator $Q(f_n)$ that would otherwise ensure the velocity integral of $\zeta Q(f_n)$ vanishes.
It is this break of symmetry that generates residual terms on the right-hand side, known as \emph{conservation defects}.

Since we cannot eliminate these defects for finite $n$, our strategy is to prove they vanish in the limit $g_n \weakto g$, thereby recovering the exact conservation laws \eqref{eq:formal-local-conserv} for the limiting fluctuation $g$.

In this section, we establish the main result required to control the principal, unscaled components of these defects, which will later be used in Sections \ref{sec:acoustic-limit-proof} and \ref{sec:stokes-limit}.
First developed in \cite{bgl-acoustic-2000} for the acoustic limit of the classical Boltzmann equation, this technique relies on a decomposition of the collision product and the Young-Fenchel inequality to deduce the necessary bounds.
The main result is as follows.

\begin{lemma}
	\label{lem:conservation-defects}
	Let $\zeta \in \{ 1, v_1, v_2, \dots, v_N, |v|^2 \}$ be a collision invariant.
	We have the bounds
	\[
		\vvmean{
			\zeta \frac{q_n}{N_n}
		}, \,
		\vvmean{\zeta \frac{q_n}{N_n^2}}
		=
		\mathcal{O}
		\left(
			\mu_n
			\sqrt{
				\log\Big(
					\frac{1}{\mu_n}
				\Big)
			}
		\right)
		+ \mathcal{O}\left(
			\mu_n \eta_n \log\frac{1}{\mu_n \eta_n}
		\right)
	\]
	in $L^1_{loc}(dtdx)$
\end{lemma}
\begin{proof}
	We begin by rewriting the conservation defect as
	\begin{equation*}
	    \vvmean{\zeta \frac{q_n}{N_n}}
	    = \vvmean{\zeta \left(
	        1 - \frac{1}{N_{n,*}}
	    \right) \frac{q_n}{N_n}}
	    + \vvmean{\zeta \frac{q_n}{N_n N_{n,*}}}.
	\end{equation*}
	Let us denote the first term by $T_1$ and the second by $T_2$.
	We study each of these terms separately.

	Regarding $T_1$, substituting the definition of the renormalization factor \eqref{eq:3-4-normalization} for the $(1 - 1 / N_{n,*})$ term yields
	\[
	    T_1 = \mu_n
	        \vvmean{
	            \zeta \left(\frac{3}{4} - F_*\right) \frac{g_{n,*}}{N_n N_{n,*}}
	            q_n
	        }.
	\]
	We establish a bound for the term in brackets.
	Since $|\zeta| \leq 4 \tau$, it suffices to bound this expression by replacing $\zeta$ with $\tau$.
	We rewrite this term to isolate the argument of $r$ in \eqref{eq:r-product-of-p-bound}, yielding
	\begin{equation}
	\label{eq:cons-def-factorization-T1}
	\begin{split}
		\tau
		\left(\frac{3}{4} - F_*\right) \frac{g_{n,*}}{N_n N_{n,*}} q_n
		= \bigg(
			\frac{\eta_n \mu_n q_n}{P_n P_{n,*} M'_n M'_{n,*}}
		\bigg)
		\bigg(
			\tau \frac{\eta_n \mu_n}{\alpha}
			\Big(\frac{3}{4} - {}& F_* \Big) \frac{g_{n,*}}{N_n N_{n,*}}
		\bigg) \\
		& \times
		\frac{\alpha}{\eta_n^2 \mu_n^2}
		P_n P_{n,*} M'_n M'_{n,*}.
	\end{split}
	\end{equation}

	For any $\alpha > 0$, the definition of $N_{n,*}$ and the bound $1/N_n \leq 4$ established in Lemma \ref{lem:Nn-Pn-bounds} imply that, for any $n$ large enough such that $\eta_n \leq \alpha/4$, we have
	\begin{equation*}
		\frac{\eta_n \mu_n}{\alpha}
		\left|
			\left(\frac{3}{4} - F_*\right)
			\frac{g_{n,*}}{N_n N_{n,*}}
		\right|
		\leq 4 \frac{\eta_n}{\alpha}
		\leq 1.
	\end{equation*}
	Applying the Young-Fenchel inequality \eqref{eq:young-fenchel} to the two terms in parentheses of \eqref{eq:cons-def-factorization-T1}, the bound above allows us to use the quadratic homogeneous estimate for $r^*$ from Lemma \ref{lem:h-r-properties}~(\ref{item:h*-r*-quadratic}).
	Therefore, for every sufficiently large $n$,
	\begin{equation}
	\label{eq:cons-def-young-fenchel-T1}
	\begin{split}
		\tau \left|
			\left(\frac{3}{4} - F_*\right)
			\frac{g_{n,*}}{N_n N_{n,*}}
			q_n
		\right|
		&\leq \frac{\alpha}{\eta_n^2 \mu_n^2}
		r\left(
			\frac{\eta_n \mu_n q_n}{P_n P_{n,*} M'_n M'_{n,*}}
		\right) P_n P_{n,*} M'_n M'_{n,*} \\
		&\quad + \frac{1}{\alpha} \left(\frac{3}{4} - F_*\right)^2 \frac{g^2_{n,*}}{N_n^2 N_{n,*}^2} P_n P_{n,*} M'_n M'_{n,*}
		r^* \left(\tau\right).
	\end{split}
	\end{equation}

	We estimate the second term on the right-hand side by noting that $3/4 - F \leq 1$ and $M'_n, M'_{n,*} \leq 1+e$.
	Using the bounds on $1/N_n$, $P_n/N_n$, and $P_{n,*}/N_{n,*}$ from Lemma \ref{lem:Nn-Pn-bounds}, together with the estimate for $r^*$ from Lemma \ref{lem:h-r-properties}~(\ref{item:h*-r*-exp-estimates}), we obtain
	\[
		\frac{1}{\alpha} \left(\frac{3}{4} - F_*\right)^2 \frac{g^2_{n,*}}{N_n^2 N_{n,*}^2} P_n P_{n,*} M'_n M'_{n,*}
				r^*(\tau)
		\leq
		\frac{64}{9} \frac{1}{\alpha} \frac{g^2_{n,*}}{N_{n,*}} e^{\tau}.
	\]
	Fix a final time $T>0$.
	Substituting this back into \eqref{eq:cons-def-young-fenchel-T1}, integrating over $t \in [0,T]$ and $(x,v,v_*,\omega)$ with respect to the measure $dt dx\Lambda$ and recalling \eqref{eq:r-product-of-p-bound} yields
	\begin{equation}
		\label{eq:cons-def-integrated-young-fenchel-T1}
		\int_0^T \int_{\mathbb{T}^N}
			\vvmean{
				\left|
				\tau
				\left(\frac{3}{4} - F_*\right)
				\frac{g_{n,*}}{N_n N_{n,*}}
				q_n
				\right|
			}
			\, dx dt
		\leq 4 \alpha C^{in}
		+ \frac{64}{9} \frac{1}{\alpha}
			\int_0^T \int_{\mathbb{T}^N}
				\vvmean{\frac{g^2_{n,*}}{N_{n,*}} e^{\tau}}
			\, dx dt.
	\end{equation}

	To estimate the integral on the second term, we use the collision kernel estimate \eqref{eq:b-v-v*-estimate} alongside the bounds $1-F', 1-F'_* \leq 1$ and $1/(1-F_*) \leq 1+e$.
	Evaluating the integral with respect to $\omega$ then yields
	\begin{equation*}
		\begin{split}
			\int_0^T \int_{\mathbb{T}^N}
				\vvmean{\frac{g^2_{n,*}}{N_{n,*}} e^{\tau}}
			\, dx dt
			\leq C_b (1+e) |\mathbb{S}^{N-1}|
				&\left(
					\int_{\mathbb{R}^N} (1+|v|) e^\tau F \, dv
				\right) \\
			&\quad\times \left(
					\int_0^T \int_{\mathbb{T}^N} \vmean{(1+|v|) \frac{g^2_n}{N_n}} \, dxdt
				\right).
		\end{split}
	\end{equation*}
	Since $F(v) \leq e^{1 - |v|^2/2}$ we have that $e^\tau F \leq e^{5/4} e^{-|v|^2/4}$, ensuring the first integral on the right-hand side is finite.
	For the second integral, since $1+|v| \leq 8 \tau$, Lemma \ref{lem:gn2-Nn-bounds} implies the existence of a sequence $(K_n)_n$ of order $\log(1/\mu_n)$ such that for every $\alpha > 0$ and sufficiently large $n$,
	\begin{equation*}
		\int_0^T \int_{\mathbb{T}^N}
			\vvmean{
				\left|
				\tau
				\left(\frac{3}{4} - F_*\right)
				\frac{g_{n,*}}{N_n N_{n,*}}
				q_n
				\right|
			}
			\, dx dt
		\leq 4 \alpha C^{in}
		+ \frac{1}{\alpha} K_n.
	\end{equation*}

	Finally, we optimize the parameter $\alpha$ by choosing $\alpha_n = \sqrt{K_n}$.
	Since $K_n \to \infty$ and $\eta_n \to 0$, the validity condition $\eta_n \leq \sqrt{K_n} / 4$ is satisfied for all sufficiently large $n$, yielding
	\begin{equation}
		\label{eq:cons-def-estimation-for-T1}
		\int_0^T \int_{\mathbb{T}^N}
			\vvmean{
				\left|
				\tau
				\left(\frac{3}{4} - F_*\right)
				\frac{g_{n,*}}{N_n N_{n,*}}
				q_n
				\right|
			}
			\, dx dt
		\leq (4 C^{in} + 1) \sqrt{K_n}.
	\end{equation}
	Since $T > 0$ is arbitrary, we conclude that
	\begin{equation*}
		T_1 = \mathcal{O}\left(
			\mu_n
			\sqrt{
				\log\Big(
					\frac{1}{\mu_n}
				\Big)
			}
		\right)
		\quad \text{in }
		L^1_{loc}(dt; L^1(dx)).
	\end{equation*}

	We now turn our attention to $T_2$.
	Rewriting the denominator to include the post-collisional terms, we then perform a pre/post-collisional change of variables $(v, v_*) \mapsto (v', v'_*)$.
	Using the anti-symmetry $q'_n = - q_n$ and the collision invariance identity $\zeta' + \zeta'_* = \zeta + \zeta_*$, we obtain that
	\begin{equation}
		\label{eq:cons-def-T2-rewrite}
		\begin{split}
		    T_2
		    &= \frac{1}{2} \vvmean{(\zeta + \zeta_*) \frac{q_n}{N_n N_{n,*}}} \\
		    &= \frac{1}{2} \vvmean{(\zeta + \zeta_*) \frac{N'_n N'_{n,*}}{N_n N_{n,*} N'_n N'_{n,*}} q_n} \\
		    &= - \frac{1}{2} \vvmean{(\zeta + \zeta_*) \frac{N_n N_{n,*}}{N_n N_{n,*} N'_n N'_{n,*}} q_n}.
		\end{split}
	\end{equation}
	For the last equality, we used the collision invariance $\zeta' + \zeta'_* = \zeta + \zeta_*$ and the fact that $q_n$ changes sign under this transformation.
	The renormalization factor \eqref{eq:3-4-normalization} can be written in terms of $P_n$ and $M_n$ as $N_n = \frac{3}{4} P_n + \frac{1}{4} M_n$.
	Therefore, the numerator $N_n N_{n,*}$ can be expanded as
	\begin{equation}
		\label{eq:cons-def-Nn-decomposition1}
		N_n N_{n,*}
		= \frac{9}{16} P_n P_{n,*}
		+ \frac{3}{16} \left[ P_n M_{n,*} + M_n P_{n,*} \right]
		+ \frac{1}{16} M_n M_{n,*}.
	\end{equation}

	We decompose the first term on the right-hand side to reveal the term $P_n P_{n,*} M'_n M'_{n,*}$, which we will later use to obtain the full collision product on the numerator.
	Then, expanding $(1 - M'_n M'_{n,*})$ using the definition of $M_n$, we obtain
	\begin{align*}
		P_n P_{n,*} &= P_n P_{n,*} M'_n M'_{n,*} + P_n P_{n,*} (1 - M'_n M'_{n,*}) \\
		&= P_n P_{n,*} M'_n M'_{n,*}
		+ \mu_n P_n P_{n,*} \left(  F' g'_n +  F'_* g'_{n,*} \right) \\
		&\quad - \mu_n^2 P_n P_{n,*} \left( F' F'_* g'_n g'_{n,*} \right).
	\end{align*}
	For the second term on the right-hand side of \eqref{eq:cons-def-Nn-decomposition1}, multiplying the relation $P_n - M_n = \mu_n g_n$ by $P_{n,*} - M_{n,*} = \mu_n g_{n,*}$ and rearranging yields
	\[
		P_n M_{n,*} + M_n P_{n,*}
		= P_n P_{n,*}
		+ M_n M_{n,*}
		- \mu_n^2 g_n g_{n,*}
	\]
	Finally, substituting these relations back into \eqref{eq:cons-def-Nn-decomposition1}, expanding the last term using the definition of $M_n$ and grouping the terms by powers of $\mu_n$, we obtain
	\begin{equation}
		\label{eq:cons-def-Nn-decomposition2}
			N_n N_{n,*}
			= \frac{1}{4}
			+ \frac{3}{4} A_0
			+ \frac{1}{4} \mu_n A_1
			- \frac{1}{16} \mu_n^2 A_2
	\end{equation}
	where the terms $A_k$ are defined as
	\begin{align*}
		A_0 &= P_n P_{n,*} M'_n M'_{n,*}, \\
		A_1 &= 3 P_n P_{n,*} (F' g'_n + F'_* g'_{n,*})
			-  (F g_n + F_* g_{n,*}), \\
		A_2 &= 12 P_n P_{n,*} F' F'_* g'_n g'_{n,*}
			+ (3 - 4 F F_*) g_n g_{n,*}.
	\end{align*}

	Substituting \eqref{eq:cons-def-Nn-decomposition2} into \eqref{eq:cons-def-T2-rewrite} and expanding by linearity, we note that the integral corresponding to the constant term $1/4$ vanishes.
	Indeed, by a pre/post-collisional change of variables (with $\zeta' + \zeta'_* = \zeta + \zeta_*$ and $q'_n = -q_n$), this integral, which is a constant multiple of $\vvmean{ (\zeta + \zeta_*) (N_n N_{n,*} N'_n N'_{n,*})^{-1} q_n }$, equals minus itself.
	Therefore, the term $T_2$ can be rewritten as
	\[
		T_2 = - \frac{3}{8} I_0
		- \frac{1}{8} \mu_n I_1
		+ \frac{1}{32} \mu_n^2 I_2,
	\]
	where the integrals $I_k$ are given by
	\[
		I_k =
			\vvmean{
				(\zeta + \zeta_*)
				\frac{
					A_k
				}{
					N_n N_{n,*} N'_n N'_{n,*}
				}
				q_n
			}
		\quad \text{for } k \in \{0,1,2\}.
	\]

	To estimate $I_1$, we first split the integral corresponding to $3 P_n P_{n,*} (F' g'_n + F'_* g'_{n,*})$ from the integral corresponding to $(F g_n + F_* g_{n,*})$.
	Applying the pre/post-collisional change of variables $(v, v_*) \mapsto (v', v'_*)$ to the first one, we recombine the resulting two terms to find that
	\[
		I_1 = - \vvmean{
			(\zeta + \zeta_*)
			\frac{
				3 P'_n P'_{n,*} (F g_n + F_* g_{n,*})
				+ ( F g_n + F_* g_{n,*})
			}{
				N_n N_{n,*} N'_n N'_{n,*}
			}
			q_n
		}.
	\]
	Grouping terms in the numerator gives $(3 P'_n P'_{n,*} + 1) F g_n + (3 P'_n P'_{n,*} + 1) F_* g_{n,*}$.
	Separating the integrals once more and applying the symmetry $v \leftrightarrow v_*$ to the second term produces two identical terms, which yields
	\[
		I_1 = - 2 \vvmean{
			(\zeta + \zeta_*)
			\frac{
				(3 P'_n P'_{n,*} + 1) F g_n
			}{
				N_n N_{n,*} N'_n N'_{n,*}
			}
			q_n
		}.
	\]

	To bound this expression, we again use the upper bound $|\zeta + \zeta_*| \leq 4 (\tau + \tau_*)$.
	Regrouping the terms to isolate the argument of $r$ in \eqref{eq:r-product-of-p-bound}, an application the Young-Fenchel inequality \eqref{eq:young-fenchel} yields that for every $n \in \mathbb{N}$,
	\begin{equation}
	\label{eq:cons-def-young-fenchel-I1}
	\begin{split}
		&(\tau + \tau_*) \left|
			\frac{(3 P'_n P'_{n,*} + 1) F g_n}{N_n N_{n,*} N'_n N'_{n,*}} q_n
		\right| \\
		&\quad\leq \frac{1}{\eta_n^2 \mu_n^2}
			r\left(
				\frac{\eta_n \mu_n q_n}{P_n P_{n,*} M'_n M'_{n,*}}
			\right) P_n P_{n,*} M'_n M'_{n,*}\\
		&\qquad+ \frac{1}{\eta^2_n \mu^2_n}
			r^*\left(
				(\tau + \tau_*)
				\left[
					\eta_n \mu_n
					\left|
						\frac{(3 P'_n P'_{n,*} + 1) F g_n}{N_n N_{n,*} N'_n N'_{n,*}}
					\right|
				\right]
			\right) P_n P_{n,*} M'_n M'_{n,*}.
	\end{split}
	\end{equation}

	We aim to estimate the second term on the right-hand side using Lemma \ref{lem:h-r-properties} (\ref{item:h*-r*-quadratic}).
	For the term in square brackets observe that, using the bounds for $P_n / N_n$, $1 / N_n$ and $\mu_n g_n/N_n$ established in Lemma \ref{lem:Nn-Pn-bounds}, we find
	\begin{equation*}
	\begin{split}
	    \eta_n \mu_n
	    \bigg|
	        \frac{(3 P'_n P'_{n,*} + 1) F g_n}{N_n N_{n,*} N'_n N'_{n,*}}
	    \bigg|
	    &= \eta_n F \cdot
	       \bigg| \frac{\mu_n g_n}{N_n} \bigg| \cdot
	       \bigg( \frac{3 P'_n P'_{n,*} + 1}{N'_n N'_{n,*}} \bigg) \cdot
	       \frac{1}{N_{n,*}} \\
	    &\leq \frac{256  M_{\mathrm{ren}}}{3} \eta_n.
	\end{split}
	\end{equation*}
	For $n$ large enough such that $\eta_n \leq 3/(256 M_{\mathrm{ren}})$, we apply Lemma \ref{lem:h-r-properties} (\ref{item:h*-r*-quadratic}) alongside the exponential bound from item (\ref{item:h*-r*-exp-estimates}) of the same lemma.
	Using that $F e^{\tau} \leq 1$ and $M'_n, M'_{n,*} \leq 1+e$, along with the bounds for $P_n / N_n$ and $1 / N_n$ from Lemma \ref{lem:Nn-Pn-bounds}, we deduce that there exists a constant $C_1 > 0$ such that
	\begin{equation*}
	\begin{split}
		&\frac{1}{\eta^2_n \mu^2_n}
			r^*\left(
				(\tau + \tau_*)
				\left[
					\eta_n \mu_n
					\left|
						\frac{(3 P'_n P'_{n,*} + 1) F g_n}{N_n N_{n,*} N'_n N'_{n,*}}
					\right|
				\right]
			\right) P_n P_{n,*} M'_n M'_{n,*} \\
		&\leq
			\frac{g^2_n}{N_n}
			\left[
				\frac{3 P'_n P'_{n,*} + 1}{N'_n N'_{n,*}}
			\right]^2
			\frac{P_n P'_n}{N_n N'_n}
			\frac{M'_n M'_{n,*}}{N_{n,*}}
			F^2 e^{\tau}
			e^{\tau_*} \\
		&\leq C_1 F \frac{g^2_n}{N_n} e^{\tau_*}.
	\end{split}
	\end{equation*}

	Fix a final time $T>0$.
	We substitute this inequality back into \eqref{eq:cons-def-young-fenchel-I1} and integrate over $t \in [0,T]$ and $(x,v,v_*,\omega)$ with respect to the measure $dt dx \Lambda$.
	Applying \eqref{eq:r-product-of-p-bound} shows that the first term is bounded by the initial constant $4 C^{in}$, yielding
	\[
		\int_0^T \int_{\mathbb{T}^N}
			\vvmean{
			(\tau + \tau_*)
				\left|
					\frac{
						(3 P'_n P'_{n,*} + 1) F g_n
					}{
						N_n N_{n,*} N'_n N'_{n,*}
					}
					q_n
				\right|
			} \, dx dt
		\leq 4 C^{in}
		+ C_1
			\int_0^T \int_{\mathbb{T}^N}
			\vvmean{
				F \frac{g_n^2}{N_n} e^{\tau_*}
			} \, dx dt.
	\]

	Let us estimate the integral in the second term.
	Applying the collision kernel bound \eqref{eq:b-v-v*-estimate} alongside the inequalities $1-F', 1-F'_* \leq 1$, we evaluate the constant integral over $\omega$.
	Then, separating the $v_*$ integral and using the facts that $(1+|v|) F \leq 2$ and $(1-F)^{-1} \leq 1+e$, we obtain
	\[
		\begin{split}
			\int_0^T \int_{\mathbb{T}^N}
				\vvmean{
					F \frac{g_n^2}{N_n} e^{\tau_*}
				} \, dx dt
			&\leq C_b |\mathbb{S}^{N-1}|
				\int_0^T \iiint_{\mathbb{T}^N \times \mathbb{R}^{2N}}
					(1+|v|)(1+|v_*|) F^2 \frac{g_n^2}{N_n} F_* e^{\tau_*} \, dxdvdv_* dt \\
			&\leq 2 (1+e) C_b |\mathbb{S}^{N-1}|
			\left(
				\int_{\mathbb{R}^N} (1 + |v_*|) F_* e^{\tau_*} \, dv_*
			\right)
			\int_0^T \int_{\mathbb{T}^N} \vmean{\frac{g_n^2}{N_n}} \, dxdt.
		\end{split}
	\]
	The first integral on the right-hand side is finite because $e^{\tau_*} F_* \leq e^{5/4} e^{-|v_*|^2/4}$.
	For the second integral, Lemma \ref{lem:gn2-Nn-bounds} ensures it is of order $\mathcal{O}(1)$.
	Given that the time $T>0$ is arbitrary, we conclude
	\begin{equation*}
		I_1 = \mathcal{O}(1)
		\quad \text{in }
		L^1_{loc}(dt; L^1(dx)).
	\end{equation*}

	Next, we estimate the term $I_2$.
	We first separate the integrals according to the sum in the numerator, using that $\zeta + \zeta_* = \zeta' + \zeta'_*$ on the first term and applying the pre/post-collisional change of variables $(v, v_*) \mapsto (v', v'_*)$ (with $q'_n = -q_n$) to the second.
	Merging these terms, we obtain a single integral involving $g'_n g'_{n,*}$.
	Finally, separating the integrals for $\zeta'$ and $\zeta'_*$ and swapping the variables $v \leftrightarrow v_*$ in the last one, we obtain
	\begin{equation*}
	\begin{split}
		I_2 &= - \vvmean{
			(\zeta' + \zeta'_*)
			\frac{
				(3 - 4 F' F'_* - 12 P_n P_{n,*} F' F'_*)
				g'_n g'_{n,*}
			}{
				N_n N_{n,*} N'_n N'_{n,*}
			}
			q_n
		} \\
		&= - 2 \vvmean{
			\zeta'
			\frac{
				(3 - 4 F' F'_* - 12 P_n P_{n,*} F' F'_*)
				g'_n g'_{n,*}
			}{
				N_n N_{n,*} N'_n N'_{n,*}
			}
			q_n
		}.
	\end{split}
	\end{equation*}

	The term in parentheses is bounded.
	Indeed, since $0 \leq F', F'_* \leq 1$, the term $4 F' F'_*$ clearly lies in $[0,4]$.
	Rewriting the term $P_n P_{n,*} F' F'_*$ using the equilibrium balance identity \eqref{eq:equilibrium-balance-identity}, we obtain a fraction whose numerator is bounded by $1$.
	Since $(1-F)^{-1}, (1-F_*)^{-1} \leq 1 + e$, we find
	\begin{equation*}
	    P_n P_{n,*} F' F'_*
	    = \frac{f_n f_{n,*} (1-F') (1-F'_*)}{(1-F) (1-F_*)}
	    \leq (1 + e)^2.
	\end{equation*}

	Let us consider the term $\mu_n I_2$.
	Recalling the bound for $\mu_n g_n/N_n$ established in Lemma \ref{lem:Nn-Pn-bounds} and the fact that $|\zeta'| \leq 4 \tau'$, there exists a constant $C_2 > 0$ such that
	\begin{equation}
		\label{eq:cons-def-estimate-I2}
		|\mu_n I_2| \leq C_2 \vvmean{
			\left|
				\tau'
				\frac{
					g'_{n,*}
				}{
					N_n N_{n,*} N'_{n,*}
				}
				q_n
			\right|
		}.
	\end{equation}

	We estimate this integrand using a factorization analogous to \eqref{eq:cons-def-factorization-T1}.
	Applying the Young-Fenchel inequality \eqref{eq:young-fenchel} then yields, for every $n \in \mathbb{N}$,
	\begin{equation*}
	\begin{split}
		\left|
			\tau'
			\frac{g'_{n,*}}{N_n N_{n,*} N'_{n,*}} q_n
		\right|
		&\leq \frac{\alpha}{\eta_n^2 \mu_n^2}
		r\left(
			\frac{\eta_n \mu_n q_n}{P_n P_{n,*} M'_n M'_{n,*}}
		\right) P_n P_{n,*} M'_n M'_{n,*}\\
		&\quad + \frac{\alpha}{\eta_n^2 \mu_n^2}
		r^*\left(
			\tau'
			\left[
				\frac{\eta_n \mu_n}{\alpha}
				\left|
					\frac{g'_{n,*}}{N_n N_{n,*} N'_{n,*}}
				\right|
			\right]
		\right) P_n P_{n,*} M'_n M'_{n,*}.
	\end{split}
	\end{equation*}
	Applying the bounds from Lemma \ref{lem:Nn-Pn-bounds} we can bound the term in square brackets above as
	\[
		\frac{\eta_n \mu_n}{\alpha}
		\left|
			\frac{g'_{n,*}}{N_n N_{n,*} N'_{n,*}}
		\right|
		= \frac{\eta_n}{\alpha}
		\left|
			\frac{\mu_n g'_{n,*}}{N'_{n,*}}
		\right|
		\frac{1}{N_n N_{n,*}}
		\leq \frac{16 M_{\mathrm{ren}}}{\alpha} \eta_n.
	\]
	Therefore, for $n$ large enough such that $\eta_n \leq \alpha/(16 M_{\mathrm{ren}})$, we apply the bound for $r^*$ in Lemma \ref{lem:h-r-properties} (\ref{item:h*-r*-quadratic}).
	Combining this with the estimates in Lemma \ref{lem:Nn-Pn-bounds}, we find a constant $C_3 > 0$ such that
	\begin{equation*}
	\begin{split}
		&\frac{\alpha}{\eta_n^2 \mu_n^2}
		r^*\left(
			\tau'
			\left[
				\frac{\eta_n \mu_n}{\alpha}
				\left|
					\frac{g'_{n,*}}{N_n N_{n,*} N'_{n,*}}
				\right|
			\right]
		\right) P_n P_{n,*} M'_n M'_{n,*} \\
		&\leq \frac{1}{\alpha}
		\frac{g'_{n,*}}{N'_{n,*}}
		\frac{1}{N_n N_{n,*} N'_{n,*}}
		\frac{P_n P_{n,*}}{N_n N_{n,*}}
		M'_n M'_{n,*} e^{\tau'}\\
		&\leq \frac{C_3}{\alpha} \frac{(g'_{n,*})^2}{N'_{n,*}} e^{\tau'}.
	\end{split}
	\end{equation*}

	Fix a final time $T>0$.
	Substituting these estimates into \eqref{eq:cons-def-estimate-I2}, integrating against the measure $dt dx \Lambda$, and applying \eqref{eq:r-product-of-p-bound} to the first term, we find that for every $n$ large enough such that $\eta_n \leq \alpha/(16 M_{\mathrm{ren}})$,
	\[
		\int_0^T \int_{\mathbb{T}^N} |\mu_n I_2| \, dxdt
		\leq 4 \alpha C_2 C^{in}
		+ \frac{C_2 C_3}{\alpha}
			\int_0^T \int_{\mathbb{T}^N}
				\vvmean{
					\frac{(g'_{n,*})^2}{N'_{n,*}} e^{\tau'}
				}
			\, dxdt.
	\]

	To estimate the integral in the second term, we perform a pre/post-collisional change of variables.
	Using the collision kernel bound \eqref{eq:b-v-v*-estimate}, combined with the estimates $(1-F_*)^{-1} \leq 1+e$ and $(1+|v_*|) \leq 8 \tau_*$, we obtain
	\[
		\int_0^T \int_{\mathbb{T}^N}
			\vvmean{
				\frac{(g'_{n,*})^2}{N'_{n,*}} e^{\tau'}
			}
		\, dxdt
		\leq 8 C_b |\mathbb{S}^{N-1}| (1+e)
			\left(
				\int (1+|v|) F e^\tau \, dv
			\right)
			\int_0^T \int_{\mathbb{T}^N} \vmean{\tau \frac{g_n^2}{N_n}} \, dxdt,
	\]
	The first integral on the right-hand side is finite since $e^\tau F \leq e^{5/4} e^{-|v|^2/4}$, while the second integral is $\mathcal{O}(\log(1/\mu_n))$ from Lemma \ref{lem:gn2-Nn-bounds}.
	Substituting this estimate into the bound for $\mu_n I_2$ and optimizing $\alpha$ as was done for the $T_1$ term, we conclude that
	\begin{equation*}
		\mu_n I_2
		= \mathcal{O}\left(
			\sqrt{\log\left(\frac{1}{\mu_n}\right)}
		\right)
		\quad \text{in }
		L^1_{loc}(dt; L^1(dx)).
	\end{equation*}

	Finally, we turn to the estimation of $I_0$.
	Symmetrizing this integral by a pre/post-collisional change of variables (with $\zeta' + \zeta'_* = \zeta + \zeta_*$ and $q'_n = -q_n$), we can rewrite the resulting numerator using \eqref{eq:scaled-collision-product}, which yields $A_0 - A'_0 = -\mu_n \eta_n q_n$.
	Therefore,
	\begin{equation}
		\label{eq:cons-def-new-exp-I0}
		I_0
		= - \frac{1}{2} \mu_n \eta_n
		\vvmean{
			(\zeta + \zeta_*) \frac{q^2_n}{N_n N_{n,*} N'_n N'_{n,*}}
		}.
	\end{equation}
	We estimate this quantity using the $s(z)$ function \eqref{eq:z-function}, explored in the proof of Lemma \ref{lem:gn2-Nn-bounds}.

	Unpacking the definition of $s$, we expand the expression \eqref{eq:scaled-collision-product} for the term $\mu_n \eta_n q_n$ in the denominator.
	Using the fact that $P'_n P'_{n,*} M_n M_{n,*} \geq 0$, we obtain
	\begin{equation*}
	\begin{split}
		\frac{1}{\mu^2_n \eta^2_n}
		s\left(
			\frac{\mu_n \eta_n q_n}{P_n P_{n,*} M'_n M'_{n,*} }
		\right)
		P_n P_{n,*} M'_n M'_{n,*}
		&= \frac{3}{2} \frac{q^2_n}{P'_n P'_{n,*} M_n M_{n,*} + 2 P_n P_{n,*} M'_n M'_{n,*}} \\
		&\geq \frac{3}{4} \frac{q^2_n}{P'_n P'_{n,*} M_n M_{n,*} + P_n P_{n,*} M'_n M'_{n,*}}.
	\end{split}
	\end{equation*}
	Moreover, observe that from the bounds for $P_n / N_n$ established in Lemma \ref{lem:Nn-Pn-bounds} together with the fact that $M_n \leq 1+e$, we find
	\begin{equation*}
	\begin{split}
		\frac{P'_n P'_{n,*} M_n M_{n,*} + P_n P_{n,*} M'_n M'_{n,*}}{N_n N_{n,*} + N'_n N'_{n,*}}
		&= \frac{P'_n P'_{n,*}}{N'_n N'_{n,*}}
			M_n M_{n,*}
			+ \frac{P_n P_{n,*}}{N_n N_{n,*}}
			M'_n M'_{n,*} \\
		&\leq \frac{32}{9} (1+e)^2.
	\end{split}
	\end{equation*}

	We multiply the preceding lower bound for $s$ by $(\tau + \tau_*) (N_n N_{n,*})^{-1}$ and we integrate over $(v,v_*,\omega)$ against the measure $\Lambda$.
	Next, by performing a pre/post-collisional change of variables (with $q'_n = -q_n$) and using the identity $\tau' + \tau'_* = \tau + \tau_*$, we symmetrize the term $(N_n N_{n,*})^{-1}$ in the integral.
	Finally, applying the reciprocal of the upper bound derived just above, we obtain
	\begin{equation}
	\label{eq:cons-def-qn-square-adaptation}
	\begin{split}
		&\vvmean{
			(\tau + \tau_*)
			\frac{1}{\mu^2_n \eta^2_n}
				s\left(
					\frac{\mu_n \eta_n q_n}{P_n P_{n,*} M'_n M'_{n,*} }
				\right)
			\frac{P_n P_{n,*} M'_n M'_{n,*}}{N_n N_{n,*}}
		} \\
		&\geq \frac{3}{4} \vvmean{
			(\tau + \tau_*) \frac{q^2_n}{P_n P_{n,*} M'_n M'_{n,*} + P'_n P'_{n,*} M_n M_{n,*}}
			\frac{1}{N_n N_{n,*}}
		}\\
		&= \frac{3}{8}
		\vvmean{
			(\tau + \tau_*) \frac{q^2_n}{P_n P_{n,*} M'_n M'_{n,*} + P'_n P'_{n,*} M_n M_{n,*}}
			\left(
				\frac{1}{N_n N_{n,*}} + \frac{1}{N'_n N'_{n,*}}
			\right)
		}\\
		&\geq
		\frac{27}{256(1+e)^2}
		\vvmean{
			(\tau + \tau_*) \frac{q^2_n}{N_n N_{n,*} N'_n N'_{n,*}}
		}.
	\end{split}
	\end{equation}

	Following an approach similar to the proof of Lemma \ref{lem:gn2-Nn-bounds}, we estimate the term on the left-hand side using the inequality \eqref{eq:young-inequality-h-s}.
	For each $n \in \mathbb{N}$ and $v,v_* \in \mathbb{R}^N$, let $w = w_n(\tau + \tau_*) > 0$ be the unique solution of
	\[
		\frac{h'(w_n(\tau + \tau_*))}{s'(w_n(\tau + \tau_*))}
		= 1 + \log\left(
			1 + \mu_n \eta_n \exp\left(\frac{1}{8}(\tau + \tau_*)\right)
		\right).
	\]
	Substituting $z = (\mu_n \eta_n q_n)/(P_n P_{n,*} M'_n M'_{n,*})$ and $w = w_n(\tau + \tau_*)$ into inequality \eqref{eq:young-inequality-h-s} and multiplying the whole expression by $(P_n P_{n,*} M'_n M'_{n,*})/(\mu_n^2 \eta_n^2 N_n N_{n,*})$ to obtain
	\begin{equation}
	\label{eq:cons-def-h'-s'-s-inequality}
	\begin{split}
		&\frac{h'(w_n(\tau + \tau_*))}{s'(w_n(\tau + \tau_*))}
		\frac{1}{\mu_n^2 \eta_n^2}
		s\left(
			\frac{\mu_n \eta_n q_n}{P_n P_{n,*} M'_n M'_{n,*}}
		\right) \frac{P_n P_{n,*} M'_n M'_{n,*}}{N_n N_{n,*}} \\
		&\leq \frac{2}{\mu_n^2 \eta_n^2}
		h\left(
			\frac{\mu_n \eta_n q_n}{P_n P_{n,*} M'_n M'_{n,*}}
		\right) \frac{P_n P_{n,*} M'_n M'_{n,*}}{N_n N_{n,*}} \\
		&\quad + \frac{2}{\mu_n^2 \eta_n^2}
		\left[
			\frac{h'(w_n(\tau + \tau_*))}{s'(w_n(\tau + \tau_*))} s(w_n(\tau + \tau_*)) - h(w_n(\tau + \tau_*))
		\right] \frac{P_n P_{n,*} M'_n M'_{n,*}}{N_n N_{n,*}}.
	\end{split}
	\end{equation}

	We estimate the terms on both sides separately, beginning with the right-hand side.
	For the second term in \eqref{eq:cons-def-h'-s'-s-inequality}, arguing as in the proof of Lemma \ref{lem:gn2-Nn-bounds}, we find that there exists a constant $C_4 > 0$ such that
	\[
		\frac{h'(w_n(\tau + \tau_*))}
		{s'(w_n(\tau + \tau_*))} s(w_n(\tau + \tau_*)) - h(w_n(\tau + \tau_*))
		\leq C_4 \mu_n^2 \eta_n^2 \exp\left(\frac{7}{16} (\tau + \tau_*)\right).
	\]
	Meanwhile, bounding the factor containing $P_n P_{n,*} M'_n M'_{n,*}$ using the $P_n / N_n$ estimate from Lemma \ref{lem:Nn-Pn-bounds} and the fact that $M_n \leq 1+e$ yields
	\[
		\frac{P_n P_{n,*}}{N_n N_{n,*}} M'_n M'_{n,*}
		\leq \frac{16}{9} (1+e)^2.
	\]

	For the first term on the right-hand side, using the inequality $\log(1+z) \leq z$ for $z > -1$, we deduce that $h(z) = r(z) + \log(1+z) - z \leq r(z)$.
	We also bound the factor $1/(N_n N_{n,*})$ by applying Lemma \ref{lem:Nn-Pn-bounds}.
	Turning to the left-hand side, analogously to the derivation of \eqref{eq:zeta_n-h'-s'-from-bgl} in Lemma \ref{lem:gn2-Nn-bounds}, we have
	\begin{equation*}
	\begin{split}
		\frac{h'(w_n(\tau + \tau_*))}{s'(w_n(\tau + \tau_*))}
		\geq \zeta_n \frac{1}{8} (\tau + \tau_*),
	\end{split}
	\end{equation*}
	where $\zeta_n$ is a sequence satisfying $\zeta_n^{-1} = \mathcal{O}(\log(1/(\mu_n \eta_n)))$.

	Combining these estimates, there exists a constant $C_5 > 0$ such that
	\begin{equation*}
	\begin{split}
		\zeta_n \, (\tau + \tau_*)
		&\frac{1}{\mu^2_n \eta^2_n}
			s\left(
				\frac{\mu_n \eta_n q_n}{P_n P_{n,*} M'_n M'_{n,*} }
			\right)
		\frac{P_n P_{n,*} M'_n M'_{n,*}}{N_n N_{n,*}}\\
		&\quad\leq \frac{C_5}{\mu_n^2 \eta_n^2}
		r\left(
			\frac{\mu_n \eta_n q_n}{P_n P_{n,*} M'_n M'_{n,*}}
		\right) P_n P_{n,*} M'_n M'_{n,*}
		+ C_5 \exp\left(\frac{7}{16} (\tau + \tau_*)\right).
	\end{split}
	\end{equation*}

	Fix a final time $T > 0$.
	Integrating both sides over $t\in [0,T]$ and $(x,v,v_*,\omega)$ with respect to the measure $dt dx \Lambda$, we bound the first term on the right-hand side using \eqref{eq:r-product-of-p-bound}.
	Next, we estimate the second term using \eqref{eq:b-v-v*-estimate} and the trivial bounds $1-F', 1-F'_* \leq 1$, yielding
	\begin{equation}
	\label{eq:cons-def-zetan-tau-Nn}
	\begin{split}
		\zeta_n
		\int_0^T \int_{\mathbb{T}^N}
		&\vvmean{
			(\tau + \tau_*)
			\frac{1}{\mu^2_n \eta^2_n}
				s\left(
					\frac{\mu_n \eta_n q_n}{P_n P_{n,*} M'_n M'_{n,*} }
				\right)
			\frac{P_n P_{n,*} M'_n M'_{n,*}}{N_n N_{n,*}}
		} \, dx dt \\
		&\quad \leq 4 C_5 C^{in}
		+ T C_5 C_b |\mathbb{S}^{N-1}|
		\left(
			\int_{\mathbb{R}^N}
			(1+|v|) F \exp\left(\frac{7}{16} \tau\right)
		\right)^2.
	\end{split}
	\end{equation}
	The right-hand side is bounded since $F \leq e^{3/2} e^{-2\tau}$.

	Substituting this result into \eqref{eq:cons-def-qn-square-adaptation}, recalling the expression \eqref{eq:cons-def-new-exp-I0}, and noting that $|\zeta + \zeta_*| \leq 4(\tau + \tau_*)$, we conclude that
	\begin{equation*}
		I_0 =
			\mathcal{O}\left(
				\mu_n \eta_n \log\left(
					\frac{1}{\mu_n \eta_n}
				\right)
			\right)
			\quad \text{in }
			L^1_{loc}(dt; L^1(dx)).
	\end{equation*}
	Combining this estimate with the bounds for $I_1$ and $I_2$, along with the expression for $T_2$ and the estimate for $T_1$, we obtain the desired result.

	Finally, for the second type of conservation defects, we decompose
	\begin{equation*}
		\vvmean{\zeta \frac{q_n}{N_n^2}}
		= \vvmean{\zeta \left(1 - \frac{1}{N^2_{n,*}}\right) \frac{q_n}{N^2_n}}
		+ \vvmean{\zeta \frac{q_n}{N^2_n N^2_{n,*}}}.
	\end{equation*}
	By comparison with the first case, we denote the first term by $\widetilde{T}_1$ and the second by $\widetilde{T}_2$.

	Factoring the difference of squares, we rewrite $\widetilde{T}_1$ as
	\begin{equation*}
		\widetilde{T}_1
		= \vvmean{\zeta \left(1 - \frac{1}{N_{n,*}}\right) \frac{q_n}{N_n} \left(\frac{1}{N_n} + \frac{1}{N_n N_{n,*}}\right)}.
	\end{equation*}
	We expand the first parenthesis using the definition of $N_{n,*}$.
	Then, applying the bounds for $1/N_n$ from Lemma \ref{lem:Nn-Pn-bounds}, as well as the fact that $|\zeta| \leq 4\tau$, we have that
	\begin{equation*}
	\begin{split}
		\widetilde{T}_1
		&= \mu_n
			\vvmean{
				\zeta \left(\frac{3}{4} - F_*\right) \frac{g_{n,*}}{N_n N_{n,*}}
				q_n
				\left(\frac{1}{N_n} + \frac{1}{N_n N_{n,*}}\right)
			} \\
		&\leq 80\mu_n
			\vvmean{
			\left|
				\tau \left(\frac{3}{4} - F_*\right) \frac{g_{n,*}}{N_n N_{n,*}}
				q_n
			\right|
			}
	\end{split}
	\end{equation*}
	The right-hand side is, up to multiplication by a constant, identical to the quantity bounded in \eqref{eq:cons-def-estimation-for-T1} during our estimation of $T_1$.
	Therefore, we can immediately conclude that
	\begin{equation*}
		\widetilde{T}_1 = \mathcal{O}\left(
			\mu_n
			\sqrt{
				\log\Big(
					\frac{1}{\mu_n}
				\Big)
			}
		\right)
		\quad \text{in }
		L^1_{loc}(dt; L^1(dx)).
	\end{equation*}

	For $\widetilde{T}_2$, we exchange $v$ and $v_*$ to extract the conserved quantity $(\zeta + \zeta_*)$ as a factor.
	Performing a pre/post-collisional change of variables (with $\zeta' + \zeta'_* = \zeta + \zeta_*$ and $q'_n = -q_n$) and factoring the resulting difference of squares yields
	\begin{equation*}
	\begin{split}
		\widetilde{T}_2
		&= \frac{1}{2} \vvmean{(\zeta + \zeta_*) \frac{q_n}{N^2_n N^2_{n,*}}}\\
		&= \frac{1}{4} \vvmean{
			(\zeta + \zeta_*) q_n
			\left(
				\frac{1}{N_n N_{n,*}}
				- \frac{1}{N'_n N'_{n,*}}
			\right)
			\left(
				\frac{1}{N_n N_{n,*}}
				+ \frac{1}{N'_n N'_{n,*}}
			\right)
		}.
	\end{split}
	\end{equation*}

	Distributing the difference term allows us to split the expression into two integrals.
	Applying the same pre/post-collisional change of variables to the first integral demonstrates that it is equal to the second.
	Combining these two integrals and rewriting the fraction to produce a symmetric denominator, we obtain
	\begin{equation*}
		\widetilde{T}_2
		= - \frac{1}{2}
		\vvmean{
			(\zeta + \zeta_*)
			\frac{N_n N_{n,*}}{ N_n N_{n,*} N'_n N'_{n,*}}
			\tilde{q}_n
		},
	\end{equation*}
	where we have defined $\tilde{q}_n = q_n \left( (N_n N_{n,*})^{-1} + (N'_n N'_{n,*})^{-1} \right)$.

	As before, substituting the decomposition \eqref{eq:cons-def-Nn-decomposition2} for the numerator $N_n N_{n,*}$ splits $\widetilde{T}_2$ into the sum

	\[
		\widetilde{T}_2 = - \frac{3}{8} \tilde{I}_0
		- \frac{1}{8} \mu_n \tilde{I}_1
		+ \frac{1}{32} \mu_n^2 \tilde{I}_2,
	\]
	where each $\tilde{I}_k$ is defined exactly as $I_k$, but with $\tilde{q}_n$ replacing $q_n$.
	Because $\tilde{q}_n$ inherits the collision symmetries of $q_n$ (namely, invariance under the exchange $v \leftrightarrow v_*$ and under pre/post-collisional changes of variables), applying the same symmetrization techniques as before yields
	\begin{align*}
		\tilde{I}_0 &= - \frac{1}{2} \mu_n \eta_n
			\vvmean{
				(\zeta + \zeta_*)
				\frac{
					q_n \tilde{q}_n
				}{
					N_n N_{n,*} N'_n N'_{n,*}
				}
			}, \\
		\tilde{I}_1 &= - 2 \vvmean{
			(\zeta + \zeta_*)
				\frac{
					(3 P'_n P'_{n,*} + 1) F g_n
				}{
					N_n N_{n,*} N'_n N'_{n,*}
				}
				\tilde{q}_n
			}, \\
		\tilde{I}_2 &= 2 \vvmean{
			\zeta'
			\frac{
				(3 - 4 F F_* - 12 P_n P_{n,*} F' F'_*)
				g'_n g'_{n,*}
			}{
				N_n N_{n,*} N'_n N'_{n,*}
			}
			\tilde{q}_n
		}.
	\end{align*}

	To conclude, we bound $\tilde{I}_0$, $\tilde{I}_1$ and $\tilde{I}_2$ to show they share the exact same asymptotic order as their counterparts $I_0$, $I_1$ and $I_2$.
	For $\tilde{I}_0$, observe that the bound for $1/N_n$ in Lemma \ref{lem:Nn-Pn-bounds} yields $\tilde{q}_n \leq 32 q_n$.
	Combining this with the fact that $|\zeta + \zeta_*| \leq 4(\tau + \tau_*)$, we obtain
	\[
		|\tilde{I}_0| \leq 64 \mu_n \eta_n
		\vvmean{
			(\tau + \tau_*)
			\frac{
				q_n^2
			}{
				N_n N_{n,*} N'_n N'_{n,*}
			}
		}.
	\]
	Therefore, inequality \eqref{eq:cons-def-qn-square-adaptation} applies directly, implying $\tilde{I}_0$ has the same order as $I_0$.

	For $\tilde{I}_1$, the same pointwise inequalities yield
	\[
		|\tilde{I}_1| \leq 256
		\vvmean{
			(\tau + \tau_*)
			\left|
				\frac{
					(3 P'_n P'_{n,*} + 1) F g_n
				}{
					N_n N_{n,*} N'_n N'_{n,*}
				}
				q_n
			\right|
		}.
	\]
	Applying inequality \eqref{eq:cons-def-young-fenchel-I1} and all the subsequent estimates for $I_1$, we obtain that $\tilde{I}_1$ is of the same order as $I_1$.

	Finally, for $\tilde{I}_2$, we incorporate the additional bound for $\mu_n g_n/N_n$, established in Lemma \ref{lem:Nn-Pn-bounds}.
	Hence, there exists a constant $C_6 > 0$ such that
	\begin{equation*}
		|\mu_n \tilde{I}_2| \leq C_6 \vvmean{
			\left|
				\tau'
				\frac{
					g'_{n,*}
				}{
					N_n N_{n,*} N'_{n,*}
				}
				q_n
			\right|
		}.
	\end{equation*}
	Applying inequality \eqref{eq:cons-def-estimate-I2} and repeating the subsequent steps then show that $\mu_n \tilde{I}_2$ has the same asymptotic order as $\mu_n I_2$, which completes the estimates for $\widetilde{T}_2$.
\end{proof}

\section{The acoustic limit theorem}
\label{sec:acoustic-limit-proof}

\proofpart{1}{The renormalized formulation}

Substituting the decomposition \eqref{eq:micro-macro-decomposition} into the scaled BFD equation \eqref{eq:scaled-bfd-fn} and isolating $g_n$ yields
\begin{equation}
	\label{eq:acoustic-limit-gn}
	\partial_t g_n
	+ v \cdot \nabla_x g_n
	= \frac{1}{\kappa_n \mu_n} \frac{Q(f_n)}{F(1-F)}.
\end{equation}
We define the renormalized fluctuations $\lambda_n$ as
\[
	\lambda_n = \frac{1}{\left(\frac{3}{4} - F\right) \mu_n}
	\log\left[
		1 + \left(\frac{3}{4} - F\right) \mu_n g_n
	\right].
\]

To deduce the governing equation for $\lambda_n$, we apply a standard mollification argument.
Let $\chi = \chi(t,x)$ be a nonnegative $C^\infty_c((0,\infty) \times \mathbb{R}^N_x)$ function such that $\int \chi \; dtdx = 1$.
Consider the mollifying sequence $\chi^m(t,x) = m^{N+1} \chi(mt,mx)$ and define $g^m_n = \chi^m *_{t,x} g_n$, where the convolution is taken only with respect to the variables $t$ and $x$.
Note that to compute this convolution, we extend $g_n$ by zero for $t \leq 0$.
Convolving both sides of \eqref{eq:acoustic-limit-gn} with $\chi^m$ and commuting derivatives with convolutions, we obtain
\[
	\partial_t g^m_n
	+ v \cdot \nabla_x g^m_n
	= \frac{1}{\kappa_n \mu_n} \frac{1}{F(1-F)}
	(\chi^m *_{t,x} Q(f_n)).
\]

For each $n$, let $N^m_n = 1 + (3/4 - F) \mu_n g^m_n$ be the regularized renormalization factor and define $\lambda^m_n = \log(N^m_n) / ((3/4 - F) \mu_n)$ as the corresponding regularized fluctuation.
Dividing both sides of the previous equation by $N^m_n$ and applying the chain rule, we obtain
\begin{equation}
	\label{eq:acoustic-limit-equation-lambda-mn}
	\partial_t \lambda^m_n
	+ v \cdot \nabla_x \lambda^m_n
	= \frac{1}{\kappa_n \mu_n} \frac{1}{F(1-F) N^m_n}
	(\chi^m *_{t,x} Q(f_n)).
\end{equation}

We pass this equation to the limit as $m \to \infty$.
Since $g_n \in L^1_{loc}((0,\infty) \times \Rxv)$, it follows by standard properties of approximation by convolution that $g^m_n \to g_n$ strongly in this space.
Passing to a subsequence in $m$, this convergence also holds almost everywhere and the sequence $(g^m_n)_n$ is dominated by an $L^1_{loc}$ function.
Therefore, $\lambda_n^m$ converges to $\lambda_n$ almost everywhere.

Recall from Lemma \ref{lem:Nn-Pn-bounds} that $1/4 \leq N_n \leq 3/(4F)$.
Since the mollifier $\chi^m$ is nonnegative and has unit mass, it follows that $N^m_n = \chi^m *_{t,x} N_n$ satisfies $1/4 \leq N^m_n \leq 3/(4F)$ for every $m \in \mathbb{N}$.
The sequence $(\log(N^m_n))_m$ is therefore bounded uniformly in $m$.
Combined with the bound $(3/4 - F)^{-1} \leq 4(e+1)/(3-e)$, the dominated convergence theorem implies that $\lambda_n^m \to \lambda_n$ in $L^1_{loc}((0,\infty) \times \Rxv)$.
Thus, the left-hand side of \eqref{eq:acoustic-limit-equation-lambda-mn} converges to $\partial_t \lambda_n + v \cdot \nabla_x \lambda_n$ in $\mathcal{D}'((0,\infty) \times \Rxv)$.

For the right-hand side, we use the entropy dissipation $R(f_n)$ to show that the function $Q(f_n)$ is bounded in $L^1_{loc}((0,\infty) \times \Rxv)$.
We apply the inequality $x \leq 2y + (\log 2)^{-1} (x-y) \log(x/y)$ for all $x, y > 0$, choosing $x = f_n' f'_{n,*} (1 - f_n) (1 - f_{n,*})$ and $y = f_n f_{n,*} (1-f'_n) (1 - f'_{n,*})$.
Multiplying the resulting inequality for $x - y$ by the collision kernel and a test function $\varphi = \varphi(v) \in C^\infty_c(\mathbb{R}^N_v)$, followed by integration over $(x,v,v_*,\omega) \in \mathbb{T}^N_x \times \mathbb{R}^{2N}_{v,v_*} \times \mathbb{S}^{N-1}_\omega$ and bounding $\varphi$ by its $L^\infty$ norm on the right-hand side yields
\[
	\begin{split}
		\int_{xv} Q(f_n) \varphi(v)
		&\leq \int_{xv v_* \omega} b(v-v_*, \omega) f_n f_{n,*} (1 - f'_n) (1 - f'_{n,*}) \varphi(v) \\
		&\quad + \frac{1}{\log 2} R(f_n) \| \varphi \|_{L^\infty(\mathbb{R}^N_v)}.
	\end{split}
\]

To estimate this resulting integral, we bound the collision kernel term using \eqref{eq:b-v-v*-estimate} and apply the estimate $f_n (1 - f'_n) (1 - f'_{n,*}) \leq 1$.
Evaluating the constant integral over $\omega$ and separating the integrals over $v$ and $(x, v_*)$, we obtain
\[
	\begin{split}
		\int_{xv} Q(f_n) \varphi(v)
		&\leq C_b |\mathbb{S}^{N-1}|
		\left(\int_v (1 + |v|) \varphi(v)\right)
		\int_{x v_*} (1 + |v_*|) f_{n,*} \\
		&\quad + \frac{1}{\log 2} R(f_n) \| \varphi \|_{L^\infty(\mathbb{R}^N_v)}.
	\end{split}
\]

The integral in parentheses is finite due to the compact support of $\varphi$.
Meanwhile, the first term is bounded uniformly in $t$ by the global conservation laws applied to $f_{n,*}$.
Furthermore, integrating both sides over a bounded time interval $[0,T]$ shows that the second term remains finite due to the entropy inequality \eqref{eq:entropy-inequality}.

Therefore, $Q(f_n)$ is bounded in $L^1_{loc}$.
From standard properties of mollification, it then follows that the sequence $(\chi^m *_{t,x} Q(f_n))$ converges almost everywhere and is, up to a subsequence, dominated by an $L^1_{loc}$ function.

Next, recall the uniform lower bound $N^m_n \geq 1/4$ for all $m \in \mathbb{N}$.
Since $(g^m_n)_m$ converges almost everywhere, we have $N^m_n \to N_n$ almost everywhere.
Combined with the convergence of the mollified collision operator, the right-hand side of \eqref{eq:acoustic-limit-equation-lambda-mn} converges almost everywhere.
Furthermore, because $(N^m_n)^{-1}$ is bounded and $(\chi^m *_{t,x} Q(f_n))$ is dominated by an $L^1_{loc}$ function, the entire right-hand side is also dominated by an $L^1_{loc}$ function for any fixed $n$, allowing us to pass to the limit by the dominated convergence theorem.

Equating this with the limit of the left-hand side as $m \to \infty$, we obtain the scaled and renormalized version of the BFD equation.
Expanding the definition of the collision integral and using the definition of $q_n$ yields
\[
	\partial_t \lambda_n
	+ v \cdot \nabla_x \lambda_n
	= \frac{\eta_n}{\kappa_n} \frac{1}{F(1-F)}
	\int_{v_* \omega} \frac{q_n}{N_n} b(v-v_*, \omega) F F_* (1-F') (1-F'_*).
\]

\proofpart{2}{Local conservation laws}

Multiplying this renormalized equation by $\zeta(v) \in \ker \mathcal{L}$ and integrating over $v$ yields the approximate local conservation laws
\begin{equation}
	\label{eq:acoustic-limit-local-conservation-laws}
	\partial_t \vmean{\zeta \lambda_n}
	+ \Div_x \vmean{v \zeta \lambda_n}
	= \frac{\eta_n}{\kappa_n} \vvmean{\zeta \frac{q_n}{N_n}}.
\end{equation}

To pass these conservation laws to the limit as $n \to \infty$, recall that under the acoustic scaling, $\eta_n = \sqrt{\kappa_n}$.
By Lemma \ref{lem:conservation-defects}, the conservation defects on the right-hand side satisfy
\begin{equation}
	\label{eq:acoustic-limit-conservation-defects-scaling}
	\frac{\eta_n}{\kappa_n} \vvmean{\zeta \frac{q_n}{N_n}}
	= \mathcal{O}
	\left(
		\frac{\mu_n}{\sqrt{\kappa_n}} \sqrt{\log \bigg( \frac{1}{\mu_n} \bigg)}
	\right)
	+ \mathcal{O}
	\left(\mu_n \log \bigg(\frac{1}{\mu_n}\bigg)\right)
	+ \mathcal{O}
	\left(\mu_n \log \bigg(\frac{1}{\kappa_n}\bigg)\right)
\end{equation}
in $L^1_{loc}(dtdx)$.
For sufficiently large $n$, the inequality $\log(1 / \kappa_n) \leq \log(1 / \mu_n)$ ensures the third term is bounded by the second.
Since $\mu_n \to 0$, both of these terms vanish in the limit, while the first term vanishes by the hypothesis \eqref{eq:acoustic-limit-technical-scaling}.
Therefore, the entire expression tends to zero as $n \to \infty$.

To pass the left-hand side of \eqref{eq:acoustic-limit-local-conservation-laws} to the limit by weak convergence, consider the inequality $\left(\log(1+z)\right)^2 \leq z^2 / (1+z)$ for all $z > -1$.
Substituting $z = (3/4-F)\mu_n g_n$ and dividing the resulting inequality by $(3/4-F)^2 \mu_n^2$, we obtain
\[
	\lambda_n^2 \leq \frac{g^2_n}{N_n}.
\]
By Lemma \ref{lem:gn2-Nn-bounds}, the right-hand side is uniformly bounded in $L^\infty(dt; L^1(F(1-F)dvdx))$.
Therefore, the sequence $(\lambda_n)_n$ is uniformly bounded in $L^\infty(dt; L^1(F(1-F)dvdx))$.
Applying the Banach-Alaoglu theorem and passing to a subsequence (which we still denote by index $n$) then yields
\[
	\lambda_n \weakstarto \lambda
	\quad\text{in } L^\infty(dt; L^2(F(1-F)dvdx)).
\]

For every test function $\varphi = \varphi(t,x) \in L^1(dt; L^2(dx))$, the product $\varphi(t,x)\zeta(v)$ naturally belongs to $L^1(dt; L^2(F(1-F)dvdx))$.
Therefore, the preceding limit implies the weak-* convergences $\vmean{\zeta \lambda_n} \weakstarto \vmean{\zeta \lambda}$ and $\vmean{\zeta v_j \lambda_n} \weakstarto \vmean{\zeta v_j \lambda}$ in $ L^\infty(dt; L^2(dx))$.
It follows that the left-hand side of \eqref{eq:acoustic-limit-local-conservation-laws} converges to $\partial_t \vmean{\zeta \lambda} + \Div_x \vmean{v \zeta \lambda}$ in $\mathcal{D}'(dtdx)$.

By the weak compactness established in Lemma \ref{lem:tau-gn-compactness}, we use a diagonal argument to extract a further subsequence such that $g_n \weakto g$ in $L^1_{loc}(dt; L^1(F(1-F) dvdx))$.
To identify the renormalized limit $\lambda$ with this physical limit $g$, we apply the inequality $0 \leq z - \log(1+z) \leq z^2 / (1+z) $ for all $z > -1$.
Substituting $z = (3/4 - F) \mu_n g_n$ and dividing both sides by $(3/4-F) \mu_n$ yields
\begin{equation}
	\label{eq:acoustic-limit-gn-lambdan-diff}
	0 \leq g_n - \lambda_n
	\leq \left(\frac{3}{4} - F\right) \mu_n \frac{g^2_n}{N_n}.
\end{equation}

From the first estimate in Lemma \ref{lem:gn2-Nn-bounds}, the right-hand side is uniformly bounded in $L^\infty(dt; L^1(\tau F(1-F)dvdx))$.
Since $\mu_n \to 0$, we deduce that $g_n - \lambda_n \to 0$ in this space, implying $\lambda_n \weakto g$ weakly in $L^1([0,T] \times \mathbb{T}^N_x \times \mathbb{R}^N_v; dtdx F(1-F)dv)$.
Uniqueness of the limit in the sense of distribution then yields $\lambda = g$ and the previously established weak-* limits become $\vmean{\zeta \lambda_n} \weakstarto \vmean{\zeta g}$ and $\vmean{\zeta v_j \lambda_n} \weakstarto \vmean{\zeta v_j g}$ in $L^\infty(dt; L^2(dx))$.
Therefore, passing the local conservation laws \eqref{eq:acoustic-limit-local-conservation-laws} to the limit yields, for every $\zeta \in \operatorname{span} \{ 1, v_j, |v|^2 \}$,
\begin{equation}
	\label{eq:acoustic-local-conserv}
	\partial_t \vmean{\zeta g} + \Div_x \vmean{v \zeta g} = 0.
\end{equation}

\proofpart{3}{Derivation of the acoustic system}

Recalling the weak compactness of $(\tau g_n)_n$ established in Lemma \ref{lem:tau-gn-compactness}, we deduce from Lemma \ref{lem:limit-local-maxwellian} that the limit $g$ takes the explicit infinitesimal Fermi-Dirac form
\[
	g(t,x,v) = \rho(t,x)
		+ u(t,x) \cdot v
		+ \theta(t,x) \frac{1}{2} \left( |v|^2 - K \right).
\]
Furthermore, the first result from Lemma \ref{lem:lower-semi-continuity-entropy} implies that $g \in L^\infty(dt; L^2(F(1-F) dvdx))$.
Since the velocity basis functions $1$, $v_j$ and $\frac{1}{2} (|v|^2 - K)$ are mutually orthogonal in $L^2(F(1-F) dv)$, it follows that the macroscopic parameters $\rho$, $u$ and $\theta$ belong to $L^\infty(dt; L^2(dx))$.

To derive the macroscopic equations, we substitute the explicit infinitesimal Fermi-Dirac form for $g$ into the local conservation laws \eqref{eq:acoustic-local-conserv} with $\zeta = 1$, $\zeta = v_i$ and $\zeta = |v|^2 - K$.
In evaluating the moments $\vmean{\zeta g}$ and $\vmean{\zeta v g}$, all integrals with odd powers of $v$ vanish by parity.
For the remaining terms, we use the definition of $K$, which implies $\vmean{|v|^2 - K} = 0$, alongside the isotropy identities
\[
	\begin{split}
		&\vmean{v_i v_j}
		= \frac{\vmean{|v|^2}}{N} \delta_{ij}
		= \frac{K}{N} \vmean{1} \delta_{ij}, \\
		&\vmean{(|v|^2 - K)v_i v_j}
		= \frac{\vmean{|v|^4} - K \vmean{|v|^2}}{N} \delta_{ij}
		= \frac{2K}{N} \vmean{1} \gamma \delta_{ij},
	\end{split}
\]
where $\gamma = \vmean{|v|^4} / (2 \vmean{|v|^2}) - K / 2$.

For $\zeta = 1$, the density and flux terms reduce to $\vmean{g} = \vmean{1} \rho$ and $\vmean{v g} = \frac{K}{N} \vmean{1} u$.
Dividing the resulting conservation law by $\vmean{1}$ yields the continuity equation
\[
	\partial_t \rho + \frac{K}{N} \Div_x u = 0.
\]

For $\zeta = v_i$, the density term is $\vmean{v_i g} =  \frac{K}{N} \vmean{1} u_i$.
Applying the isotropy relations to the flux term gives $\vmean{v_i v_j g} = \frac{K}{N} \vmean{1} (\rho + \gamma \theta) \delta_{ij}$.
Dividing the resulting conservation law by $\frac{K}{N} \vmean{1}$ yields the velocity equation
\[
	\partial_t u + \nabla_x ( \rho + \gamma \theta ) = 0.
\]

Finally, for $\zeta = |v|^2 - K$, the vanishing and isotropy arguments simplify the density term to $\vmean{\zeta g} = \frac{1}{2} \theta \vmean{(|v|^2 - K)^2}$, while the flux term reduces to $\vmean{\zeta v_i g} = \frac{2K}{N} \vmean{1} \gamma u_i$.
Expanding the square and applying the definition of $\gamma$, we obtain that $\vmean{(|v|^2 - K)^2} = 2 \gamma K \vmean{1}$, implying the density term can be rewritten as $\vmean{\zeta g} = \gamma K \vmean{1} \theta$.
Substituting into the conservation law and dividing by $\gamma K \vmean{1}$ yields the temperature equation
\[
	\partial_t \theta + \frac{2}{N} \Div_x u = 0.
\]

\proofpart{4}{Identification of the initial data}

Assume the moments of the initial fluctuations $(g^{in}_n)_n$ converge in the sense of distributions to $(\rho_0, u_0, \theta_0)$ as in the statement of the theorem.
To identify these parameters with the initial data of the acoustic system, we apply a variant of the Arzelà-Ascoli theorem.

Returning to the approximate local conservation laws \eqref{eq:acoustic-limit-local-conservation-laws}, we test the equation against a purely spatial function $\varphi \in C^\infty(\mathbb{T}^N)$.
Defining $\Phi_n(t) = \int_{\mathbb{T}^N} \vmean{\zeta \lambda_n}(t,x) \varphi(x) \, dx$, its time derivative in the sense of distributions is given by $\Phi'_n(t) = a_n(t) + b_n(t)$, where
\[
	\begin{split}
		a_n(t)
		&= \int_{\mathbb{T}^N} \vmean{v \zeta \lambda_n} \cdot \nabla_x \varphi(x) \, dx, \\
		b_n(t)
		&= \frac{\eta_n}{\kappa_n}
			\int_{\mathbb{T}^N} \vvmean{\zeta \frac{q_n}{N_n}} \varphi(x) \, dx.
	\end{split}
\]

Fix $T > 0$.
Applying the Cauchy-Schwarz inequality in the $(x,v)$ variables, followed by an $L^\infty$ bound in $t$ on the right-hand side, we have for almost every $t \in [0,T]$ that
\[
	| \Phi_n(t) |
	\leq \| \varphi \|_{L^2(dx)} \| \zeta \|_{L^2(F(1-F) dv)} \| \lambda_n \|_{L^\infty(dt; L^2(F(1-F) dvdx))}.
\]
Since $(\lambda_n)_n$ is uniformly bounded in $L^\infty(dt; L^2(F(1-F) dvdx))$, this implies that the sequence $(\Phi_n)_n$ is uniformly bounded in $L^\infty([0,T])$.

We estimate the $a_n$ and $b_n$ terms separately.
For $a_n$, expanding the definition of $\vmean{\cdot}$, applying the Cauchy-Schwarz inequality in $(x,v)$ and an $L^\infty$ bound in time yields, for almost every $t \in [0,T]$,
\begin{equation*}
	\begin{split}
		|a_n(t)|
		&= \left|
			\iint_{\Rxv} F(1-F) v \zeta \lambda_n \cdot \nabla_x \varphi(x) \, dxdv
		\right| \\
		&\leq \| \zeta v \|_{L^2(F(1-F) dv)}
			\| \nabla_x \varphi \|_{L^2(dx)}
			\| \lambda_n \|_{L^\infty(dt; L^2(F(1-F) dvdx))},
	\end{split}
\end{equation*}
which shows that $(a_n)_n$ is uniformly bounded in $L^\infty([0,T])$.

For $b_n$, bounding $\varphi$ by its $L^\infty$ norm and integrating over $[0,T]$ yields
\[
	\int_0^T | b_n(t) | \, dt
	\leq
		\left\|
			\frac{\eta_n}{\kappa_n} \vvmean{\zeta \frac{q_n}{N_n}}
		\right\|_{L^1(dtdx)}
		\| \varphi \|_{L^\infty(dx)}.
\]
From the vanishing of the conservation defects established in \eqref{eq:acoustic-limit-conservation-defects-scaling}, it follows that the right-hand side converges to zero in the limit as $n \to \infty$.

Combining this with the uniform $L^\infty$ bound on $a_n$, we deduce that for each $n$, the derivative $\Phi'_n$ is integrable on $[0,T]$.
Since $\Phi_n$ is also uniformly bounded in $L^\infty([0,T])$, it follows that $\Phi_n \in W^{1,1}([0,T])$.
Therefore, $\Phi_n$ is absolutely continuous, allowing us to apply the fundamental theorem of calculus to write
\[
	\Phi_n(t)
	= \Phi_n(0)
	+ \int_0^t a_n(\tau) \, d\tau
	+ \int_0^t b_n(\tau) \, d\tau.
\]

Let $A_n(t) = \int_0^t a_n(\tau) \, d\tau$.
Since its derivative in the sense of distributions is $a_n$ and $(a_n)_n$ is uniformly bounded in $L^\infty([0,T])$, the sequence $(A_n)_n$ uniformly Lipschitz continuous, thus equicontinuous.
Furthermore, since $|A_n(t)| \leq T \| a_n \|_{L^\infty([0,T])}$, the sequence is also uniformly bounded.
By the Arzela-Ascoli theorem, up to extraction of a subsequence, $(A_n)_n$ converges uniformly on $[0,T]$.

Let $B_n(t) = \int_0^t b_n(\tau) \, d\tau$.
Since $b_n \to 0$ strongly in $L^1([0,T])$ from the vanishing of conservation defects, we have
\[
	\sup_{t \in [0,T]} | B_n (t) |
	\leq \int_0^T |b_n(\tau)| \, d\tau \to 0
	\quad\text{as }  n \to \infty,
\]
which implies that $B_n$ converges to $0$ uniformly on $[0,T]$.
Moreover, $(\Phi_n(0))_n$ being a bounded real sequence, extracting a convergent subsequence we obtain that $\Phi_n(t) = \Phi_n(0) + A_n(t) + B_n(t)$ converges uniformly on $[0,T]$.

Let $\Phi \in C^0([0,T])$ be this uniform limit.
Since $\lambda_n \weakto g$ in $L^1([0,T] \times \Rxv)$, testing this weak convergence against $\varphi(x) \psi(t)$ for any test function $\psi \in C^\infty_c(0,T)$ implies that $\Phi = \int_{\mathbb{T}^N} \vmean{\zeta g} \varphi(x) \, dx$ almost everywhere in time.
Up to redefining $g$ in a set of null measure, we conclude that for every $\varphi \in C^\infty_c(\mathbb{T}^N)$,
\[
	\Phi_n(t)
	= \int_{\mathbb{T}^N} \vmean{\zeta \lambda_n}(t) \varphi \, dx
	\to \int_{\mathbb{T}^N} \vmean{\zeta g}(t) \varphi \, dx
\]
uniformly on $[0,T]$.

By the hypothesis of the theorem, the macroscopic moments of the initial fluctuations converge in the sense of distributions.
With a slight abuse of notation, we formally denote these limiting distributions by $\vmean{\zeta g^{in}}$.
Specifically, for $\zeta \in \{1, v_i, |v|^2 - K\}$, we denote
\[
	\vmean{\zeta g^{in}}
	\coloneq \lim_{n \to \infty} \vmean{\zeta g^{in}_n},
\]
where the limit is understood in $\mathcal{D}'(dx)$.

We now show that $\int_{\mathbb{T}^N} \vmean{\zeta g}(0) \varphi \, dx = \int_{\mathbb{T}^N} \vmean{\zeta g^{in}} \varphi \, dx$.
Definition \ref{def:weak-solution} guarantees the weak solutions $f_n$ are weakly continuous at $t = 0$, which implies the initial trace of the fluctuations $g_n(0)$ coincides with $g^{in}_n$.
Decomposing the difference using the triangle inequality, we obtain
\begin{equation*}
	\begin{split}
		\left|
			\int_{\mathbb{T}^N} \vmean{\zeta g}(0) \varphi \, dx
			- \int_{\mathbb{T}^N} \vmean{\zeta g^{in}} \varphi \, dx
		\right|
		&\leq \left|
			\int_{\mathbb{T}^N} \vmean{\zeta g}(0) \varphi \, dx
			- \int_{\mathbb{T}^N} \vmean{\zeta \lambda_n}(0) \varphi \, dx
		\right| \\
		&\quad + \left|
			\int_{\mathbb{T}^N} \vmean{\zeta \lambda_n}(0) \varphi \, dx
			- \int_{\mathbb{T}^N} \vmean{\zeta g_n}(0) \varphi \, dx
		\right| \\
		&\quad + \left|
			\int_{\mathbb{T}^N} \vmean{\zeta g^{in}_n} \varphi \, dx
			- \int_{\mathbb{T}^N} \vmean{\zeta g^{in}} \varphi \, dx
		\right|.
	\end{split}
\end{equation*}

The first term on the right-hand side vanishes as $n \to \infty$ by the uniform convergence established above, applied to $t = 0$.
The third term vanishes by definition of the notation $\vmean{\zeta g^{in}}$ and the assume convergence of the initial data in the sense of distributions.
Therefore, it remains only to show that the second term vanishes in the limit.

Since $|3/4 - F| \leq 1$, the estimate \eqref{eq:acoustic-limit-gn-lambdan-diff} implies that $|\zeta g_n - \zeta \lambda_n| \leq \mu_n |\zeta| g_n^2 / N_n$.
Using that $|\zeta| \leq C_\zeta \tau$ for some constant $C_\zeta > 0$, we integrate the resulting inequality against the spatial test function $\varphi$ and the velocity measure $F(1-F) dv$.
Applying an $L^\infty$ bound on $\varphi$ and the estimate for $\tau g_n^2 / N_n$ established in Lemma \ref{lem:gn2-Nn-bounds}, we deduce that there exists a constant $C_0 > 0$ such that, for almost every $t \in [0,\infty)$,
\begin{equation*}
	\begin{split}
		\left|
			\int_{\mathbb{T}^N} \vmean{\zeta g_n}(t) \varphi \, dx
			- \int_{\mathbb{T}^N} \vmean{\zeta \lambda_n}(t) \varphi \, dx
		\right|
		&\leq \int_{\mathbb{T}^N} \vmean{|\zeta g_n - \zeta \lambda_n|}(t) |\varphi| \, dx \\
		&\leq C_0 \| \varphi \|_{L^\infty} \mu_n \log\left(\frac{1}{\mu_n}\right).
	\end{split}
\end{equation*}

As established above, the function $t \mapsto \int_{\mathbb{T}^N} \vmean{\zeta \lambda_n}(t) \varphi \, dx = \Phi_n(t)$ is continuous.
Moreover, the function $t \mapsto \int_{\mathbb{T}^N} \vmean{\zeta g_n}(t) \varphi \, dx$ is also continuous by Definition \ref{def:weak-solution} of a weak solution.
Therefore, the above inequality actually holds for every $t \geq 0$.

In particular, evaluating this bound at $t = 0$ shows that the second term in our triangle inequality decomposition is bounded by $C_0 \| \varphi \|_{L^\infty} \mu_n \log(1 / \mu_n)$.
Since this vanishes in the limit $n \to \infty$, all three terms of the decomposition approach zero.
We therefore conclude that
\[
	\int_{\mathbb{T}^N} \vmean{\zeta g}(0) \varphi \, dx
	= \int_{\mathbb{T}^N} \vmean{\zeta g^{in}} \varphi \, dx.
\]
Choosing the $\zeta \in \{1, v_i, |v|^2 - K \}$ corresponding to the macroscopic quantities $\rho$, $u$ and $\theta$, we recover the desired identification of the initial data.

With the limit Cauchy problem now fully determined, we can upgrade the convergence of $g_n$ (until now established only for a subsequence) to the full sequence.
Let $g_{n_k} \weakto g$ be any weakly convergent subsequence.
By the arguments in Part 2, we can extract a further sub-subsequence along which $g$ is identified as an infinitesimal Fermi-Dirac distribution.
From Part 3, the parameters $(\rho, u, \theta)$ of this distribution form a weak solution to the acoustic system and, as just established, its initial data is precisely the $(\rho_0, u_0, \theta_0)$ given in the statement of the theorem.

Because this acoustic system admits a unique weak solution, the limiting distribution $g$ is uniquely determined.
Therefore, every weakly convergent subsequence of $g_n$ must converge to this identical limit.
By the weak compactness of the sequence, this guarantees that the entire sequence $g_n$ converges to $g$ as $n \to \infty$.

\proofpart{5}{Zero mean of the macroscopic parameters}

Finally, we establish the zero mean for the parameters $\rho$, $u$ and $\theta$.
By hypothesis \eqref{eq:fn-invariants-concentrated}, the initial data is constructed such that the total mass, momentum and kinetic energy are concentrated on the global Fermi-Dirac distribution $F$.
Thus, the fluctuations satisfy, for any collision invariant $\zeta \in \operatorname{span}\{1, v_i, |v|^2\}$,
\[
	\int_{\mathbb{T}^N} \vmean{\zeta g^{in}_n (x)} \, dx = 0.
\]

Using the convergence in the sense of distributions of the initial moments given in the statement of the theorem, we pass this relation to the limit.
This implies that the spatial integrals of the macroscopic initial data $\rho_0$, $u_0$ and $\theta_0$ over $\mathbb{T}^N$ vanish.

For $t > 0$, the parameters $\rho$, $u$ and $\theta$ are governed by the acoustic system subject to this initial data.
Integrating these macroscopic equations over the periodic domain $x \in \mathbb{T}^N$ causes the divergence and gradient terms to vanish, yielding the conservation of the spatial means of $\rho$, $u$ and $\theta$.
Because these means are zero for $\rho_0$, $u_0$ and $\theta_0$, we conclude that they remain zero for all $t \geq 0$, completing the proof.

\section{The Stokes-Fourier limit}
\label{sec:stokes-limit}

Before tackling the proof of Theorem \ref{thm:stokes-limit}, we establish a technical lemma that connects renormalized forms of the collision integral \eqref{eq:collision-integral} and the linearized collision integral \eqref{eq:linearized-op-def} asymptotically.
This estimate will be crucial for passing the flux term of the local approximate conservation laws to the limit in the proof of the main theorem.

\begin{lemma}
	\label{lem:stokes-fluxes-difference-vanishing}
	We have that
	\[
		\frac{q_n}{N_n N_{n,*} N'_n N'_{n,*}}
		- \frac{1}{\eta_n} \left(
			\frac{g'_{n,*}}{N'_{n,*}}
			+ \frac{g'_n}{N'_n}
			- \frac{g_n}{N_n}
			- \frac{g_{n,*}}{N_{n,*}}
		\right)
		= \mathcal{O}\left(
			\frac{\mu_n}{\eta_n} \left(\log \frac{1}{\mu_n}\right)^{1/2}
		\right)
	\]
	in $L^\infty(dt; L^1(dx; L^2(\Lambda)))$.
\end{lemma}
\begin{proof}
	Recall the decomposition established in \eqref{eq:linear-part-decomposition}.
	By rearranging this identity to isolate $\eta_n q_n$, we obtain
	\begin{equation*}
	\begin{split}
		\eta_n q_n
		&= (
			g'_{n,*} M'_n M_n M_{n,*}
			+ g'_n M'_{n,*} M_n M_{n,*}
			- g_{n,*} M_n M'_n M'_{n,*}
			- g_n M_{n,*} M'_n M'_{n,*}
		) \\
		&\quad + \mu_n
		(
			g'_n g'_{n,*} M_n M_{n,*} - g_n g_{n,*} M'_n M'_{n,*}
		)
	\end{split}
	\end{equation*}
	Substituting this expansion into our quantity of interest allows us to separate the $\mathcal{O}(\mu_n)$ terms from the rest.
	Specifically, we have
	\begin{equation}
	\label{eq:stokes-flux-quantity}
	\begin{split}
		\frac{q_n}{N_n N_{n,*} N'_n N'_{n,*}}
		- \frac{1}{\eta_n} \left(
			\frac{g'_{n,*}}{N'_{n,*}}
			+ \frac{g'_n}{N'_n}
			- \frac{g_n}{N_n}
			- \frac{g_{n,*}}{N_{n,*}}
		\right)
		&= \frac{\mu_n}{\eta_n} \frac{g'_n g'_{n,*} M_n M_{n,*} - g_n g_{n,*} M'_n M'_{n,*}}{N_n N_{n,*} N'_n N'_{n,*}} \\
		&\quad + \frac{1}{\eta_n} \mathcal{R}_n
	\end{split}
	\end{equation}
	where the remainder $\mathcal{R}_n$ is defined as
	\[
		\mathcal{R}_n = \frac{M_n M_{n,*} M'_n M'_{n,*}}{N_n N_{n,*} N'_n N'_{n,*}}
		\left(
			\frac{g'_{n,*}}{M'_{n,*}}
			+ \frac{g'_n}{M'_n}
			- \frac{g_n}{M_n}
			- \frac{g_{n,*}}{M_{n,*}}
		\right)
		- \left(
			\frac{g'_{n,*}}{N'_{n,*}}
			+ \frac{g'_n}{N'_n}
			- \frac{g_n}{N_n}
			- \frac{g_{n,*}}{N_{n,*}}
		\right).
	\]
	To analyse $\mathcal{R}_n$ more clearly, it is convenient at this stage to introduce the $M$-scaled abbreviations $\gamma_n := g_n / M_n$ and $\nu_n := N_n / M_n$.
	Rewriting $\mathcal{R}_n$ in terms of these new sequences, we can decompose it into two parts, $\mathcal{R}_n = \mathcal{R}^{(1)}_n + \mathcal{R}^{(2)}_n$, where
	\begin{align*}
		\mathcal{R}^{(1)}_n &\coloneq
			\frac{
				\gamma'_n + \gamma'_{n,*} - \gamma_n - \gamma_{n,*}
			}
			{
				\nu_n \nu_{n,*} \nu'_n \nu'_{n,*}
			}
			-
			\left(
				\frac{
					\gamma'_n + \gamma'_{n,*}
				}{
					\nu'_n \nu'_{n,*}
				}
				- \frac{
					\gamma_n + \gamma_{n,*}
				}{
					\nu_n \nu_{n,*}
				}
			\right), \\
		\mathcal{R}^{(2)}_n &\coloneq
			\left(
				\frac{
					\gamma'_n + \gamma'_{n,*}
				}{
					\nu'_n \nu'_{n,*}
				}
				- \frac{\gamma'_n}{\nu'_n}
				- \frac{\gamma'_{n,*}}{\nu'_{n,*}}
			\right)
			- \left(
				\frac{
					\gamma_n + \gamma_{n,*}
				}{
					\nu_n \nu_{n,*}
				}
				- \frac{\gamma_n}{\nu_n}
				- \frac{\gamma_{n,*}}{\nu_{n,*}}
			\right).
	\end{align*}

	We rewrite $\mathcal{R}^{(1)}_n$ by combining its terms over a common denominator, yielding
	\[
		\mathcal{R}^{(1)}_n = \frac{
			(\gamma'_n + \gamma'_{n,*}) (1 - \nu_n \nu_{n,*})
			- (\gamma_n + \gamma_{n,*}) (1 - \nu'_n \nu'_{n,*})
		}{
			\nu_n \nu_{n,*} \nu'_n \nu'_{n,*}
		}.
	\]
	Expanding the definition of $M_n - N_n$ yields the identity $1 - \nu_n = - (3/4) \mu_n \gamma_n$.
	We expand $1 - \nu_n \nu_{n,*}$ by separating the $1 - \nu_n$ terms.
	Substituting the identity into the resulting equation, we obtain
	\begin{equation*}
	\begin{split}
		1 - \nu_n \nu_{n,*}
		&= (1 - \nu_n) + (1 - \nu_{n,*}) - (1 - \nu_n) (1 - \nu_{n,*}) \\
		&=	- \frac{3}{4} \mu_n (\gamma_n + \gamma_{n,*})
			- \frac{9}{16} \mu_n^2 \gamma_n \gamma_{n,*}.
	\end{split}
	\end{equation*}

	By symmetry, an analogous relation holds for the primed term $1 - \nu'_n \nu'_{n,*}$.
	When we substitute this expression back into the numerator of $\mathcal{R}^{(1)}_n$, the first-order $\mu_n$ terms cancel out entirely and we have
	\[
		\mathcal{R}^{(1)}_n =
		- \frac{9}{16} \mu_n^2 \frac{\gamma'_n + \gamma'_{n,*}}{\nu'_n \nu'_{n,*}}
		\frac{\gamma_n \gamma_{n,*}}{\nu_n \nu_{n,*}}
		+ \frac{9}{16} \mu_n^2 \frac{\gamma_n + \gamma_{n,*}}{\nu_n \nu_{n,*}}
		\frac{\gamma'_n \gamma'_{n,*}}{\nu'_n \nu'_{n,*}}.
	\]

	For $\mathcal{R}^{(2)}_n$, notice its structure allows us to evaluate its primed and unprimed parenthesis symmetrically.
	Focusing on the unprimed bracket, collecting terms over a common denominator and applying the identity for $1 - \nu_n$, we obtain
	\begin{equation*}
	\begin{split}
		\frac{\gamma_n + \gamma_{n,*}}{\nu_n \nu_{n,*}}
		- \frac{\gamma_n}{\nu_n}
		- \frac{\gamma_{n,*}}{\nu_{n,*}}
		&= \frac{
			\gamma_n
			\left(
				1 - \nu_{n,*}
			\right)
			+ \gamma_{n,*}
			\left(
				1 - \nu_n
			\right)
		}
		{
			\nu_n \nu_{n,*}
		} \\
		&= - \frac{3}{2} \mu_n
		\frac{\gamma_n \gamma_{n,*}}{\nu_n \nu_{n,*}}.
	\end{split}
	\end{equation*}
	By symmetry, an analogous result holds for the primed parenthesis.
	Substituting these expressions into $\mathcal{R}^{(2)}_n$ yields
	\[
		\mathcal{R}^{(2)}_n = - \frac{3}{2} \mu_n \frac{\gamma'_n \gamma'_{n,*}}{\nu'_n \nu'_{n,*}}
		+ \frac{3}{2} \mu_n \frac{\gamma_n \gamma_{n,*}}{\nu_n \nu_{n,*}}.
	\]

	Returning to \eqref{eq:stokes-flux-quantity}, we rewrite the first term on the right-hand side in terms of the $M$-scaled abbreviations $\nu_n$ and $\gamma_n$, which yields $(g'_n g'_{n,*} M_n M_{n,*} - g_n g_{n,*} M'_n M'_{n,*}) / (N_n N_{n,*} N'_n N'_{n,*}) = (\gamma'_n \gamma'_{n,*} - \gamma_n \gamma_{n,*}) / (\nu_n \nu_{n,*} \nu'_n \nu'_{n,*})$.
	Substituting this identity, along with the expressions for $\mathcal{R}^{(1)}_n$ and $\mathcal{R}^{(2)}_n$ shown above, we obtain
	\begin{equation*}
	\begin{split}
		\frac{q_n}{N_n N_{n,*} N'_n N'_{n,*}}
		- {}& \frac{1}{\eta_n} \left(
			\frac{g'_{n,*}}{N'_{n,*}}
			+ \frac{g'_n}{N'_n}
			- \frac{g_n}{N_n}
			- \frac{g_{n,*}}{N_{n,*}}
		\right) \\
		& = - \frac{3}{2} \frac{\mu_n}{\eta_n}
			\left(
				\frac{3}{8} \mu_n
				\frac{\gamma'_n + \gamma'_{n,*}}{\nu'_n \nu'_{n,*}}
				+ \frac{2}{3} \frac{1}{\nu'_n \nu'_{n,*}}
				- 1
			\right)
			\frac{\gamma_n \gamma_{n,*}}{\nu_n \nu_{n,*}} \\
		&\quad + \frac{3}{2} \frac{\mu_n}{\eta_n}
 			\left(
 				\frac{3}{8} \mu_n
 				\frac{\gamma_n + \gamma_{n,*}}{\nu_n \nu_{n,*}}
 				+ \frac{2}{3} \frac{1}{\nu_n \nu_{n,*}}
 				- 1
 			\right)
 			\frac{\gamma'_n \gamma'_{n,*}}{\nu'_n \nu'_{n,*}},
 	\end{split}
	\end{equation*}
	Since the second term on the right-hand side is simply the primed version of the first (up to a sign change), it suffices to analyze the first term.

	We rewrite the quotient $(\gamma_n \gamma_{n,*})/(\nu_n \nu_{n,*})$ back in terms of $g_n$ and $N_n$.
	Inside the parentheses, we apply the identity $\mu_n \gamma'_n = - (4/3) (1 - \nu'_n)$, as well as the analogous relation for $\gamma'_{n,*}$, on the first fraction.
	Simplifying the result, we find that this first term equals
	\[
		- \frac{3}{2} \frac{\mu_n}{\eta_n}
		\left(
			\frac{1}{2 \nu'_n}
			+ \frac{1}{2 \nu'_{n,*}}
			- \frac{1}{3 \nu'_n \nu'_{n,*}}
			- 1
		\right)
		\frac{g_n g_{n,*}}{N_n N_{n,*}}.
	\]

	Using the bounds $1/N_n \leq 4$ (from Lemma \ref{lem:Nn-Pn-bounds}) and $M_n \leq (1-F)^{-1} \leq 1+e$, we deduce that $1/\nu_n \leq 4(1+e)$, which implies the term in parentheses above is uniformly bounded.
	Meanwhile, for the $(g_n g_{n,*})/(N_n N_{n,*})$ factor, we apply \eqref{eq:b-v-v*-estimate} alongside the bound $(1 - F)^{-1} \leq 1+e$ and the trivial estimates $1-F', 1-F'_* \leq 1$ to obtain
	\begin{equation*}
		\vvmean{
			\left(
				\frac{g_n g_{n,*}}{N_n N_{n,*}}
			\right)^2
		}
		\leq C_b |\mathbb{S}^{N-1}| (1+e)^2
			\left(
				\int_{\mathbb{R}^N} (1+|v|) \frac{g_n^2}{N_n^2} F(1-F) \, dv
			\right)^2.
	\end{equation*}

	We take the square root of this inequality and integrate it over $x \in \mathbb{T}^N$.
	By applying the Cauchy-Schwarz inequality, the bound $(1+|v|^2) \leq 8 \tau$ and the estimate for $N_n$ from Lemma \ref{lem:Nn-Pn-bounds}, we have that
	\begin{equation*}
		\int_{\mathbb{T}^N}
			\vvmean{
				\left(
					\frac{g_n g_{n,*}}{N_n N_{n,*}}
				\right)^2
			}^{1/2} dx
		\leq 8 \sqrt{2 C_b |\mathbb{S}^{N-1}|} (1+e)
		\left\| \tau \frac{g_n^2}{N_n} \right\|^{1/2}_{L^1(F(1-F) dvdx)}
		\left\| \frac{g_n^2}{N_n} \right\|^{1/2}_{L^1(F(1-F) dvdx)}.
	\end{equation*}
	Finally, from Lemma \ref{lem:gn2-Nn-bounds}, the first norm is $\mathcal{O}(\log(1/\mu_n))$ and the second one is $\mathcal{O}(1)$, which yields the desired bound.
\end{proof}

We are now ready to prove the main theorem of this section.

\begin{proof}[Proof of Theorem \ref{thm:stokes-limit}]
	\proofpart{1}{Renormalization and weak compactness}

	We start with the scaled BFD equation,
	\[
		\sigma_n \partial_t g_n
		+ v \cdot \nabla_x g_n
		= \frac{1}{\kappa_n \mu_n} \frac{Q(f_n)}{F(1-F)}.
	\]

	To renormalize the equation in the Stokes setting, we divide it by $N_n^2$.
	In order to rigorously justify this procedure, we apply the exact same mollification and limiting arguments developed for the acoustic limit in Section \ref{sec:acoustic-limit-proof}.

	Convolving the equation in $(t,x)$ with the same mollifier $\chi^m = \chi^m(t,x)$ and dividing by the regularized factor $(N^m_n)^2$, we apply the chain rule on the left-hand side to deduce that
	\[
		\sigma_n \partial_t \left[ \frac{g^m_n}{N^m_n} \right]
		+ v \cdot \nabla_x \left[ \frac{g^m_n}{N^m_n} \right]
		= \frac{1}{\kappa_n \mu_n} \frac{1}{F(1-F)}  \frac{1}{(N^m_n)^2} [\chi^m *_{t,x} Q(f_n)].
	\]
	As before, we pass to a subsequence such that $(g^m_n)_m$ converges almost everywhere to $g_n$ and is dominated by an $L^1_{loc}((0,\infty) \times \Rxv)$ function.
	This implies that $N^m_n = 1 + (3/4 - F) \mu_n g^m_n$ converges almost everywhere to $N_n$.
	From the uniform lower bound $N^m_n \geq 1/4$ for all $m$, we conclude by dominated convergence that $(g^m_n / N^m_n)_m$ converges in $L^1_{loc}$, which allows us to pass to the limit in the sense of distributions on the left-hand side.

	For the right-hand side, we recall that for each fixed $n$, the collision integral is bounded in $L^1_{loc}((0,\infty) \times \Rxv)$.
	Therefore, from standard properties of mollification, the sequence $(\chi^m *_{t,x} Q(f_n))_m$ converges almost everywhere and is, up to a subsequence, dominated by an $L^1_{loc}$ function.
	The almost everywhere convergence and the uniform lower bound for $(N^m_n)_m$ then allow us to pass to the limit on the right-hand side as $m \to \infty$ by dominated convergence.

	Equating these limits in the sense of distributions, we expand the collision integral $Q(f_n)$ on the right-hand side using the definition \eqref{eq:scaled-collision-product} of $q_n$.
	Recalling that the Stokes scaling yields $\sigma_n = \kappa_n = \eta_n$, we obtain
	\begin{equation}
	\label{eq:stokes-limit-scaled-renorm-equation}
	\begin{split}
		&\kappa_n \partial_t \left[ \frac{g_n}{N_n} \right]
		+ v \cdot \nabla_x \left[ \frac{g_n}{N_n} \right] \\
		&= \frac{1}{F(1-F)}
			\iint_{\mathbb{R}^N \times \mathbb{S}^{N-1}}
			b(v-v_*, \omega) F F_* (1-F') (1-F'_*)
			\frac{q_n}{N_n^2}
			\, dv_* d\omega.
	\end{split}
	\end{equation}

	By the weak compactness established in Lemma \ref{lem:tau-gn-compactness}, we can extract a subsequence (still indexed by $n$) such that $g_n \weakto g$ in $L^1_{loc}(dt; L^1(F(1-F) dvdx))$.
	Define the sequence of renormalized fluctuations $g^\flat_n \coloneq g_n / N_n$.
	The $L^\infty(dt; L^1(F(1-F) dvdx))$ bound for $g_n^2 / N_n$ established in Lemma \ref{lem:gn2-Nn-bounds}, together with the uniform bound on $1 / N_n$ from Lemma \ref{lem:Nn-Pn-bounds}, implies that $(g^\flat_n)_n$ is uniformly bounded in $L^\infty(dt; L^2(F(1-F) dvdx))$.
	Therefore, by the Banach-Alaoglu theorem, we can pass to a further subsequence to obtain
	\[
		g^\flat_n \weakstarto g^*
		\quad\text{in }
		L^\infty(dt; L^2(F(1-F) dvdx)).
	\]

	Expanding the definition of $g^\flat_n$, we express the difference $g_n - g^\flat_n$ in terms of $N_n - 1$.
	Substituting the definition of $N_n$ and using that $| 3 / 4 - F | \leq 1$, we find that
	\begin{equation}
		\label{eq:stokes-limit-gn-gnflat-diff}
		|g_n - g^\flat_n|
		\leq \mu_n \frac{g_n^2}{N_n}.
	\end{equation}
	From the first estimate in Lemma \ref{lem:gn2-Nn-bounds}, the fraction on the right-hand side is uniformly bounded in $L^\infty(dt; L^1(F(1-F)dvdx))$.
	Since $\mu_n \to 0$, we deduce that $g_n - g^\flat_n \to 0$ strongly in this space, implying $g^* = g$.

	Moreover, by Lemma \ref{lem:limit-local-maxwellian}, we find that $g$ has the explicit infinitesimal Fermi-Dirac form
	\begin{equation}
		\label{eq:stokes-limit-explicit-infinitesimal-fd}
		g(t,x,v) = \rho(t,x)
			+ u(t,x) \cdot v
			+ \theta(t,x) \frac{1}{2} \left( |v|^2 - K \right)
	\end{equation}
	on $(0,\infty) \times \Rxv$.
	Since the velocity basis functions $1$, $v_j$ and $\frac{1}{2} \left( |v|^2 - K \right)$ are orthogonal in $L^2(F(1-F) dv)$, it follows from Lemma \ref{lem:lower-semi-continuity-entropy} that the macroscopic fluctuations $\rho$, $u$ and $\theta$ belong to $L^\infty(dt; L^2(dx))$.

	The uniform boundedness of $(g_n)_n$ established in Lemma \ref{lem:tau-gn-compactness} guarantees that $N_n \to 1$ in the space $L^\infty(dt; L^1(F(1-F) dvdx))$.
	We extract a further subsequence along which this convergence holds almost everywhere in $(0,\infty) \times \Rxv$ and, by Lemma \ref{lem:collision-product-compactness}, the sequence $(q_n / N_n)_n$ converges weakly in $L^1(dtdx\Lambda)$.
	Together with the uniform bound on $(N_n)_n$ from Lemma \ref{lem:Nn-Pn-bounds}, the product limit theorem implies that
	\[
		\frac{q_n}{N_n} \weakto q
		\quad\text{and}\quad
		\frac{q_n}{N_n^2} \weakto q
		\quad\text{in }
		L^1(dtdx\Lambda).
	\]

	Multiplying \eqref{eq:stokes-limit-scaled-renorm-equation} by $F(1-F)$ and passing both sides to the limit in the sense of distributions, we obtain the limiting BFD equation
	\begin{equation}
		\label{eq:stokes-limit-transport-g}
		F(1-F) v \cdot \nabla_x g
		= \iint_{\mathbb{R}^N \times \mathbb{S}^{N-1}}
			b(v-v_*, \omega) F F_* (1-F') (1-F'_*)
			q
			\, dv_* d\omega.
	\end{equation}

	\proofpart{2}{Boussinesq and incompressibility relations}

	We start by deducing the approximate local conservation laws.
	Let $\zeta \in \ker \mathcal{L} = \operatorname{span}\{1, v_j, |v|^2\}$.
	We multiply \eqref{eq:stokes-limit-scaled-renorm-equation} by $F(1-F) \zeta(v)$ and integrate over $v$.
	Dividing the resulting expression by $\kappa_n$, we find
	\begin{equation}
		\label{eq:stokes-limit-approximate-local-conserv}
		\partial_t \vmean{\zeta g^\flat_n}
		+ \frac{1}{\kappa_n} \Div_x \vmean{\zeta v g^\flat_n}
		= \frac{1}{\kappa_n} \vvmean{ \zeta \frac{q_n}{N_n^2}}.
	\end{equation}

	The conservation defects on the right-hand side vanish as $n \to \infty$.
	Indeed, in the Stokes scaling we have $\eta_n = \kappa_n$ and the result of Lemma \ref{lem:conservation-defects} yields
	\begin{equation}
		\label{eq:stokes-limit-conservation-defect-vanishing}
		\frac{1}{\kappa_n} \vvmean{\zeta \frac{q_n}{N_n^2}}
		= \mathcal{O}
		\left(
			\frac{\mu_n}{\kappa_n} \sqrt{\log \bigg( \frac{1}{\mu_n} \bigg)}
		\right)
		+ \mathcal{O}
		\left(\mu_n \log \bigg(\frac{1}{\mu_n}\bigg)\right)
		+ \mathcal{O}
		\left(\mu_n \log \bigg(\frac{1}{\kappa_n}\bigg)\right)
	\end{equation}
	in $L^1_{loc}(dtdx)$.
	The first term vanishes directly by the scaling assumption $\mu_n \sqrt{\log(1 / \mu_n)} = o \left( \kappa_n \right)$.
	Furthermore, since this implies $\mu_n = o(\kappa_n)$, we have $\log(1 / \kappa_n) \leq \log(1 / \mu_n)$ for $n$ sufficiently large.
	Therefore, the last two terms are bounded by $\mathcal{O}(\mu_n \log(1 / \mu_n))$, which vanishes as $\mu_n \to 0$.

	To deduce the Boussinesq and incompressibility relations, we multiply the approximate local conservation law \eqref{eq:stokes-limit-approximate-local-conserv} by $\kappa_n$ and take the limit as $n \to \infty$ in the sense of distributions.
	As established above, the right-hand side vanishes in $L^1_{loc}(dtdx)$.

	For the first term on the left-hand side, let $\varphi = \varphi(t,x) \in L^\infty(dtdx)$.
	Testing the weak convergence of $(g^\flat_n)_n$ against functions of the form $\varphi(t,x) \zeta(v)$ and $\varphi(t,x) \zeta(v) v_i$ for $\varphi \in L^\infty(dtdx)$, we obtain the weak-$*$ convergences $\vmean{\zeta g^\flat_n} \weakstarto \vmean{\zeta g}$ and $\vmean{\zeta v_i g^\flat_n} \weakstarto \vmean{\zeta v_i g}$ in $L^\infty(dt; L^2(dx))$.

	Because $\vmean{\zeta g^\flat_n}$ is bounded in this space, the time derivative term $\kappa_n \partial_t \vmean{\zeta g^\flat_n}$ vanishes in the sense of distributions as $\kappa_n \to 0$.
	Thus, in the limit we have
	\[
		\Div_x \vmean{v \zeta g} = 0,
	\]
	for every $\zeta \in \operatorname{span}\{1, v_j, |v|^2\}$.

	Choose $\zeta(v) = 1$ in the equation above.
	Substituting the explicit infinitesimal Fermi-Dirac form for the limit distribution $g$, we observe that the odd integrals $\vmean{v_i}$ and $\vmean{v_i (|v|^2 - K)}$ vanish.
	Using the isotropy identity $\vmean{v_i v_j} = \vmean{|v|^2} \delta_{ij} / N = (K/N) \vmean{1} \delta_{ij}$, we obtain the incompressibility relation $\Div_x u = 0$.

	Next, we choose $\zeta(v) = v_i$ for each $i = 1, \dots, N$ and once again substitute the explicit form for $g$.
	Observe that the odd integral $\vmean{v_i v_j v_k}$ vanishes. Recalling the isotropy identities for $\vmean{v_i v_j}$ and $\vmean{v_i v_j (|v|^2 - K)} = (2K / N) \vmean{1} \gamma \delta_{ij}$ employed in the proof of the acoustic limit, we obtain the weak Boussinesq relation $\nabla_x (\rho + \gamma \theta) = 0$.

	Since the spatial gradient vanishes on the connected domain $\mathbb{T}^N$, it follows that $\rho + \gamma \theta$ is independent of $x$.
	Thus, there exists a function $C_B(t)$ such that $\rho + \gamma \theta = C_B(t)$.
	To show that the function $C_B$ is constant in time, we choose $\zeta = |v|^2$ in the approximate local conservation law \eqref{eq:stokes-limit-approximate-local-conserv} and test this equation against a function of the form $\Psi(t,x) = \varphi(t)$, where $\varphi \in C^\infty_c(0,\infty)$.
	Because $\nabla_x \Psi = 0$, the divergence term vanishes and we obtain, in $\mathcal{D}'(0,\infty)$,
	\[
		\partial_t
		\left(
			\int_{\mathbb{T}^N} \vmean{\zeta g^\flat_n} \, dx
		\right)
		= \int_{\mathbb{T}^N} \frac{1}{\kappa_n} \vvmean{ \zeta \frac{q_n}{N_n^2}} \, dx.
	\]

	Passing to the limit $n \to \infty$, the right-hand side vanishes in $L^1_{loc}(dt)$ by \eqref{eq:stokes-limit-conservation-defect-vanishing}, while the weak convergence of $\vmean{\zeta g^\flat_n}$ established above implies the left-hand side converges in the sense of distributions.
	Therefore, the time derivative of $\int_{\mathbb{T}^N} \langle \zeta g(t) \rangle \, dx$ in the sense of distributions is zero, implying this integral is constant for almost every $t \geq 0$.
	Substituting the explicit form \eqref{eq:stokes-limit-explicit-infinitesimal-fd} of $g$ and observing that the odd integral $\langle |v|^2 v_i \rangle$ vanishes, we use the definition of the constant $\gamma$ to deduce that
	\[
		\int_{\mathbb{T}^N} \vmean{|v|^2 g(t)} \, dx
		= \vmean{|v|^2} \int_{\mathbb{T}^N} (\rho + \gamma \theta) \, dx
		= \vmean{|v|^2} C_B(t),
	\]
	which implies the function $C_B$ is constant for almost every $t \geq 0$.
	This yields the Boussinesq relation
	\begin{equation}
		\label{eq:stokes-limit-boussinesq-cb}
		\rho + \gamma \theta = C_B.
	\end{equation}
	In Part 5 of this proof, we will establish that the constant $C_B$ is in fact zero.

	\proofpart{3}{The temperature equation}

	We now pass to the limit in the approximate local conservation laws
	\eqref{eq:stokes-limit-approximate-local-conserv} to recover the Stokes-Fourier system.
	Let $K' = \langle |v|^4 \rangle / (2 \langle |v|^2 \rangle)$ and choose $\zeta = |v|^2 / 2 - K'$.
	The vanishing of the conservation defects shown in \eqref{eq:stokes-limit-conservation-defect-vanishing} implies that the right-hand side converges to zero in the sense of distributions.
	Furthermore, from the weak-$*$ convergence $\vmean{\zeta g^\flat_n} \weakstarto \vmean{\zeta g}$ established in the previous step, the time derivative term $\partial_t \vmean{\zeta g^\flat_n}$ converges to $\partial_t \vmean{\zeta g}$ in the sense of distributions as $n \to \infty$.

	Substituting the explicit form \eqref{eq:stokes-limit-explicit-infinitesimal-fd} of the infinitesimal Fermi-Dirac distribution $g$, the odd integrals $\vmean{(|v|^2 / 2 - K') v_i}$ vanish.
	Moreover, the definition of $K'$ yields the identity $\langle (|v|^2 / 2 - K') (|v|^2 - K) \rangle = - K \vmean{|v|^2 / 2 - K'}$.
	Finally, substituting $\rho$ as a function of $\theta$ using the Boussinesq relation \eqref{eq:stokes-limit-boussinesq-cb}, we expand the limit as
	\[
		\vmean{ \left(\frac{|v|^2}{2} - K'\right) g}
		= \vmean{\frac{|v|^2}{2} - K'} C_B + C_1 \theta,
	\]
	where
	\[
		C_1
		= \frac{\vmean{|v|^4}}{4 \vmean{|v|^2}^2}
		\left(
			\vmean{|v|^4} \vmean{1} - \vmean{|v|^2}^2
		\right).
	\]

	This constant $C_1$ is strictly positive.
	Indeed, the Cauchy-Schwarz inequality ensures the term in parentheses is nonnegative and the linear independence of $1$ and $|v|^2$ in $L^2(F(1-F))$ guarantees that the inequality is strict.
	Therefore, taking the time derivatives in the sense of distributions, we conclude that
	\[
		\partial_t \vmean{\zeta g^\flat_n} \to C_1 \partial_t \theta
		\quad\text{in }
		\mathcal{D}'((0,\infty) \times \mathbb{T}^N_x).
	\]

	To establish the limit of the flux term, consider $B(v) = (|v|^2 / 2 - K')v$.
	For each component $i = 1, \dots, N$, we have $B_i \perp \ker \mathcal{L}$.
	Since the linearized collision operator $\mathcal{L}$ is self-adjoint, it follows that $B_i \in \overline{\operatorname{Im}(\mathcal{L})}$.
	Recall from Lemma \ref{lem:linearized-properties} that $\mathcal{L}$ is Fredholm and therefore its image must be closed.
	This means there exists a $\hat{B}_i \in L^2(\tau F(1-F) dv)$ such that $\mathcal{L}(\hat{B}_i) = B_i$.
	Since the domain of $\mathcal{L}$ decomposes orthogonally as $\ker\mathcal{L} \oplus (\ker\mathcal{L})^\perp$, we can uniquely define $\hat{B}_i$ by further imposing the condition that $\hat{B}_i \in (\ker \mathcal{L})^\perp$.
	For convenience, we denote this unique pre-image using the pseudo-inverse as $\hat{B}_i = \mathcal{L}^{-1}[B_i]$.

	With this definition, we rewrite the flux term using the self-adjointness of $\mathcal{L}$, yielding
	\begin{equation*}
		\frac{1}{\kappa_n} \Div_x \vmean{\zeta v g^\flat_n}
		= \frac{1}{\kappa_n} \Div_x \vmean{B g^\flat_n}
		= \Div_x \vmean{\hat{B} \frac{1}{\kappa_n} \mathcal{L}(g^\flat_n)}.
	\end{equation*}
	We intend to pass the term inside the divergence to the limit in the sense of distributions.
	Expanding the definition \eqref{eq:linearized-op-def} of $\mathcal{L}$ in the angle brackets, we obtain
	\begin{equation*}
		\vmean{\hat{B} \frac{1}{\kappa_n} \mathcal{L}(g^\flat_n)}
		= - \vvmean{\hat{B} \frac{1}{\kappa_n}
			\Big(
				(g^\flat_n)' + (g^\flat_n)'_*
				- g^\flat_n - (g^\flat_n)_*
			\Big)}.
	\end{equation*}

	The main ingredient needed to pass this expression to the limit is the strong convergence established in Lemma \ref{lem:stokes-fluxes-difference-vanishing}.
	However, to apply this result, we must first prove that $\hat{B}$ belongs to $L^2(\Lambda)$.

	Applying the estimate \eqref{eq:b-v-v*-estimate} on the collision kernel, using the trivial bound $(1 - F') (1 - F'_*) \leq 1$ and evaluating the constant integral over $\omega$, we obtain
	\begin{equation}
		\label{eq:stokes-limit-bbar-L2-Lambda}
		\| \hat{B} \|_{L^2(\Lambda)}^2
		\leq C_b | \mathbb{S}^{N-1} |
			\left(
				\int \big( \hat{B} \big)^2 (1 + |v|) F \, dv
			\right)
			\left(
				\int (1 + |v_*|) F_* \, dv_*
			\right).
	\end{equation}
	The bounds $1 + |v| \leq 1 + 2 \tau \leq 6 \tau$ and $(1-F)^{-1} \leq 1 + e$ imply that the first integral on the right-hand side is bounded by $6 (1 + e) \| \hat{B} \|_{L^2(\tau F(1 - F) dv)}^2$.
	Since $\hat{B}$ belongs to $L^2(\tau F(1 - F) dv)$ by construction and the remaining integral with respect to $v_*$ is finite, we conclude that $\hat{B} \in L^2(\Lambda)$.

	Applying the Cauchy-Schwarz inequality in $L^2(\Lambda)$, integrating the resulting inequality over $x \in \mathbb{T}^N_x$ and taking the supremum over $t \in (0,\infty)$, we obtain
	\[
		\begin{split}
			&\left\|
				\hat{B}
				\left[
					\frac{q_n}{N_n N_{n,*} N'_n N'_{n,*}}
					- \frac{1}{\kappa_n}
					\left(
						(g^\flat_n)' + (g^\flat_n)'_*
						- g^\flat_n - (g^\flat_n)_*
					\right)
				\right]
			\right\|_{L^\infty(dt;L^1(dx\Lambda))} \\
			&\leq \| \hat{B} \|_{L^2(\Lambda)}
			\left\|
				\frac{q_n}{N_n N_{n,*} N'_n N'_{n,*}}
				- \frac{1}{\kappa_n}
				\left(
					(g^\flat_n)' + (g^\flat_n)'_*
					- g^\flat_n - (g^\flat_n)_*
				\right)
			\right\|_{L^\infty(dt;L^1(dx;L^2(\Lambda)))}.
		\end{split}
	\]

	Since $\eta_n = \sqrt{\kappa_n \sigma_n} = \kappa_n$ under the Stokes scaling, Lemma \ref{lem:stokes-fluxes-difference-vanishing} implies that the second factor on the right-hand side vanishes as $n \to \infty$.
	Therefore, it follows that
	\begin{equation}
		\label{eq:stokes-limit-b-qn-minus-linear-gn-vanishing}
		\vvmean{
			\hat{B}
			\left[
				\frac{q_n}{N_n N_{n,*} N'_n N'_{n,*}}
				- \frac{1}{\kappa_n}
				\left(
					(g^\flat_n)' + (g^\flat_n)'_*
					- g^\flat_n - (g^\flat_n)_*
				\right)
			\right]
		} \to 0
	\end{equation}
	in $L^\infty(dt; L^1(dx))$.

	On the other hand, the weak compactness result of Lemma \ref{lem:collision-product-compactness} implies that, up to extraction of a subsequence, $q_n / N_n \weakto q$ in $L^1(dtdx\Lambda)$.
	Combining the uniform bound in $1 / N_n$ given by Lemma \ref{lem:Nn-Pn-bounds} with the pointwise almost everywhere convergence $N_n \to 1$ (up to a further subsequence), the product limit theorem yields
	\begin{equation}
		\label{eq:stokes-limit-qn-renormalized-weakly-conv}
		\frac{q_n}{N_n N_{n,*} N'_n N'_{n,*}}
		\weakto q
	\end{equation}
	in $L^1_{loc}(dt; L^1(dx\Lambda))$.

	Since Lemma \ref{lem:collision-product-compactness} further establishes that this sequence is uniformly bounded in $L^2(dtdx\Lambda)$, the convergence also holds weakly in this space.
	Testing this weak convergence against the function $\hat{B} \varphi(t,x) \in L^2(dtdx\Lambda)$, we deduce
	\[
		\vvmean{
			\hat{B}
			\frac{q_n}{N_n N_{n,*} N'_n N'_{n,*}}
		}
		\weakto \vvmean{\hat{B} q}
	\]
	weakly in $L^2(dtdx)$.
	Because the spatial domain has finite measure, this implies weak convergence in $L^1_{loc}(dt; L^1(dx))$.
	Combining this limit with \eqref{eq:stokes-limit-b-qn-minus-linear-gn-vanishing}, we find that the flux term $\frac{1}{\kappa_n} \langle \hat{B} \mathcal{L}(g^\flat_n) \rangle$ converges weakly in the same space.
	Taking the divergence yields, in the sense of distributions,
	\[
		\frac{1}{\kappa_n} \Div_x \vmean{\zeta v g^\flat_n}
		= \Div_x \vmean{
			\hat{B} \frac{1}{\kappa_n}
			\mathcal{L}(g^\flat_n)
		}
		\to - \Div_x \vvmean{\hat{B} q}.
	\]

	We now relate the quantity $\vvmean{\hat{B} q}$ to the limit distribution $g$.
	Multiplying equation \eqref{eq:stokes-limit-transport-g} by the component $\hat{B}_i$ and integrating with respect to $v$, we obtain $\Div_x \langle \hat{B}_i v g \rangle = \vvmean{\hat{B}_i q}$.
	Adopting the tensor divergence convention $A = A_{ij}$ as $\Div_x A = \partial_{x_i} A_{ij}$, we rewrite this as
	\[
		\vvmean{\hat{B} q}
		= \Div_x \vmean{(v \otimes \hat{B}) g}.
	\]

	To evaluate the term inside the divergence on the right-hand side, observe that the $\mathcal{L}$ operator preserves parity, that is, $\mathcal{L}[f(- \cdot)](v) = \mathcal{L}[f](-v)$.
	By uniqueness, the pseudo-inverse $\mathcal{L}^{-1}$ defined on $(\ker \mathcal{L})^\perp$ also respects the parity of functions.

	Therefore, since $B_i(v)$ is an odd function in $v$, it follows that $\hat{B}_i(v) = \mathcal{L}^{-1}[B_i](v)$ is also odd with respect to $v$.
	Substituting the explicit form \eqref{eq:stokes-limit-explicit-infinitesimal-fd} of the infinitesimal Fermi-Dirac distribution $g$, we expand the tensor product as
	\begin{equation}
		\label{eq:stokes-limit-v-otimes-hatB-g-expansion}
		\vmean{(v \otimes \hat{B}) g}
		= \left(\rho - \frac{1}{2} K \theta\right) \vmean{v \otimes \hat{B}}
		+ \sum_i u_i \vmean{(v \otimes \hat{B}) v_i}
		+ \frac{1}{2} \theta \vmean{|v|^2 (v \otimes \hat{B})}.
	\end{equation}

	Since $\hat{B}$ is an odd function of $v$, the components $\langle (v \otimes \hat{B}) v_i \rangle_{jk} = \langle v_j \hat{B}_k v_i \rangle$ vanish identically because the integrand is odd.
	Furthermore, by construction, we have $\hat{B}_i \perp \ker \mathcal{L}$.
	Since the velocity components $v_j$ belong to $\ker \mathcal{L}$, we have $\langle v \otimes \hat{B} \rangle = 0$.
	Therefore, in the above expression, only the term involving $\langle |v|^2 (v \otimes \hat{B}) \rangle$ remains.
	Recalling the definition $B(v) = (|v|^2 / 2 - K')v$ and applying the vanishing of $\langle v \otimes \hat{B} \rangle$ once more, we obtain
	\[
		\vmean{(v \otimes \hat{B}) g}
		= \frac{1}{2} \theta \vmean{|v|^2 (v \otimes \hat{B})}
		= \theta \vmean{B \otimes \hat{B}}.
	\]

	The linearized collision operator $\mathcal{L}$ is invariant under orthogonal transformations.
	While this is proved for the classical Boltzmann equation in $\mathbb{R}^3$ in \cite[Appendix 2, Lemma 12]{golse-2012}, the argument extends analogously to the Boltzmann-Fermi-Dirac linearized operator in $\mathbb{R}^N$.
	Indeed, for any $R \in O_N(\mathbb{R})$, the integral substitutions $(w_*, u) = (R v_*, R \omega)$ remain valid in the $N$-dimensional setting with unit Jacobian.
	Furthermore, the Fermi-Dirac analog of the Maxwellian weight $M$, given by the map $(v, v_*) \mapsto F_* (1 - F') (1 - F'_*) / (1 - F)$, is invariant under the action of $O_N(\mathbb{R})$.
	Therefore, for all $R \in O_N(\mathbb{R})$ and $\phi \in \operatorname{Dom} \mathcal{L}$, we have $\mathcal{L}(\phi) \circ R = \mathcal{L}(\phi \circ R)$.

	This property implies that the vector-valued function $\hat{B}: \mathbb{R}^N \to \mathbb{R}^N$ is isotropic, meaning for all $R \in O_N(\mathbb{R})$, we have $\hat{B}(Rv) = R\hat{B} (v)$.
	Indeed, applying the rotational invariance of $\mathcal{L}$ and using that $B(v)$ is isotropic yields
	\[
		\mathcal{L}(\hat{B} \circ R) = \mathcal{L}(\hat{B}) \circ R = R \circ \mathcal{L}(\hat{B}) = \mathcal{L}(R \circ \hat{B}).
	\]
	By linearity, this implies $\hat{B} \circ R - R \circ \hat{B}$ belongs to $\ker \mathcal{L}$.
	On the other hand, by construction, $\hat{B} \perp \ker \mathcal{L}$.
	Because the kernel of $\mathcal{L}$ is invariant under orthogonal transformations, both $\hat{B} \circ R$ and $R \circ \hat{B}$ belong to $(\ker \mathcal{L})^\perp$.
	Since $\hat{B} \circ R - R \circ \hat{B}$ is simultaneously in $\ker \mathcal{L}$ and in $(\ker \mathcal{L})^\perp$, it must vanish identically, implying $\hat{B}$ is isotropic.

	By \cite[Appendix 1, Lemma 8]{golse-2012}, any isotropic vector-valued function must be of the form $\hat{B}(v) = \beta(|v|) v$ for some real-valued scalar function $\beta$.
	Substituting this into the tensor product yields
	\[
		\theta \vmean{B \otimes \hat{B}}
		= \theta \vmean{
			\beta(|v|)
			\left( \frac{|v|^2}{2} - K' \right)
			(v \otimes v)
		}.
	\]
	Applying \cite[Appendix 1, Lemma 9]{golse-2012}, the off-diagonal terms of this integral vanish.
	The above tensor then reduces to a multiple of the identity matrix, $\theta \langle B \otimes \hat{B} \rangle = \theta C_2 I$, where the scalar $C_2$ is given by
	\[
		C_2 =
		\frac{1}{N}
		\vmean{
			\beta(|v|)
			\left( \frac{|v|^2}{2} - K' \right)
			|v|^2
		}.
	\]

	Furthermore, since $\mathcal{L}$ is a nonnegative self-adjoint operator and $\hat{B} \in (\ker \mathcal{L})^\perp$, taking the trace of the tensor yields $N C_2 = \langle \hat{B} \cdot B \rangle = \langle \hat{B} \cdot \mathcal{L}(\hat{B}) \rangle > 0$, ensuring that $C_2$ is strictly positive.
	It follows that the flux term converges as
	\[
		\frac{1}{\kappa_n}
		\vmean{ \zeta v g^\flat_n}
		\to - \Div_x (\theta C_2 I)
		= - C_2 \nabla_x \theta
		\quad\text{in }
		\mathcal{D}'((0,\infty) \times \mathbb{T}^N).
	\]

	Combining the convergence of the time derivative $\partial_t \vmean{\zeta g^\flat_n}$, the vanishing of the conservation defects and the limit of the flux term, we pass to the limit in the approximate local conservation law.
	Defining the thermal conductivity $\kappa = C_2 / C_1$, we obtain
	\[
		\partial_t \theta - \kappa \Delta_x \theta = 0,
	\]

	\proofpart{4}{The momentum equation}

	While the approximate local conservation laws \eqref{eq:stokes-limit-approximate-local-conserv} are defined for scalar collision invariants $\zeta$, we may formally extend them to vector-valued invariants by applying the operations component by component.
	To derive the momentum equation, we choose the collision invariant $\zeta(v) = v$ and \eqref{eq:stokes-limit-approximate-local-conserv} takes the form
	\[
		\partial_t \vmean{v g^\flat_n}
		+ \frac{1}{\kappa_n} \Div_x \vmean{(v \otimes v) g^\flat_n}
		= \frac{1}{\kappa_n} \vvmean{v \frac{q_n}{N_n^2}}.
	\]

	As before, the conservation defects vanish by \eqref{eq:stokes-limit-conservation-defect-vanishing} and the weak-$*$ convergence $\vmean{\zeta g^\flat_n} \weakstarto \vmean{\zeta g}$ established in Step 2 yields that $\partial_t \vmean{v g^\flat_n} \to \partial_t \vmean{v g}$ in the sense of distributions.
	Substituting the explicit form \eqref{eq:stokes-limit-explicit-infinitesimal-fd} for $g$ into $\vmean{\zeta g}$, the odd integrals $\vmean{v}$ and $\vmean{(|v|^2 - K) v}$ vanish.
	Using the isotropy identity $\vmean{v_i v_j} = (K/N) \vmean{1} \delta_{ij}$, we find that $\vmean{v g} = (K/N) \vmean{1} u$.
	Therefore, we conclude that
	\begin{equation}
		\label{eq:stokes-limit-velocity-conv}
		\partial_t \vmean{v g^\flat_n}
		\to \frac{K}{N} \vmean{1} \partial_t u
		\quad\text{in }
		\mathcal{D}'((0,\infty) \times \mathbb{T}^N_x; \mathbb{R}^N).
	\end{equation}

	Next, we pass the flux term to the limit.
	To do this, we separate the traceless part of $v \otimes v$, yielding the decomposition
	\[
		\frac{1}{\kappa_n} \vmean{(v \otimes v) g^\flat_n}
		= \frac{1}{\kappa_n} \vmean{
			\left(
				v \otimes v - \frac{1}{N} |v|^2 I
			\right)
			g^\flat_n
		}
		+ \frac{1}{\kappa_n} \frac{1}{N} \vmean{|v|^2 g^\flat_n} I,
	\]
	where $I$ denotes the $N \times N$ identity matrix.
	Applying the divergence on both sides, the second term on the right-hand side becomes $\nabla_x \vmean{|v|^2 g^\flat_n} / (N \kappa_n)$.
	As a pure gradient, this term disappears when tested against divergence-free vector fields, as required by the weak formulation of the Stokes momentum equation.
	Therefore, it suffices to study the limit of this traceless component.

	To this end, let $A(v) = v \otimes v - \frac{1}{N} |v|^2 I$.
	Since $\mathcal{L}$ is a Fredholm operator and $A_{ij} \perp \ker \mathcal{L}$ for all $i, j = 1, \dots, N$, there exists a unique $\hat{A}_{ij} \in L^2(\tau F(1-F) dv)$ satisfying both $\hat{A}_{ij} \in (\ker \mathcal{L})^\perp$ and $\mathcal{L}(\hat{A}_{ij}) = A_{ij}$.
	For convenience, we denote this unique pre-image using the pseudo-inverse, writing $\hat{A}_{ij} = \mathcal{L}^{-1}[A_{ij}]$.
	Substituting $A = \mathcal{L}(\hat{A})$ and using the self-adjointness of $\mathcal{L}$, we obtain
	\[
		\frac{1}{\kappa_n} \vmean{A g^\flat_n}
		= \frac{1}{\kappa_n} \vmean{\mathcal{L}(\hat{A}) g^\flat_n}
		= - \vvmean{
			\hat{A}
			\frac{1}{\kappa_n}
			\left(
				(g^\flat_n)' + (g^\flat_n)'_*
				- g^\flat_n - (g^\flat_n)_*
			\right)
		}.
	\]

	Observe that the estimates applied to $\hat{B}$ in the previous step also hold for $\hat{A}$.
	Therefore, the bound \eqref{eq:stokes-limit-bbar-L2-Lambda} remains valid with $\hat{A}$ in place of $\hat{B}$ and an identical argument allows us to deduce that $\hat{A} \in L^2(\Lambda)$.
	Together with Lemma \ref{lem:stokes-fluxes-difference-vanishing}, this implies that
	\[
		\vvmean{
			\hat{A}
			\left[
				\frac{q_n}{N_n N_{n,*} N'_n N'_{n,*}}
				- \frac{1}{\kappa_n}
				\left(
					(g^\flat_n)' + (g^\flat_n)'_*
					- g^\flat_n - (g^\flat_n)_*
				\right)
			\right]
		}
		\to 0
	\]
	in $L^\infty(dt; L^1(dx))$.
	Recalling the weak convergence \eqref{eq:stokes-limit-qn-renormalized-weakly-conv} established in Step 3 (which extends to weak convergence in $L^2(dtdx\Lambda)$ by uniform boundedness), we have
	\[
		\vvmean{
			\hat{A}
			\frac{q_n}{N_n N_{n,*} N'_n N'_{n,*}}
		}
		\weakto
		\vvmean{\hat{A} q}
	\]
	weakly in $L^1_{loc}(dt; L^1(dx))$.
	Therefore, we deduce that $\frac{1}{\kappa_n} \vmean{A g^\flat_n}$ converges weakly to $- \vvmean{\hat{A} q}$ in $L^1_{loc}(dt; L^1(dx))$.

	To explicitly evaluate this limit term, we multiply the limiting equation \eqref{eq:stokes-limit-transport-g} by $\hat{A}$ and integrate with respect to $v$.
	This yields the identity $\Div_x \langle (v \otimes \hat{A}) g \rangle = \vvmean{\hat{A} q}$.
	Substituting this into our weak limit, we conclude that
	\begin{equation}
		\label{eq:stokes-limit-flux-velocity-convergence}
		\frac{1}{\kappa_n} \vmean{A g^\flat_n}
		\weakto
		- \Div_x \vmean{(v \otimes \hat{A}) g}
	\end{equation}
	weakly in $L^1_{loc}(dt; L^1(dx))$.

	Because $A(v)$ is even and the pseudo-inverse $\mathcal{L}^{-1}$ preserves parity, $\hat{A}(v)$ must also be even.
	Substituting the explicit form \eqref{eq:stokes-limit-explicit-infinitesimal-fd} of $g$ yields the exact same expansion as \eqref{eq:stokes-limit-v-otimes-hatB-g-expansion}, but with $\hat{A}$ in place of $\hat{B}$.
	Since $\hat{A}$ is even, the odd integrals $\langle v \otimes \hat{A} \rangle$ and $\langle |v|^2 (v \otimes \hat{A}) \rangle$ vanish, yielding
	\[
		\vmean{(v \otimes \hat{A}) g}
		= \sum_i u_i \vmean{(v \otimes \hat{A}) v_i}
		= u \cdot \vmean{v \otimes v \otimes \hat{A}},
	\]
	where the dot product denotes contraction over the first index.

	By the definition of the pseudo-inverse, $\hat{A}_{ij} \perp \ker \mathcal{L}$, which implies $\langle |v|^2 \hat{A} \rangle = 0$.
	Therefore, substituting $v \otimes v = A + \frac{1}{N} |v|^2 I$, the trace component disappears and we obtain $\langle v \otimes v \otimes \hat{A} \rangle = \langle A \otimes \hat{A} \rangle$.

	Furthermore, since $\mathcal{L}$ is invariant by orthogonal transformations, the same reasoning applied previously to $\hat{B}$ yields that the isotropy of $A$ implies the isotropy of $\hat{A}$.
	By \cite[Appendix 1, Lemma 8]{golse-2012}, this guarantees the existence of real-valued scalar functions $\lambda_1$ and $\lambda_2$ such that
	\[
		\hat{A}(v) = \lambda_1(|v|) I + \lambda_2(|v|) v \otimes v.
	\]

	Since $A$ is traceless and $\mathcal{L}(\hat{A}) = A$, we have $\tr(\mathcal{L}(\hat{A})) = 0$.
	Because $\mathcal{L}$ is a linear operator acting component-wise, it commutes with the trace, implying $\mathcal{L}(\tr(\hat{A})) = 0$.
	Therefore, $\tr(\hat{A}) \in \ker\mathcal{L}$.
	However, by the definition of the pseudo-inverse, each component $\hat{A}_{ij}$ is orthogonal to $\ker\mathcal{L}$, which implies $\tr(\hat{A}) \perp \ker\mathcal{L}$.
	Since $\tr(\hat{A})$ is both in the kernel and orthogonal to it, we conclude that $\tr(\hat{A}) = 0$.

	Taking the trace of the expression for $\hat{A}$ yields $\lambda_1(|v|) N + \lambda_2(|v|) |v|^2 = 0$.
	Solving for $\lambda_1(|v|)$ and substituting back into the formula, we find
	\[
		\hat{A}(v)
		= \lambda_2(|v|)
		\left(
			v \otimes v - \frac{1}{N} |v|^2 I
		\right)
		= \lambda_2(|v|) A.
	\]

	To compute the flux, we must evaluate the components of the tensor $A \otimes \hat{A}$.
	Substituting the above expression and expanding the definition of $A$ yields, for every $i, j, k, l = 1, \dots, N$,
	\begin{equation*}
		\begin{split}
			\vmean{A \otimes \hat{A}}_{ijkl}
			&= \vmean{\lambda_2(|v|) v_i v_j v_k v_l}
			- \frac{1}{N} \vmean{\lambda_2(|v|) |v|^2 v_i v_j} \delta_{kl} \\
			&\quad - \frac{1}{N}\vmean{\lambda_2(|v|) |v|^2 v_k v_l} \delta_{ij}
			+ \frac{1}{N^2} \vmean{\lambda_2(|v|) |v|^4} \delta_{ij} \delta_{kl}.
		\end{split}
	\end{equation*}

	Using \cite[Appendix 1, Lemma 9]{golse-2012}, we evaluate the integral in the second term as $\vmean{\lambda_2(|v|) |v|^2 v_i v_j} = \frac{1}{N} \vmean{\lambda_2(|v|) |v|^4} \delta_{ij}$, with an analogous expression holding for the third term.
	Similarly, we evaluate the fourth-order moments in the first term using \cite[Appendix 1, Lemma 10]{golse-2012}, which yields $\vmean{\lambda_2(|v|) v_i v_j v_k v_l} = \frac{1}{N(N+2)} \vmean{\lambda_2(|v|) |v|^4} (\delta_{ij} \delta_{kl} + \delta_{ik} \delta_{jl} + \delta_{il} \delta_{jk})$.
	Substituting these identities back into the expansion, we obtain
	\begin{equation*}
		\vmean{A \otimes \hat{A}}_{ijkl}
		= \frac{\vmean{\lambda_2(|v|) |v|^4}}{N(N+2)}
		\left[
			\delta_{ik} \delta_{jl}
			+ \delta_{il} \delta_{jk}
			- \frac{2}{N} \delta_{ij} \delta_{kl}
		\right].
	\end{equation*}

	Contracting this tensor with the velocity $u$ and taking the divergence with respect to $x$ yields
	\[
		\left[
			\Div_x \left( u \cdot \vmean{A \otimes \hat{A}} \right)
		\right]_{kl}
		= \frac{\vmean{\lambda_2(|v|) |v|^4}}{N(N+2)}
		\left[
			\partial_{x_l} u_k + \partial_{x_k} u_l
			- \frac{2}{N} (\Div_x u) \delta_{kl}
		\right].
	\]
	From the incompressibility relation $\Div_x u = 0$, the last term vanishes.
	Taking the second divergence with respect to $x$, we commute the derivatives in the first term to find $\sum_k \partial_{x_l} \partial_{x_k} u_k = \sum_k \partial_{x_l} (\Div_x u) = 0$.
	The term in square brackets therefore simplifies to $\Delta_x u_l$.

	Recalling that $\langle (v \otimes \hat{A}) g \rangle = u \cdot \langle A \otimes \hat{A} \rangle$, we apply the divergence to both sides of the weak limit \eqref{eq:stokes-limit-flux-velocity-convergence} in the sense of distributions, yielding
	\[
		\Div_x \left[
			\frac{1}{\kappa_n} \vmean{A g^\flat_n}
		\right]
		\to - \frac{\vmean{\lambda_2(|v|) |v|^4}}{N(N+2)} \Delta_x u
	\]
	in $\mathcal{D}'((0,\infty) \times \mathbb{T}^N_x; \mathbb{R}^N)$.
	Together with the limit \eqref{eq:stokes-limit-velocity-conv}, we obtain the velocity equation of the Stokes-Fourier system.

	\proofpart{5}{Identification of the initial data}

	Suppose the moments of the initial data $(g^{in}_n)_n$ converge in the sense of distributions to $(u_0, \theta_0)$ as in the statement of the theorem.
	Let $\zeta = \zeta(v)$ denote either the scalar collision invariant $(C_1)^{-1} ( |v|^2 / 2 - K')$ or the vector collision invariant $N v / (K \vmean{1})$.
	We test the approximate local conservation law \eqref{eq:stokes-limit-approximate-local-conserv} against a spatial test function $\varphi = \varphi(x) \in C^\infty(\mathbb{T}^N)$, taking values in $\mathbb{R}$ or $\mathbb{R}^N$, respectively.
	In the vector case, we further assume that $\varphi$ is divergence-free, $\Div_x \varphi = 0$.

	To unify the notation, let us denote by $\cdot$ the standard inner product in the respective space.
	That is, it represents the vector dot product in the scalar case (temperature equation) and the full tensor contraction in the vector case (momentum equation).
	Defining $\Phi_n(t) = \int_{\mathbb{T}^N} \vmean{\zeta g^\flat_n} \varphi \, dx$, its time derivative in the sense of distributions is given by $\Phi'_n(t) = a_n(t) + b_n(t)$, where
	\[
		\begin{split}
			a_n(t) &= \int_{\mathbb{T}^N} \frac{1}{\kappa_n} \vmean{J g^\flat_n} \cdot \nabla_x \varphi(x) \, dx, \\
			b_n(t) &= \int_{\mathbb{T}^N} \frac{1}{\kappa_n} \vvmean{\zeta \frac{q_n}{N_n^2}} \varphi(x) \, dx.
		\end{split}
	\]
	The function $J = J(v)$ corresponds to the effective macroscopic flux associated with the invariant $\zeta$.
	For the temperature equation, where $\zeta = |v|^2 / 2 - K'$, this flux is $J(v) = \zeta(v) v = B(v)$.
	For the momentum equation, since $\varphi$ is divergence-free, the trace part of the standard momentum flux $v \otimes v$ vanishes when contracted with $\nabla_x \varphi$.
	Thus, we define $J(v)$ as its traceless part, $J(v) = A(v)$.

	Fix $T > 0$.
	From the weak-$*$ convergence $\vmean{\zeta g^\flat_n} \weakstarto \vmean{\zeta g}$ in $L^\infty(dt; L^2(dx))$, we deduce that the sequence $(\Phi_n)_n$ is uniformly bounded in $L^\infty([0,T])$.

	Furthermore, for each $n$, the function $a_n(t)$ is bounded in $L^1_{loc}([0,T])$, and this bound is uniform in $n$.
	To see this, consider the two cases for $\zeta$.
	For the temperature equation, where $\zeta = |v|^2 / 2 - K'$, we established in part 3 that the sequence $\frac{1}{\kappa_n} \langle \hat{B} \mathcal{L}(g^\flat_n) \rangle = \frac{1}{\kappa_n} \langle B g^\flat_n \rangle$ converges weakly in $L^1_{loc}(dt; L^1(dx))$.
	For the momentum equation, where $\zeta = v$, the $a_n(t)$ reduces to the integral of the traceless part  $\frac{1}{\kappa_n} \langle \hat{A} \mathcal{L}(g^\flat_n) \rangle = \frac{1}{\kappa_n} \langle A g^\flat_n \rangle$.
	As shown in \eqref{eq:stokes-limit-flux-velocity-convergence}, this also converges weakly in $L^1_{loc}(dt; L^1(dx))$.

	On the other hand, the scaling \eqref{eq:stokes-limit-conservation-defect-vanishing} of the conservation defects implies that the integrand of $b_n(t)$ vanishes in $L^1_{loc}(dtdx)$.
	Therefore, the sequence $(b_n)_n$ vanishes in $L^1([0,T])$.
	In particular, the expression of $\Phi'_n$ ensures that the sequence $(\Phi_n)_n$ is uniformly bounded in $W^{1,1}([0,T])$.

	This implies that, for each $n$, the function $\Phi_n$ is absolutely continuous.
	By the fundamental theorem of calculus we have, for every $t \in [0,T]$,
	\begin{equation}
		\label{eq:stokes-limit-Phi-n-fund-theorem-calculus}
		\Phi_n(t)
		= \Phi_n(0)
		+ \int_0^t a_n(\tau) \, d\tau
		+ \int_0^t b_n(\tau) \, d\tau.
	\end{equation}
	Let $\mathcal{A}_n(t)$ and $\mathcal{B}_n(t)$ denote the first and second integral terms on the right-hand side, respectively.

	As in the acoustic limit, the vanishing of the conservation defects in $L^1_{loc}(dtdx)$ implies that the sequence $(\mathcal{B}_n)_n$ converges uniformly to zero on $[0,T]$.
	For the sequence $(\mathcal{A}_n)_n$, we apply the Arzelà-Ascoli theorem, but in a slightly different manner than for the acoustic limit.

	The sequence $(\mathcal{A}_n)_n$ is uniformly bounded in $L^\infty([0,T])$.
	Indeed, applying an $L^\infty$ bound on $\nabla_x \varphi$ we have, for every $t \in [0,T]$,
	\[
		|\mathcal{A}_n(t)| \leq \| \nabla_x \varphi \|_{L^\infty(dx)}
			\left\|
				\frac{1}{\kappa_n} \vmean{J g^\flat_n}
			\right\|_{L^1([0,T] \times \mathbb{T}^N)}.
	\]
	The second factor is uniformly bounded with respect to $n$, since the sequence $(\frac{1}{\kappa_n} \vmean{J g^\flat_n})_n$ converges weakly in $L^1_{loc}(dt; L^1(dx))$.

	Furthermore, we note that, by the Dunford-Pettis theorem, the sequence $(\frac{1}{\kappa_n} \vmean{J g^\flat_n})_n$ is uniformly integrable.
	Therefore, for every $\varepsilon > 0$, there exists a $\delta > 0$ such that if $t_1, t_2 \in [0,T]$ satisfy $|t_2 - t_1| = |[t_1, t_2] \times \mathbb{T}^N| < \delta$, then
	\[
		\int_{t_1}^{t_2} \int_{\mathbb{T}^N}
			\left|
				\frac{1}{\kappa_n}
				\vmean{J g^\flat_n}
			\right| \, dx dt
		< \varepsilon
	\]
	for all $n \in \mathbb{N}$.
	This implies $|\mathcal{A}_n(t_2) - \mathcal{A}_n(t_1)| \leq \varepsilon \| \nabla_x \varphi \|_{L^\infty(dx)}$
	for all $n \in \mathbb{N}$, which shows that the sequence $(\mathcal{A}_n)_n$ is uniformly equicontinuous.

	By the Arzela-Ascoli theorem, we can extract a subsequence such that $(\mathcal{A}_n)_n$ converges uniformly on $[0,T]$.
	Moreover, by passing to a further subsequence, we can ensure that the sequence of real numbers $(\Phi_n(0))_n$ also converges.
	Therefore, returning to \eqref{eq:stokes-limit-Phi-n-fund-theorem-calculus}, it follows that the full sequence $(\Phi_n)_n$ converges uniformly on $[0,T]$.

	Since $\vmean{\zeta g^\flat_n} \weakstarto \vmean{\zeta g}$ in $L^\infty(dt; L^2(dx))$, we can identify this uniform limit.
	Up to redefining $g$ on a set of measure zero, we have
	\begin{equation}
		\label{eq:stokes-limit-uniform-time}
		\Phi_n(t)
		= \int_{\mathbb{T}^N} \vmean{\zeta g^\flat_n}(t) \varphi \, dx
		\to \int_{\mathbb{T}^N} \vmean{\zeta g}(t) \varphi \, dx
	\end{equation}
	uniformly on $[0,T]$.

	Observe that this limit is a continuous function in time, as it is a uniform limit of continuous functions.
	Substituting the explicit form \eqref{eq:stokes-limit-explicit-infinitesimal-fd} of the infinitesimal Fermi-Dirac distribution $g$ and the Boussinesq relation \eqref{eq:stokes-limit-boussinesq-cb} for each respective choice of $\zeta$, we deduce that the macroscopic parameters $u$ and $\theta$ belong to $C([0,T]; \mathcal{D}'(\mathbb{T}^N))$.
	In particular, their initial traces at $t = 0$ are well-defined.

	We first apply this uniform convergence to deduce that the $C_B$ constant of the Boussinesq relation \eqref{eq:stokes-limit-boussinesq-cb} vanishes.
	Choose the constant test function $\varphi = 1$.
	By hypothesis \eqref{eq:fn-invariants-concentrated}, for any collision invariant $\zeta \in \operatorname{span}\{1, v_i, |v|^2\}$, we have
	\begin{equation}
		\label{eq:stokes-limit-zeta-initial-data-vanishing}
		\int_{\mathbb{T}^N} \vmean{\zeta g^{in}_n (x)} \, dx = 0.
	\end{equation}

	Since the initial trace of the fluctuations $g_n(0)$ coincides with the initial data $g^{in}_n$, the triangle inequality yields
	\begin{equation*}
		\begin{split}
			\left|
				\int_{\mathbb{T}^N} \vmean{\zeta g}(0) \, dx
			\right|
			&= \left|
				\int_{\mathbb{T}^N} \vmean{\zeta g}(0) \, dx
				- \int_{\mathbb{T}^N} \vmean{\zeta g^{in}_n} \, dx
			\right| \\
			&\leq \left|
				\int_{\mathbb{T}^N} \vmean{\zeta g}(0) \, dx
				- \int_{\mathbb{T}^N} \vmean{\zeta g^\flat_n}(0) \, dx
			\right| \\
			&\quad + \left|
				\int_{\mathbb{T}^N} \vmean{\zeta g^\flat_n}(0) \, dx
				- \int_{\mathbb{T}^N} \vmean{\zeta g_n}(0) \, dx
			\right|.
		\end{split}
	\end{equation*}
	The first term on the right-hand side vanishes as $n \to \infty$ by the uniform limit \eqref{eq:stokes-limit-uniform-time} applied at $t = 0$.

	For the second term, using the fact that $|\zeta| \leq C_\zeta \tau$ for some constant $C_\zeta > 0$, we deduce from inequality \eqref{eq:stokes-limit-gn-gnflat-diff} that $|\zeta g_n - \zeta g^\flat_n| \leq C_\zeta \mu_n \tau g_n^2 / N_n$.
	Integrating this inequality over $\Rxv$ against the measure $F(1-F) dvdx$ and applying the estimate for $(\tau g_n^2 / N_n)_n$ established in Lemma \ref{lem:gn2-Nn-bounds}, we guarantee the existence of a constant $C_0 > 0$ such that
	\begin{equation}
		\label{eq:stokes-limit-zeta-gn-gflat-bound}
		\begin{split}
			\left|
				\int_{\mathbb{T}^N} \vmean{\zeta g_n}(t) \, dx
				- \int_{\mathbb{T}^N} \vmean{\zeta g^\flat_n}(t) \, dx
			\right|
			&\leq \int_{\mathbb{T}^N} \vmean{|\zeta g_n - \zeta g^\flat_n|}(t) \, dx \\
			&\leq C_0 \mu_n \log\left(\frac{1}{\mu_n}\right).
		\end{split}
	\end{equation}
	Since the functions $t \mapsto \int_{\mathbb{T}^N} \vmean{\zeta g_n}(t) \, dx$ and $t \mapsto \int_{\mathbb{T}^N} \vmean{\zeta g^\flat_n}(t) \, dx$ are continuous in time, the above inequality holds everywhere, including at $t = 0$.
	Therefore, the second term of the triangle inequality also vanishes as $n \to \infty$ and by passing to the limit, we obtain $\int_{\mathbb{T}^N} \vmean{\zeta g}(0) \, dx = 0$.

	Evaluating this identity for $\zeta = |v|^2 / \vmean{|v|^2}$ using the infinitesimal Fermi-Dirac form \eqref{eq:stokes-limit-explicit-infinitesimal-fd} of $g$ yields $\int_{\mathbb{T}^N} (\rho(0,x) + \gamma \theta(0,x)) \, dx = 0$.
	Substituting the Boussinesq relation \eqref{eq:stokes-limit-boussinesq-cb}, we obtain
	\[
		C_B = 0.
	\]

	Having established the strong Boussinesq relation $\rho + \gamma \theta = 0$, we now identify the initial data of the Stokes-Fourier system $(u(0), \theta(0))$ with the limits of the macroscopic moments $(u_0, \theta_0)$ given in the statement.
	From the hypothesis of the theorem, the macroscopic moments of the initial fluctuations converge in the sense of distributions to a limit denoted by $\xi_0$.
	Specifically, we set $\xi_0 = u_0$ for the momentum equation (where $\zeta = N v / (K \vmean{1})$) and $\xi_0 = \theta_0$ for the temperature equation (where $\zeta = (C_1)^{-1} ( |v|^2 / 2 - K')$).

	We now show that the initial trace of the limit matches the limit of the initial data, that is, $\int_{\mathbb{T}^N} \vmean{\zeta g}(0) \varphi \, dx = \int_{\mathbb{T}^N} \xi_0 \varphi \, dx$.
	Returning to a generic test function $\varphi$ (divergence-free in the vector case), since the initial trace of the fluctuations $g_n(0)$ coincides with the initial data $g^{in}_n$, the triangle inequality yields
	\begin{equation*}
		\begin{split}
			\left|
				\int_{\mathbb{T}^N} \vmean{\zeta g}(0) \varphi \, dx
				- \int_{\mathbb{T}^N} \xi_0 \varphi \, dx
			\right|
			&\leq \left|
				\int_{\mathbb{T}^N} \vmean{\zeta g}(0) \varphi \, dx
				- \int_{\mathbb{T}^N} \vmean{\zeta g^\flat_n}(0) \varphi \, dx
			\right| \\
			&\quad + \left|
				\int_{\mathbb{T}^N} \vmean{\zeta g^\flat_n}(0) \varphi \, dx
				- \int_{\mathbb{T}^N} \vmean{\zeta g_n}(0) \varphi \, dx
			\right| \\
			&\quad + \left|
				\int_{\mathbb{T}^N} \vmean{\zeta g^{in}_n} \varphi \, dx
				- \int_{\mathbb{T}^N} \xi_0 \varphi \, dx
			\right|.
		\end{split}
	\end{equation*}
	The first term on the right-hand side vanishes as $n \to \infty$ due to the uniform convergence \eqref{eq:stokes-limit-uniform-time} evaluated at $t = 0$.

	Incorporating the $L^\infty$ bound of $\varphi$ into the derivation of \eqref{eq:stokes-limit-zeta-gn-gflat-bound}, we obtain
	\begin{equation*}
		\left|
			\int_{\mathbb{T}^N} \vmean{\zeta g_n}(t) \varphi \, dx
			- \int_{\mathbb{T}^N} \vmean{\zeta g^\flat_n}(t) \varphi \, dx
		\right|
		\leq C_0 \|\varphi\|_{L^\infty} \mu_n \log\left(\frac{1}{\mu_n}\right).
	\end{equation*}
	By the continuity of the functions $t \mapsto \int_{\mathbb{T}^N} \vmean{\zeta g_n}(t) \varphi \, dx$ and $t \mapsto \int_{\mathbb{T}^N} \vmean{\zeta g^\flat_n}(t) \varphi \, dx$, this bound holds for all $t \in [0,T]$.
	Evaluating at $t = 0$, we find that the second term of the triangle inequality is bounded by $C_0 \| \varphi \|_{L^\infty} \mu_n \log(1 / \mu_n)$, which vanishes as $n \to \infty$.

	To establish that the third term vanishes, we consider the scalar and vector cases separately.
	For the temperature equation, the conclusion follows directly from the convergence in the sense of distributions assumed the theorem.
	For the momentum equation, since $\varphi$ is divergence-free, its inner product with the vector field is identical to the inner product with its divergence-free projection.
	That is,
	\[
		\int_{\mathbb{T}^N}
		\vmean{
			\frac{N}{K \vmean{1}} v g^{in}_n
		} \cdot \varphi \, dx
		= \int_{\mathbb{T}^N}
		\Pi \vmean{
			\frac{N}{K \vmean{1}} v g^{in}_n
		} \cdot \varphi \, dx.
	\]
	By hypothesis, this second integral converges to $\int_{\mathbb{T}^N} \xi_0 \varphi \, dx$, which implies that the third term also vanishes.

	Since all three terms converge to zero, we conclude that the continuous trace at $t = 0$ satisfies
	\[
		\int_{\mathbb{T}^N} \vmean{\zeta g}(0) \varphi \, dx
		= \int_{\mathbb{T}^N} \xi_0 \varphi \, dx.
	\]
	By substituting the explicit form \eqref{eq:stokes-limit-explicit-infinitesimal-fd} for $g$ into the left-hand side and incorporating the Boussinesq relation $\rho + \gamma \theta = 0$, we see that the specific choices of the collision invariant $\zeta$ isolate the macroscopic parameters $u(0,x)$ and $\theta(0,x)$ by orthogonality.
	Since this identity holds for any generic test function $\varphi$, we deduce that the initial traces equal $u_0$ and $\theta_0$ in the sense of distributions, thus identifying the initial data for the Stokes-Fourier system.

	Since the Cauchy problem for the Stokes-Fourier system admits a unique weak solution $(u, \theta)$, the limiting distribution $g$ is uniquely determined.
	Following the same compactness and subsequence argument detailed in Part 4 of the proof of Theorem \ref{thm:acoustic-limit}, this uniqueness allows us to upgrade the weak convergence of $g_n$, established in Part 1, from a subsequence to the entire sequence.

	Finally, the zero spatial mean of the macroscopic parameters $u$ and $\theta$ follows from the same argument used in Part 5 of the proof of Theorem \ref{thm:acoustic-limit}, with a necessary adjustment for the initial velocity projection.
	For the temperature fluctuation $\theta_0$, we recall the zero spatial mean of the initial data moments \eqref{eq:stokes-limit-zeta-initial-data-vanishing}.
	Choosing $\zeta = (C_1)^{-1} ( |v|^2 / 2 - K')$ and passing to the limit in the sense of distributions yields $\int_{\mathbb{T}^N} \theta_0 \, dx = 0$.

	For the velocity fluctuation $u_0$, we must account for the Leray projector $\Pi$.
	Applying the Helmholtz-Hodge decomposition on the torus, we express the vector field as a sum of its divergence-free projection and a gradient term, given by
	\[
		\vmean{\frac{N}{K \vmean{1}} v g^{in}_n}
		= \Pi \vmean{\frac{N}{K \vmean{1}} v g^{in}_n}
		+ \nabla_x G_n,
	\]
	for some periodic scalar potential $G_n$.
	Because the integral of a gradient over the torus $\mathbb{T}^N$ vanishes, integrating both sides and applying \eqref{eq:stokes-limit-zeta-initial-data-vanishing} yields
	\[
		\int_{\mathbb{T}^N} \Pi \vmean{\frac{N}{K \vmean{1}} v g^{in}_n (x)} \, dx
		= \int_{\mathbb{T}^N} \vmean{\frac{N}{K \vmean{1}} v g^{in}_n (x)} \, dx
		= 0.
	\]
	Passing to the limit $n \to \infty$ in the sense of distributions, we obtain $\int_{\mathbb{T}^N} u_0 \, dx = 0$.

	Integrating the Stokes-Fourier system over the torus $\mathbb{T}^N$ eliminates all divergence and gradient terms, ensuring the spatial means of $u$ and $\theta$ are conserved for all $t > 0$.
	Since we established that their initial means vanish at $t = 0$, they identically vanish for all $t \geq 0$.
\end{proof}

\bibliographystyle{plain}
\bibliography{references}

\end{document}